\documentclass[12pt]{article}
\usepackage{latexsym,amsmath,amssymb,amsfonts,amsthm,bbm,tensor}
\usepackage{natbib}
\usepackage{enumerate}
\RequirePackage[colorlinks,citecolor=blue,urlcolor=blue,breaklinks]{hyperref}
\usepackage{graphicx}
\usepackage{graphics,enumerate}
\usepackage[x11names,svgnames]{xcolor}

\newtheorem{thm}{Theorem}
\newtheorem{prop}[thm]{Proposition}
\newtheorem{lemma}[thm]{Lemma}
\newtheorem{corollary}[thm]{Corollary}

\DeclareMathOperator*{\argmin}{argmin}
\DeclareMathOperator*{\argmax}{argmax}

\begin{document}

\title{Global rates of convergence in log-concave density estimation}
\author{Arlene K. H. Kim and Richard J. Samworth\\
Statistical Laboratory, University of Cambridge}

\date{(\today)}

\maketitle
\begin{abstract}
The estimation of a log-concave density on $\mathbb{R}^d$ represents a central problem in the area of nonparametric inference under shape constraints.  In this paper, we study the performance of log-concave density estimators with respect to global loss functions, and adopt a minimax approach.  We first show that no statistical procedure based on a sample of size $n$ can estimate a log-concave density with respect to the squared Hellinger loss function with supremum risk smaller than order $n^{-4/5}$, when $d=1$, and order $n^{-2/(d+1)}$ when $d \geq 2$.  In particular, this reveals a sense in which, when $d \geq 3$, log-concave density estimation is fundamentally more challenging than the estimation of a density with two bounded derivatives (a problem to which it has been compared).  Second, we show that for $d \leq 3$, the Hellinger $\epsilon$-bracketing entropy of a class of log-concave densities with small mean and covariance matrix close to the identity grows like $\max\{\epsilon^{-d/2},\epsilon^{-(d-1)}\}$ (up to a logarithmic factor when $d=2$).  This enables us to prove that when $d \leq 3$ the log-concave maximum likelihood estimator achieves the minimax optimal rate (up to logarithmic factors when $d = 2,3$) with respect to squared Hellinger loss.
\end{abstract}

\section{Introduction}
\label{Sec:Intro}

Log-concave densities on $\mathbb{R}^d$, namely those expressible as the exponential of a concave function that takes values in $[-\infty,\infty)$, form a particularly attractive infinite-dimensional class.  Gaussian densities are of course log-concave, as are many other well-known parametric families, such as uniform densities on convex sets, Laplace densities and many others.  Moreover, the class retains several of the properties of normal densities that make them so widely-used for statistical inference, such as closure under marginalisation, conditioning and convolution operations.  On the other hand, the set is small enough to allow fully automatic estimation procedures, e.g.\ using maximum likelihood, where more traditional nonparametric methods would require troublesome choices of smoothing parameters.  Log-concavity therefore offers statisticians the potential of freedom from restrictive parametric (typically Gaussian) assumptions without paying a hefty price.  Indeed, in recent years, researchers have sought to exploit these alluring features to propose new methodology for a wide range of statistical problems, including the detection of the presence of mixing \citep{Walther2002}, tail index estimation \citep{MullerRufibach2009}, clustering \citep{CSS2010}, regression \citep{DSS2011}, Independent Component Analysis \citep{SamworthYuan2012} and classification \citep{ChenSamworth2013}.  

However, statistical procedures based on log-concavity, in common with other methods based on shape constraints, present substantial computational and theoretical challenges and these have therefore also been the focus of much recent research.  For instance, the maximum likelihood estimator of a log-concave density, first studied by \citet{Walther2002} in the case $d=1$, and by \citet{CSS2010} for general $d$, plays a central role in all of the procedures mentioned in the previous paragraph.  \citet{DHR2011} developed a fast, Active Set algorithm for computing the estimator when $d=1$, and this is implemented in the \textbf{R} package \texttt{logcondens} \citep{RufibachDumbgen2006,DumbgenRufibach2011}.  For general $d$, a slower, non-smooth optimisation method based on Shor's $r$-algorithm is implemented in the \textbf{R} package \texttt{LogConcDEAD} \citep{CGSC2007,CGS2009}; see also \citet{KoenkerMizera2010} for an alternative approximation approach based on interior point methods.  On the theoretical side, through a series of papers \citep{PWM2007,DumbgenRufibach2009,SereginWellner2010,SchuhmacherDumbgen2010,CuleSamworth2010,DSS2011}, we now have a fairly complete understanding of the global consistency properties of the log-concave maximum likelihood estimator (even under model misspecification).  

Results on the global rate of convergence in log-concave density estimation are, however, less fully developed, and in particular have been confined to the case $d=1$.  For a fixed true log-concave density $f_0$ belonging to a H\"older ball of smoothness $\beta \in [1,2]$, \citet{DumbgenRufibach2009} studied the supremum distance over compact intervals in the interior of the support of $f_0$.  They proved that the log-concave maximum likelihood estimator $\hat{f}_n$ based on a sample of size $n$ converges in these metrics to $f_0$ at rate $O_p(\rho_n^{-\beta/(2\beta+1)})$, where $\rho_n := n/\log n$; thus $\hat{f}_n$ attains the same rates in the stated regimes as other adaptive nonparametric estimators that do not satisfy the shape constraint.  Very recently, \citet{DossWellner2015} introduced a new bracketing argument to obtain a rate of convergence of $O_p(n^{-4/5})$ in squared Hellinger distance in the case $d=1$, again for a fixed true log-concave density $f_0$.

In this paper, we present several new results on global rates of convergence in log-concave density estimation, with a focus on a minimax approach.  We begin by proving, in Theorem~\ref{Thm:LowerBound} in Section~\ref{Sec:LowerBounds}, a non-asymptotic minimax lower bound which shows that for the squared Hellinger loss function defined in~\eqref{Eq:Hellinger} below, no statistical procedure based on a sample of size $n$ can estimate a log-concave density with supremum risk smaller than order $n^{-4/5}$ when $d = 1$, and order $n^{-2/(d+1)}$ when $d \geq 2$.  The surprising feature of this result is that it is often thought that estimation of log-concave densities should be similar to the estimation of densities with two bounded derivatives, for which the minimax rate is known to be $n^{-4/(d+4)}$ for all $d \in \mathbb{N}$ \citep{IbragimovKhasminskii1983}.  The reasoning for this intuition appears to be Aleksandrov's theorem \citep{Aleksandrov1939}, which states that a convex function on $\mathbb{R}^d$ is twice differentiable (Lebesgue) almost everywhere in its domain, and the fact that for twice continuously differentiable functions, convexity is equivalent to a second derivative condition, namely that the Hessian matrix is non-negative definite.  Thus, our minimax lower bound reveals that while this intuition is valid when $d \leq 2$ (note that $4/(d+4) = 2/(d+1) = 2/3$ when $d=2$), log-concave density estimation in three or more dimensions is fundamentally more challenging in this minimax sense than estimating a density with two bounded derivatives.

The second main purpose of this paper is to provide bounds on the supremum risk with respect to the squared Hellinger loss function of a particular estimator, namely the log-concave maximum likelihood estimator $\hat{f}_n$.  The empirical process theory for studying maximum likelihood estimators is well-known \citep[e.g.][]{vanderVaartWellner1996,vandeGeer2000}, but relies on obtaining a bracketing entropy bound, which therefore becomes our main challenge.  A first step is to show that after standardising the data, and using the affine equivariance of the estimator, we can reduce the problem to maximising over a class $\mathcal{G}$ of log-concave densities having a small mean and covariance matrix close to the identity (cf. Lemma~\ref{Lemma:ThreeTerms} in the Appendix).  In Corollary~\ref{Cor:Transform} in Section~\ref{SubSec:IntEnv}, we derive an integrable envelope function for such classes, relying on certain properties of distributional limits of sequences of log-concave densities developed in Section~\ref{SubSec:Conv}.

%In Theorem~\ref{Thm:BracketingBounds} in Section~\ref{Sec:Bracketing}, we develop the bracketing entropy results that are key to deriving a rate of convergence for the log-concave maximum likelihood estimator.  
The first part of Section~\ref{Sec:Bracketing} is devoted to developing the key bracketing entropy results for the class $\mathcal{G}$.  In particular, we show that the $\epsilon$-bracketing number of $\mathcal{G}$ in Hellinger distance~$h$, denoted $N_{[]}(\epsilon,\mathcal{G},h)$ and defined at the beginning of Section~\ref{Sec:Bracketing}, satisfies
\begin{equation}
\label{Eq:Asymp}
\log N_{[]}(\epsilon,\mathcal{G},h) \gtrsim \max\{\epsilon^{-d/2},\epsilon^{-(d-1)}\}.
\end{equation}
The second term on the right-hand side of~\eqref{Eq:Asymp}, which dominates the first when $d \geq 3$, is somewhat unexpected in view of standard bracketing bounds for classes of convex functions on a compact domain taking values in $[0,1]$ \citep[e.g.][]{vanderVaartWellner1996,GuntuboyinaSen2013}, where only the first term on the right-hand side of~(\ref{Eq:Asymp}) appears.  Roughly speaking, it arises from the potential complexity of the domains of the log-densities.  Moreover, for $d \leq 3$, we obtain matching upper bounds, up to a logarithmic factor when $d=2$.  These upper bounds rely on intricate calculations of the bracketing entropy of classes of bounded, concave functions on an arbitrary closed, convex domain.  Further details on these bounds can be found in Section~\ref{Sec:Bracketing}.  
%For certain classes of convex domains, e.g. rectangles, or polyhedra that can be triangulated into a finite number of simplices, recent results of \citet{GuntuboyinaSen2013} and \citet{GaoWellner2015} have yielded $\epsilon$-bracketing entropy bounds in the $L_2$ metric of order $\epsilon^{-d/2}$.  Interestingly, however, it seems that 

%Our second main result of Section~\ref{Sec:Bracketing} involves an application of the bracketing entropy results described in the previous paragraph to derive a squared Hellinger risk bound for the log-concave maximum likelihood estimator.  
In the second part of Section~\ref{Sec:Bracketing}, we apply the bracketing entropy bounds described above to deduce that  
\begin{equation}
\label{Eq:Rates}
\sup_{f_0 \in \mathcal{F}_d} \mathbb{E}_{f_0}\{h^2(\hat{f}_n,f_0)\} = \left\{ \begin{array}{ll} O(n^{-4/5}) & \mbox{if $d=1$} \\
O(n^{-2/3}\log n) & \mbox{if $d=2$} \\
O(n^{-1/2}\log n) & \mbox{if $d=3$,} \end{array} \right.
\end{equation}
where $\mathcal{F}_d$ denotes the set of upper semi-continuous, log-concave densities on $\mathbb{R}^d$.  Thus, for $d \leq 3$, the log-concave maximum likelihood estimator attains the minimax optimal rate of convergence with respect to the squared Hellinger loss function, up to logarithmic factors when $d = 2,3$.  The stated rate when $d = 3$ is slower in terms of the exponent of $n$ than had been conjectured in the literature \citep[e.g.][p.~3778]{SereginWellner2010}, and arises as a consequence of the bracketing entropy being of order $\epsilon^{-(d-1)} = \epsilon^{-2}$ for this dimension. 

%Although~(\ref{Eq:Asymp}) describes the exact rate of growth of the bracketing number (up to multiplicative constants) as $\epsilon \searrow 0$, it remains possible that the actual maximum likelihood estimator convergence rate is faster than that stated here when $d \geq 3$.  However, \citet{BirgeMassart1993} give an example of a situation where the maximum likelihood estimator has a suboptimal rate of convergence agreeing with that predicted by the same empirical process theory from which we derive our rates.

It is interesting to note that the logarithmic penalties that appear in~\eqref{Eq:Rates} when $d=2,3$ occur for different reasons.  When $d=2$, the penalty arises from the logarithmic gap between the lower and upper bounds for the relevant bracketing entropy.  When $d=3$, the bracketing bound is sharp up to multiplicative constants, and the logarithmic penalty is due to the divergence of the bracketing entropy integral that plays the crucial role in the empirical process theory.  The bracketing entropy lower bound in~\eqref{Eq:Asymp} suggests (but does not prove) that the log-concave maximum likelihood estimator will be rate suboptimal for $d \geq 4$; indeed, \citet{BirgeMassart1993} give an example of a situation where the maximum likelihood estimator has a suboptimal rate of convergence agreeing with that predicted by the same empirical process theory from which we derive our rates.

All of our proofs are deferred to the Appendix, where we also give various auxiliary results.  We conclude this section by highlighting some related research on the pointwise rate of convergence of the log-concave maximum likelihood estimator.  \citet{BRW2009} proved that in the case $d=1$, if $f_0(x_0) > 0$ and $f_0$ is twice continuously differentiable in a neighbourhood of $x_0$ with $\phi_0''(x_0) < 0$, where $\phi_0 := \log f_0$, then $n^{2/5}\{\hat{f}_n(x_0) - f_0(x_0)\}$ converges to a non-degenerate limiting distribution related to the `lower invelope' of an integrated Brownian motion process minus a drift term.  \citet{SereginWellner2010} also derived a minimax lower bound for estimation of $f_0(x_0)$ with respect to absolute error loss of order $n^{-2/(d+4)}$, provided that $x_0$ is an interior point of the domain of $\log f_0$ and $\log f_0$ is locally strongly concave at $x_0$.  

\section{Minimax lower bounds}
\label{Sec:LowerBounds}

Let $\mu_d$ denote Lebesgue measure on $\mathbb{R}^d$, and recall that $\mathcal{F}_d$ denotes the set of upper semi-continuous, log-concave densities with respect to $\mu_d$, equipped with the $\sigma$-algebra it inherits as a subset of $L_1(\mathbb{R}^d)$.  Thus each $f \in \mathcal{F}_d$ can be written as $f = e^\phi$, for some upper semi-continuous, concave $\phi:\mathbb{R}^d \rightarrow [-\infty,\infty)$; in particular, we do not insist that $f$ is positive everywhere.  Let $X_1,\ldots,X_n$ be independent and identically distributed random vectors having some density $f \in \mathcal{F}_d$, and let $\mathbb{P}_f$ and $\mathbb{E}_f$ denote the corresponding probability and expectation operators, respectively.  An \emph{estimator} $\tilde{f}_n$ of $f$ is a measurable function from $(\mathbb{R}^d)^{\times n}$ to the class of probability densities with respect to $\mu_d$, and we write $\tilde{\mathcal{F}}_n$ for the class of all such estimators.  For $f,g \in L_1(\mathbb{R}^d)$, we define their squared Hellinger distance by 
\begin{equation}
\label{Eq:Hellinger}
h^2(f,g) := \int_{\mathbb{R}^d} (f^{1/2} - g^{1/2})^2 \, d\mu_d.
\end{equation}
This metric is both affine invariant and particularly convenient for studying maximum likelihood estimators.  Adopting a minimax approach, we define the \emph{supremum risk}
\[
R(\tilde{f}_n,\mathcal{F}_d) := \sup_{f \in \mathcal{F}_d} \mathbb{E}_f\bigl\{h^2(\tilde{f}_n,f)\};
\]
our aim in this section is to provide a lower bound for the infimum of $R(\tilde{f}_n,\mathcal{F}_d)$ over $\tilde{f}_n \in \tilde{\mathcal{F}}_n$.  
\begin{thm}
\label{Thm:LowerBound}
For each $d \in \mathbb{N}$, there exists $c_d > 0$ such that for every $n \geq d+1$,
\[
\inf_{\tilde{f}_n \in \tilde{\mathcal{F}}_n} R(\tilde{f}_n,\mathcal{F}_d) \geq \left\{ \begin{array}{ll} c_1 n^{-4/5} & \mbox{if $d=1$} \\
c_d n^{-2/(d+1)} & \mbox{if $d \geq 2$.} \end{array} \right.
\]
%where we may take $c_1 := 1/28000$, and $c_d := \frac{1}{500 \times 2^d}\bigl(\frac{15}{16}\bigr)^{(d+1)/2}$ for $d \geq 2$.  
\end{thm}
Theorem~\ref{Thm:LowerBound} reveals that when $d \geq 3$, the minimax lower bound rate for global loss functions is different from that for interior point estimation established under the local strong log-concavity condition in \citet{SereginWellner2010}.  

Our proof relies on a variant of Assouad's cube method; see, for example, \citet[][p.~347]{vanderVaart1998} or \citet[][pp.~118--9]{Tsybakov2009}.  We handle the cases $d=1$ and $d \geq 2$ separately.  For $d=1$, we bound the risk below by the risk over a finite subset of $\mathcal{F}_1$ consisting of densities that are perturbations of a semicircle $y = (r^2 - x^2)^{1/2}$ (it is convenient to raise the semicircle to be bounded away from zero on its domain so that the squared Hellinger distance can be bounded above in terms of the squared $L_2$-distance).  The perturbations are constructed by first dividing the upper portion of the semicircle into $K$ pairs of arcs, with each element of the pair being a reflection in the $y$-axis of the other.  For each $\alpha = (\alpha_1,\ldots,\alpha_K)^T \in \{0,1\}^K$ and $k=1,\ldots,K$, if $\alpha_k=1$, the $\alpha$th perturbation function $f_\alpha$ replaces the arc in the $k$th pair corresponding to $x > 0$ with a straight line joining its endpoints and retains the other arc in the pair; if $\alpha_k=0$, we reverse the roles of the two arcs in the pair.  Each function $f_\alpha$ is concave on its support $[-r,r]$, and is contructed to be a density; Assouad's lemma can therefore be applied.

For $d \geq 2$, we instead construct uniform densities on perturbations of a closed Euclidean ball $B$.  We first start with a constant function on $B$, and find $K$ pairs of disjoint caps in $B$.  For $\alpha = (\alpha_1,\ldots,\alpha_K)^T \in \{0,1\}^K$ and $k=1,\ldots,K$, if $\alpha_k=1$, the $\alpha$th perturbation function $f_\alpha$ is zero for the first element of the pair, and agrees with the constant function for the second; if $\alpha_k = 0$, the roles of the two elements of the pair are again reversed.  Since the resulting densities $\{f_\alpha:\alpha \in \{0,1\}^K\}$ are uniform on sets of the same volume, we can compute Hellinger distances between them and again apply Assouad's lemma.  

As can be seen from the above descriptions, the same lower bounds hold for the (smaller) class of upper semi-continuous densities on $\mathbb{R}^d$ that are concave on their support; indeed, for $d \geq 2$, the lower bounds hold even for the class of uniform densities on a closed, convex domain.  Since the domains in our construction are perturbations of a Euclidean ball, the problem is rather similar to that of estimating a convex body based on a sample of size $n$ with respect to the Nikodym distance, defined as the Lebesgue measure of the symmetric difference of two sets.  For this latter problem, the rate of $n^{-2/(d+1)}$ has also been obtained \citep{KorostelevTsybakov1993,MammenTsybakov1995,Brunel2014}.

An inspection of our proof further reveals that a minimax lower bound can also be obtained for the $L_2^2$ loss function.  Note that in this case, the loss function is not affine invariant, so it makes sense to restrict attention to log-concave densities $f$ with a lower bound on the determinant of the corresponding covariance matrix $\Sigma_f$.  The result obtained is that there exist $c_d' > 0$ such that for every $n \geq d+1$ and every $\rho > 0$, 
\[
\inf_{\tilde{f}_n \in \tilde{\mathcal{F}}_n} \sup_{f \in \mathcal{F}_d:\det(\Sigma_f) \geq \rho^2} \mathbb{E}_f L_2^2(\tilde{f}_n,f) \geq 
\left\{ \begin{array}{ll} c_1' n^{-4/5}/\rho & \mbox{if $d=1$} \\
c_d' n^{-2/(d+1)}/\rho & \mbox{if $d \geq 2$.} \end{array} \right.
\]

\section{Convergence and integrable envelopes}
\label{Sec:ConvInt}

We begin this section with some general results characterising the possible limits of sequences of log-concave densities on $\mathbb{R}^d$.  We will not require the full strength of these results in the rest of the paper (though we will apply Propositions~\ref{Prop:Conv} and~\ref{Prop:Conv3} when studying integrable envelopes in Section~\ref{SubSec:IntEnv} below), but we believe they will be of some independent interest.

\subsection{Convergence of log-concave densities}
\label{SubSec:Conv}

If $A$ is a $k$-dimensional affine subset of $\mathbb{R}^d$, we write $\mu_{k,A}$ for $k$-dimensional Lebesgue measure on $A$, and let $\mu_d := \mu_{d,\mathbb{R}^d}$ to agree with our previous notation.  We also write $\mathcal{F}_{k,A}$ for the class of upper semi-continuous, log-concave densities with respect to $\mu_{k,A}$ on $A$.  If $f:A \rightarrow [0,\infty)$ is a log-concave function, write $\mathrm{cl}(f)$ for its closure; thus $\mathrm{cl}(f)(x) := \limsup_{y \rightarrow x} f(y)$; if $f$ is also a density with respect to $\mu_{k,A}$ then $\mathrm{cl}(f) \in \mathcal{F}_{k,A}$.  If $\nu$ is a probability measure on $A$, we write $\mathrm{csupp}(\nu)$ for its \emph{convex support}; that is, $\mathrm{csupp}(\nu)$ is the smallest closed, convex subset of $A$ with $\nu$-measure 1.  If $C \subseteq \mathbb{R}^d$, let $C^c$, $\bar{C}$, $\mathrm{int}(C)$, $\mathrm{bd}(C)$, $\mathrm{conv}(C)$, $\mathrm{aff}(C)$ denote its complement, closure, interior, boundary, convex hull and affine hull respectively; if $C$ is convex, we write $\dim(C)$ for its dimension.  Let $B_d(x_0,\delta)$ and $\bar{B}_d(x_0,\delta)$ respectively denote the open and closed Euclidean balls of radius $\delta > 0$ centred at $x_0 \in \mathbb{R}^d$. 

Throughout this subsection, we let $f_1, f_2,\ldots$ be a sequence in $\mathcal{F}_d$, and let $\nu_n$ be the probability measure on $\mathbb{R}^d$ corresponding to $f_n$.  We suppose that $\nu_n \stackrel{d}{\rightarrow} \nu$, for some probability measure $\nu$, and let $C = \{x \in \mathbb{R}^d: \liminf f_n(x) > 0\}$.  Our first proposition deals with the most straightforward situation.
\begin{prop}
\label{Prop:Conv}
If either $\dim\bigl(\mathrm{csupp}(\nu)\bigr) = d$ or $\dim(C) = d$, then $\mathrm{csupp}(\nu) = \bar{C}$.  Moreover, under either condition, $\nu$ is absolutely continuous with respect to $\mu_d$, with Radon--Nikodym derivative $\mathrm{cl}(\liminf f_n) \in \mathcal{F}_d$.  %Moreover, again under either condition, if $S$ is a compact set not intersecting $\bar{C}$, then $\sup_{x \in S} f_n(x) \rightarrow 0$.  
\end{prop}
The second part of Proposition~\ref{Prop:Conv} weakens the hypothesis of Proposition~2(a) of \citet{CuleSamworth2010}, where the limiting measure was assumed a priori to be absolutely continuous with respect to Lebesgue measure on $\mathbb{R}^d$.  The correspondence between $\mathrm{csupp}(\nu)$ and $\bar{C}$ in the first part leads one to hope that a similar relationship might hold in more general scenarios where the dimensions of $\mathrm{csupp}(\nu)$ and $C$ are smaller than $d$ (so the limiting measure is degenerate).  The following examples, however, dispel such optimism.
\begin{enumerate}[(i)]
\item It is not in general the case that $\mathrm{csupp}(\nu) \subseteq \mathrm{aff}(C)$.  For instance, if $f_n$ denotes the (log-concave) density of a random variable with a $N(1/n,1/n^4)$ distribution, then $C = \emptyset$ but $\mathrm{csupp}(\nu) = \{0\}$.  
\item Even if $\mathrm{csupp}(\nu) \subseteq \mathrm{aff}(C)$, we do not necessarily have $\mathrm{csupp}(\nu) \subseteq \bar{C}$.  For instance, if $f_n$ denotes the density of a bivariate normal random vector with mean 0 and covariance matrix $\begin{pmatrix} 1 & \rho_n \sigma_n \\ \rho_n \sigma_n & \sigma_n^2\end{pmatrix}$, with $\sigma_n = 1/n$ and $\rho_n = \sqrt{1 - 1/\log n}$, then a straightforward calculation shows that $C = [-\sqrt{2},\sqrt{2}] \times \{0\}$, while $\mathrm{csupp}(\nu) = \mathbb{R} \times \{0\}$.
\item It is also not in general the case that $C \subseteq \mathrm{aff}\bigl(\mathrm{csupp}(\nu))$.  For instance, if $f_n$ denotes the density of a bivariate normal random vector with mean 0 and covariance matrix $\begin{pmatrix} 1/n & 0 \\ 0 & e^{-n^2} \end{pmatrix}$, then $C = \mathbb{R} \times \{0\}$, while $\mathrm{csupp}(\nu) = \{0\} \times \{0\}$.  
\item Even if $C \subseteq \mathrm{aff}\bigl(\mathrm{csupp}(\nu)\bigr)$, we do not necessarily have $\bar{C} \subseteq \mathrm{csupp}(\nu)$.  For instance, let $f_n$ denote the density of the bivariate random vector $\begin{pmatrix} X_n \\ Y_n \end{pmatrix}$, where $X_n$ and $Y_n$ are independent, where $X_n$ has density
\[
f_{n,X_n}(x) := \frac{1}{2(1+1/n)}\mathbbm{1}_{\{x \in [-1,1]\}} + \frac{1}{2(1+1/n)}e^{-n|x-1|} \mathbbm{1}_{\{|x| > 1\}},
\]
and $Y_n \sim N(0,e^{-n^2})$.  Then $\begin{pmatrix} X_n \\ Y_n \end{pmatrix} \stackrel{d}{\rightarrow} U[-1,1] \otimes \delta_0$, so $\mathrm{csupp}(\nu) = [-1,1] \times \{0\}$.  But $C = \mathbb{R} \times \{0\}$.
\end{enumerate}
Despite these chastening examples, we can still make the following statements with regard to the situation where $\nu$ is degenerate.
\begin{prop}
\label{Prop:Conv2}
\begin{enumerate}
\item If $\dim(C) = d-1$ and $S$ is a compact set not intersecting $\mathrm{aff}(C)$, then $\sup_{x \in S} f_n(x) \rightarrow 0$; in particular, $\mathrm{csupp}(\nu) \subseteq \mathrm{aff}(C)$.
\item Let $U$ denote the unique subspace of $\mathbb{R}^d$ such that $\mathrm{aff}\bigl(\mathrm{csupp}(\nu)\bigr) = U + a$, for some $a \in \mathbb{R}^d$.  Let $k = \dim(U)$, and let $U^\perp$ denote the orthogonal complement of $U$.  For $u \in U$, let $f_{n,U}(u+a) = \mathrm{cl}\bigl(\int_{U^\perp} f_n(u+a+w) \, dw\bigr)$.  Then $\nu$ is absolutely continuous with respect to $\mu_{k,U+a}$, with Radon--Nikodym derivative $\mathrm{cl}(\liminf f_{n,U}) \in \mathcal{F}_{k,U+a}$.
\end{enumerate}
\end{prop}
Finally in this subsection, we show that even in the situation where $\nu$ is degenerate, the convergence in distribution of log-concave measures implies much stronger forms of convergence.  Similar results were proved in Theorem~2.1 and Proposition~2.2 of \citet{SHD2011} under the stronger assumption that $\nu$ has a log-concave Radon--Nikodym derivative with respect to $\mu_d$.
\begin{prop}
\label{Prop:Conv3}
Let $\Theta = \bigl\{\theta \in \mathbb{R}^d: \int_{\mathbb{R}^d} e^{\theta^T x} \, d\nu(x) < \infty\bigr\}$.  Then, with $U^\perp$ defined as in Proposition~\ref{Prop:Conv2}, we have $\Theta = \Theta_0 \oplus U^\perp$, where $\Theta_0$ is relatively open in $U$, convex, and contains 0.  Moreover, for every $\theta \in \Theta$, we have
\[
\int_{\mathbb{R}^d} e^{\theta^T x} \, d\nu_n(x) \rightarrow \int_{\mathbb{R}^d} e^{\theta^T x} \, d\nu(x) 
\]
as $n \rightarrow \infty$. 
\end{prop}
We note for later use that as an immediate corollary of Proposition~\ref{Prop:Conv3}, if $\Sigma_n$ denotes the covariance matrix corresponding to $\nu_n$, and $\Sigma$ denotes the covariance matrix corresponding to $\nu$, then $\Sigma_n \rightarrow \Sigma$. 

\subsection{Integrable envelopes for classes of log-concave densities}
\label{SubSec:IntEnv}

Part~(a) of the following result is important for establishing our bracketing entropy bounds in Section~\ref{Sec:Bracketing}.  Part~(b) is used in Lemma~\ref{Lemma:ThreeTerms} to obtain a lower bound for the smallest eigenvalue of the covariance matrix corresponding to the log-concave projection of a distribution whose own covariance matrix is close to the identity.  For $f \in \mathcal{F}_d$, let $\mu_f := \int_{\mathbb{R}^d} xf(x) \, dx$ and $\Sigma_f := \int_{\mathbb{R}^d} (x-\mu_f)(x-\mu_f)^Tf(x) \, dx$.  For $\mu \in \mathbb{R}^d$ and a symmetric, positive-definite, $d \times d$ matrix $\Sigma$, let
\[
\mathcal{F}_d^{\mu,\Sigma} := \bigl\{f \in \mathcal{F}_d: \mu_{f} = \mu, \Sigma_{f} = \Sigma\bigr\}.
\]
\begin{thm}
\label{Thm:IntEnv}
\begin{enumerate}[(a)] 
\item For each $d \in \mathbb{N}$, there exist $A_{0,d},B_{0,d} > 0$ such that for all $x \in \mathbb{R}^d$, we have
\[
\sup_{f \in \mathcal{F}_d^{0,I}} f(x) \leq e^{-A_{0,d}\|x\|+B_{0,d}}.
\]
\item For $\|x\| \leq 1/4$, we have
\[
\inf_{f \in \mathcal{F}_d^{0,I}} f(x) > 0.
\]
\end{enumerate}
\end{thm}
In fact, it will be convenient to have the corresponding envelopes for slightly larger classes.  We write $\lambda_{\mathrm{min}}(\Sigma)$ and $\lambda_{\mathrm{max}}(\Sigma)$ for the smallest and largest eigenvalues respectively of a positive-definite, symmetric $d \times d$ matrix $\Sigma$.  For $\xi \geq 0$ and $\eta \in (0,1)$, let
\[
\tilde{\mathcal{F}}_d^{\xi,\eta} := \{\tilde{f} \in \mathcal{F}_d: \|\mu_{\tilde{f}}\| \leq \xi \ \text{and} \ 1-\eta \leq \lambda_{\mathrm{min}}(\Sigma_{\tilde{f}}) \leq \lambda_{\mathrm{max}}(\Sigma_{\tilde{f}}) \leq 1 + \eta\}.
\]
\begin{corollary}
\label{Cor:Transform}
\begin{enumerate}[(a)] 
\item For each $d \in \mathbb{N}$, there exist $A_{0,d}, B_{0,d} > 0$ such that for every $\xi \geq 0$, every $\eta \in (0,1)$ and every $x \in \mathbb{R}^d$, we have
\[
\sup_{\tilde{f} \in \tilde{\mathcal{F}}_d^{\xi,\eta}} \tilde{f}(x) \leq (1-\eta)^{-d/2} \exp\biggl\{-\frac{A_{0,d}\|x\|}{(1+\eta)^{1/2}} + \frac{A_{0,d}\xi}{(1+\eta)^{1/2}} + B_{0,d}\biggr\}.
\]
\item For every $\xi \geq 0$ and $\eta \in (0,1)$ satisfying $\xi < (1-\eta)^{1/2}/4$ and for every $\|x\| \leq (1-\eta)^{1/2}/4 - \xi$,  we have
\[
\inf_{\tilde{f} \in \tilde{\mathcal{F}}_d^{\xi,\eta}} \tilde{f}(x) > 0.
\]
\end{enumerate} 
\end{corollary}
As an ancillary result, we can also give a precise envelope for the class of one-dimensional log-concave densities having mean zero and with no variance restriction.  Let
\[
\mathcal{F}_1^0 := \bigl\{f \in \mathcal{F}_1: \mu_f = 0\bigr\}.
\]
\begin{prop}
\label{Prop:Env2}
For every $x_0 \in \mathbb{R}$, we have
\[
\sup_{f \in \mathcal{F}_1^0} f(x_0) = 1/|x_0|,
\]
where we interpret $1/0 = \infty$.
\end{prop}
While the envelope function here is not integrable, this result is reminiscent of the fact that $f(x) \leq 1/(2x)$ for all $x > 0$, when $f$ is a convex density on $(0,\infty)$, which was proved and exploited in \citet{GJW2001}.

\section{Bracketing entropy bounds and global rates of convergence of the log-concave maximum likelihood estimator}
\label{Sec:Bracketing}

Let $\mathcal{G}$ be a class of functions on $\mathbb{R}^d$, and let $\rho$ be a semi-metric on $\mathcal{G}$.  For $\epsilon > 0$, we write $N_{[]}(\epsilon,\mathcal{G},\rho)$ for the $\epsilon$-bracketing number of $\mathcal{G}$ with respect to $\rho$.  Thus $N_{[]}(\epsilon,\mathcal{G},\rho)$ is the minimal $N \in \mathbb{N}$ such that there exist pairs $\{(g_j^L,g_j^U)\}_{j=1}^N$ with the properties that $\rho(g_j^L,g_j^U) \leq \epsilon$ for all $j=1,\ldots,N$ and, for each $g \in \mathcal{G}$, there exists $j^* \in \{1,\ldots,N\}$ satisfying $g_{j^*}^L \leq g \leq g_{j^*}^U$.  The following entropy bound is key to establishing the rate of convergence of the log-concave maximum likelihood estimator in Hellinger distance.
\begin{thm}
\label{Thm:BracketingBounds}
Let $\eta_d > 0$ be taken from Lemma~\ref{Lemma:ThreeTerms} in the Appendix.  

(i) There exist $\overline{K}_1, \overline{K}_2, \overline{K}_3  \in (0,\infty)$ such that
\[
\log N_{[]}(\epsilon,\tilde{\mathcal{F}}_d^{1,\eta_d},h) \leq \left\{ \begin{array}{ll} \overline{K}_1 \epsilon^{-1/2} & \mbox{when $d=1$} \\
\overline{K}_2 \epsilon^{-1}\log_{++}^{3/2}(1/\epsilon) & \mbox{when $d=2$} \\
\overline{K}_3 \epsilon^{-2} & \mbox{when $d=3$,} \end{array} \right.
\]
for all $\epsilon > 0$, where $\log_{++}(x) := \max(1,\log x)$.

(ii) For every $d \in \mathbb{N}$, there exist $\epsilon_d \in (0,1]$ and $\underline{K}_d \in (0,\infty)$ such that
\[
\log N_{[]}(\epsilon,\tilde{\mathcal{F}}_d^{1,\eta_d},h) \geq \underline{K}_d \max\{\epsilon^{-d/2},\epsilon^{-(d-1)}\}
\]
for all $\epsilon \in (0,\epsilon_d]$.  
\end{thm} 
Note that in this theorem, $\eta_d$ depends only on $d$.  The proof of Theorem~\ref{Thm:BracketingBounds} is long, so we give a broad outline here.  For the upper bound, we first consider the problem of finding a set of Hellinger brackets for the class of restrictions of densities $\tilde{f} \in \tilde{\mathcal{F}}_d^{1,\eta_d}$ to $[0,1]^d$.  It is well-known \citep[e.g][Corollary~2.7.10]{vanderVaartWellner1996} that the class of concave functions from a $d$-dimensional compact, convex subset of $\mathbb{R}^d$ to $[-1,0]$ with uniform Lipschitz constant $L > 0$ satisfies a uniform norm bracketing entropy bound of the form $(1+L)^{d/2}\epsilon^{-d/2}$.  The class $\{\log \tilde{f}: \tilde{f} \in \tilde{\mathcal{F}}_d^{1,\eta_d}\}$ does not satisfy a uniform Lipschitz condition, however.  Nevertheless, some hope is provided by a result of \citet{GuntuboyinaSen2013}, who showed that when working with rectangular domains and the $L_2$-metric (or more generally, $L_r$-metrics with $r \in [1,\infty)$), a metric entropy bound of the same order in $\epsilon$ can be obtained without the Lipschitz condition (but still with the uniform lower bound condition).  This result was recently extended both from metric to bracketing entropy, and from rectangular to convex polyhedral domains, by \citet{GaoWellner2015}.  Unfortunately, it remains a substantial challenge to provide bracketing entropy bounds for general convex domains when $d \geq 2$.  In Proposition~\ref{Prop:BoundedBrackets} in the Appendix, we are able to obtain such bounds when $d=2,3$ by constructing inner layers of convex polyhedral approximations where the number of simplices required to triangulate the region between successive layers can be controlled using results from discrete convex geometry.  It is the absence of corresponding convex geometry results for $d \geq 4$ that means we are currently unable to provide bracketing entropy bounds in these higher dimensions.
 
A further challenge is to deal with the fact that if $\tilde{f} \in \tilde{\mathcal{F}}_d^{1,\eta_d}$, then $\log \tilde{f}$ can take negative values of arbitrarily large magnitude, and may even be $-\infty$.  We therefore define a finite sequence of levels $y_0,y_1,\ldots,y_{k_0}$, where $y_0$ is a uniform upper bound for the class $\{\log \tilde{f}:\tilde{f} \in \tilde{\mathcal{F}}_d^{1,\eta_d}\}$ obtained from Corollary~\ref{Cor:Transform}, and divide the class of restrictions of densities $\tilde{f} \in \tilde{\mathcal{F}}_d^{1,\eta_d}$ to $[0,1]^d$ into $(k_0+1)$ subclasses, where in the $k$th class ($k=1,\ldots,k_0$), the log-density is bounded below by $-y_k$ on its domain, with the remaining functions placed in the $(k_0+1)$th subclass.  The domains are unknown, so we derive inductively upper bounds for the bracketing Hellinger entropy of the $k$th class ($k=1,\ldots,k_0$) by first constructing a bracketing set for its domain, and then, for each such bracket, using Proposition~\ref{Prop:BoundedBrackets} to construct a bracketing set for the log-density on the inner domain-bracketing set.  Since we can only use crude bounds for the brackets on the (small) region between the inner and outer domain bracketing sets, and since the domain of a function in the $k$th subclass can be an arbitrary $d$-dimensional, closed, convex subset of $[0,1]^d$, we need for instance $e^{O(\epsilon^{-2})}$ brackets to cover these domains when $d = 3$.  This is a stark contrast with the univariate setting studied by \citet{DossWellner2015}, where a similar general strategy was introduced, but where only $O(\epsilon^{-2})$ brackets are needed for the domains. 

Crucially, we can afford to be more liberal in the accuracy of our coverage as $k$ increases, because the contribution to the Hellinger distance is small when the log-density has a negative value of large magnitude.  This enables us to show that the total number of brackets required to construct a bracketing set with Hellinger distance at most $\epsilon$ between the brackets is bounded above by an expression not depending on $k_0$.  For the $(k_0+1)$th class, we can modify the brackets used for the $k_0$th class in a straightforward way.

Translations of these brackets can be used to cover the restrictions of densities $\tilde{f} \in \tilde{\mathcal{F}}_d^{1,\eta_d}$ to other unit boxes.  We use our integrable envelope function for the class $\tilde{\mathcal{F}}_d^{1,\eta_d}$ from Corollary~\ref{Cor:Transform} again to allow us to use fewer brackets as the boxes move further from the origin, yet still cover with higher accuracy, enabling us to obtain the desired conclusion.  

For the lower bound, we treat the cases $d=1$ and $d \geq 2$ separately.  In both cases, we use the Gilbert--Varshamov theorem and packing set bounds for the unit sphere to construct a finite subset of $\tilde{\mathcal{F}}_d^{1,\eta_d}$ of the desired cardinality where each pair of functions is well separated in Hellinger distance. The key observation here is that, while in the $d=1$ case it suffices to consider a fixed domain, when $d \geq 2$, the domains of the functions in our finite subset are allowed to vary. 

We are now in a position to state our main result on the supremum risk of the log-concave maximum likelihood estimator for the squared Hellinger loss function.
\begin{thm}
\label{Thm:Main}
Let $\hat{f}_n$ denote the log-concave maximum likelihood estimator based on a sample of size $n$.  Then, for the squared Hellinger loss function,
\[
R(\hat{f}_n,\mathcal{F}_d) = \left\{ \begin{array}{ll} O(n^{-4/5}) & \mbox{if $d=1$} \\
O(n^{-2/3}\log n) & \mbox{if $d=2$} \\
O(n^{-1/2}\log n) & \mbox{if $d=3$.} \end{array} \right.
\]
\end{thm}
The proof of this theorem first involves standardising the data and using affine equivariance to reduce the problem to that of bounding the supremum risk over the class of log-concave densities with mean vector 0 and identity covariance matrix.  Writing $\hat{g}_n$ for the log-concave maximum likelihood estimator for the standardised data, we show in Lemma~\ref{Lemma:ThreeTerms} in the Appendix that  
\[
\sup_{g_0 \in \mathcal{F}_d^{0,I}} \mathbb{P}_{g_0}(\hat{g}_n \notin \tilde{\mathcal{F}}_d^{1,\eta_d}) = O(n^{-1}).
\]
As well as using various known results on the relationship between the mean vector and covariance matrix of the log-concave maximum likelihood estimator in relation to its sample counterparts, the main step here is to show that, provided none of the sample covariance matrix eigenvalues are too large, the only way an eigenvalue of the covariance matrix corresponding to the maximum likelihood estimator can be small is if an eigenvalue of the sample covariance matrix is small.

The other part of the proof of Theorem~\ref{Thm:Main} is to control
\[
\sup_{g_0 \in \mathcal{F}_d^{0,I}} \mathbb{E}\bigl\{h^2(\hat{g}_n,g_0)\mathbbm{1}_{\{\hat{g}_n \in \tilde{\mathcal{F}}_d^{1,\eta_d}\}}\bigr\}.
\]
This can be done by appealing to empirical process theory for maximum likelihood estimators, and using the Hellinger bracketing entropy bounds developed in Theorem~\ref{Thm:BracketingBounds}. 

\vspace{1cm}

{\large \textbf{Acknowledgements}}: The work of the second author was supported by an EPSRC Early Career Fellowship and a grant from the Leverhulme Trust.  The authors are very grateful for helpful comments on an earlier draft from Charles Doss, Roy Han and Jon Wellner, as well as anonymous reviewers.

\section{Appendix}
\label{Sec:Appendix}

\subsection{Proofs from Section~\ref{Sec:LowerBounds}}
\label{Sec:ProofLowerBound}

\begin{proof}[Proof of Theorem~\ref{Thm:LowerBound}] \emph{The case $d=1$}: We define a finite subset $\bar{\mathcal{F}}_1$ of $\mathcal{F}_1$ to which we can apply the version of Assouad's lemma stated as Lemma~\ref{Lemma:Kim} in Section~\ref{Sec:ProofLowerBound}.  Recall from the description of the proof in Section~\ref{Sec:LowerBounds} that the densities in our finite subset are perturbations of a semi-circle, raised to be bounded away from zero on its support.  Fix $\epsilon := n^{-1/5}/2 \leq 1/2$ and set $r := 2/3$ and $\theta_k := k\arcsin \epsilon$ for $k \in \mathbb{N}$.  Let $K := \lfloor \frac{\pi}{6\theta_1} \rfloor \geq 1$, so $K$ is the largest positive integer such that $\cos \theta_{2K} \geq 1/2$.  For $k=1,\ldots,K$ and $\ell \in \{0,1\}$, set 
\[
x_{k,\ell} := (-1)^\ell r(1-\epsilon^2)^{1/2}\sin \theta_k.
\]
For $k=1,\ldots,K$, we also define intervals
\[
R_{k,0} := (r \sin \theta_{2k-2},r \sin \theta_{2k}),
\]
and set $R_{k,1} := -R_{k,0} = \{-x:x \in R_{k,0}\}$.  Writing $y_k := r(1-\epsilon^2)^{1/2}\cos \theta_{2k-1}$, for $k=1,\ldots,K$, we define auxiliary functions 
\begin{align*}
\psi_k(x) &:= (r^2-x^2)^{1/2} \mathbbm{1}_{\{x \in R_{k,0}\}} + \frac{1}{y_k}\{(1 - \epsilon^2)r^2 - x_{k,1} x\}\mathbbm{1}_{\{x \in R_{k,1}\}}, \\
\tilde{\psi}_k(x) &:= \frac{1}{y_k}\{(1 - \epsilon^2)r^2 - x_{k,0} x\}\mathbbm{1}_{\{x \in R_{k,0}\}} + (r^2-x^2)^{1/2} \mathbbm{1}_{\{x \in R_{k,1}\}}.
\end{align*}
A generic perturbation $\tilde{\psi}_k$ is illustrated in Figure~\ref{Fig:psi}.  
\begin{figure}
\begin{center}
\includegraphics[width=0.7\textwidth]{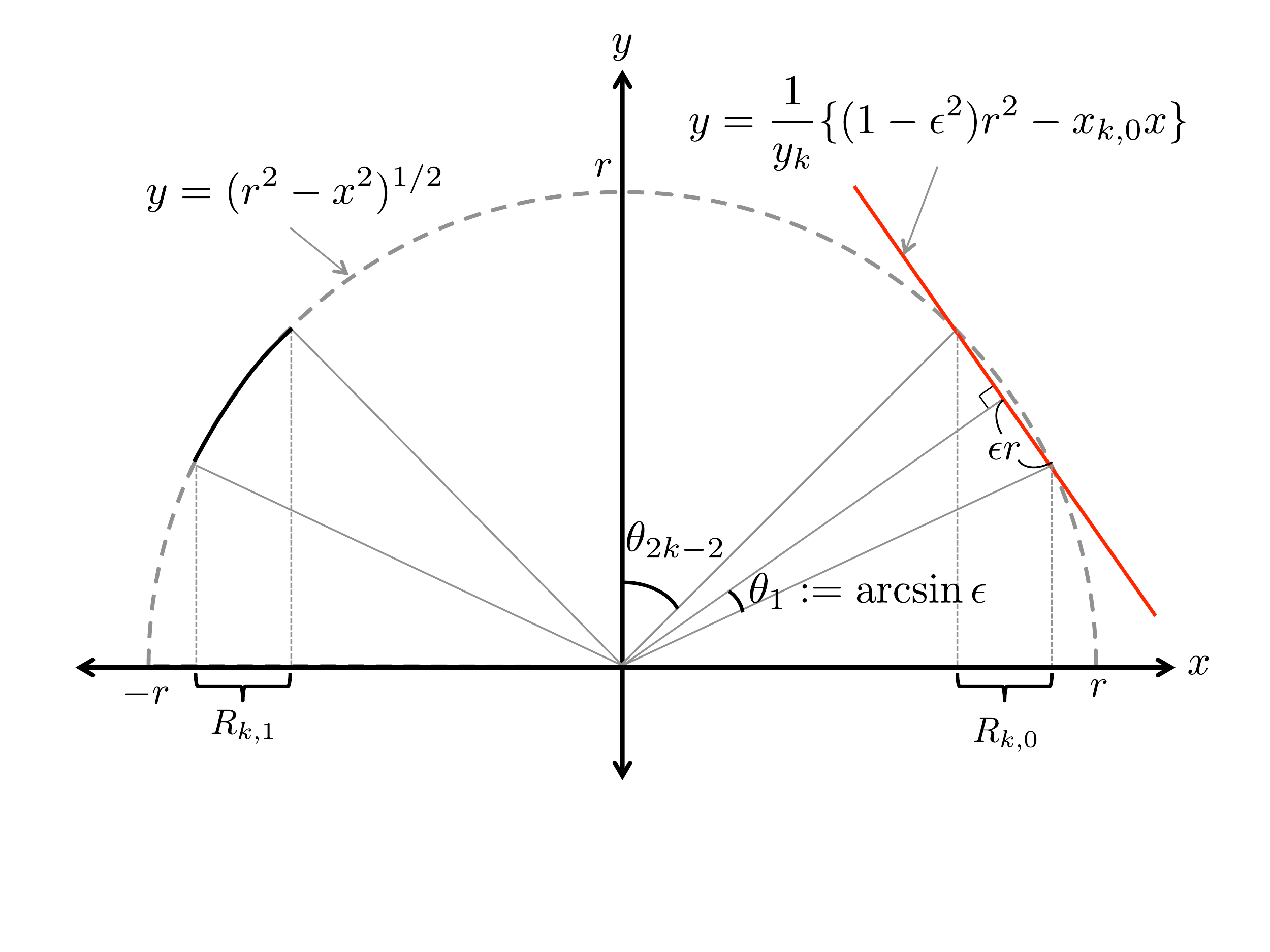}
\caption{\label{Fig:psi} A generic perturbation function $\tilde{\psi}_k$ used in the proof of Theorem~\ref{Thm:LowerBound} when $d=1$.}
\end{center}
\end{figure}
Finally, then, we can define $\bar{\mathcal{F}}_1 := \{f_\alpha: \alpha = (\alpha_1,\ldots,\alpha_K)^T \in \{0,1\}^K\}$, where
\[
f_\alpha(x) := c_{r,K,\epsilon}\mathbbm{1}_{\{|x| \leq r\}} + (r^2-x^2)^{1/2} \mathbbm{1}_{\{|x| \leq r\}}\mathbbm{1}_{\{x \notin \cup_{k=1}^K (R_{k,0} \cup R_{k,1})\}} + \sum_{k=1}^K \bigl\{\alpha_k \psi_k(x) + (1-\alpha_k)\tilde{\psi}_k(x)\bigr\},
\]
and
\[
c_{r,K,\epsilon} := \frac{1}{2r}\biggl[1 - \frac{1}{2}\pi r^2 + Kr^2\{\theta_1 - \epsilon(1-\epsilon^2)^{1/2}\}\biggr].
\]
With $r=2/3$, we have $c_{r,K,\epsilon} \geq \frac{3}{4}(1 - 2\pi/9) =: c_0$.  Note that the hypograph (or subgraph) of $f_\alpha$, defined by $\mathrm{hyp}(f_\alpha) := \{(x,y) \in \mathbb{R} \times \mathbb{R}: y \leq f_\alpha(x)\}$, is the intersection of the closed, convex set $\{(x,y) \in [-r,r] \times \mathbb{R}:y \leq c_{r,K,\epsilon} + (r^2-x^2)^{1/2}\}$ with $K$ closed halfspaces, so is closed and convex.  Hence, $f_\alpha$ is upper semi-continuous and concave on $[-r,r]$, so by, e.g., \citet[][p.~86]{DharmadhikariJoagDev1988}, $\bar{\mathcal{F}}_1 \subseteq \mathcal{F}_1$, and it remains to verify the two conditions of Lemma~\ref{Lemma:Kim}.  First, note that if $\alpha = (\alpha_1,\ldots,\alpha_K)^T,\beta = (\beta_1,\ldots,\beta_K)^T \in \{0,1\}^K$, then
\[
h^2(f_\alpha,f_\beta) = \int_{\mathbb{R}^d} \frac{(f_\alpha^{1/2} + f_\beta^{1/2})^2}{(f_\alpha^{1/2} + f_\beta^{1/2})^2}(f_\alpha^{1/2} - f_\beta^{1/2})^2 \geq \frac{1}{4(r+c_{r,K,\epsilon})}L_2^2(f_\alpha,f_\beta) \geq \frac{1}{4}L_2^2(f_\alpha,f_\beta).
\]
Moreover, if $|\alpha_k - \beta_k| = 1$, then
\begin{align*}
\int_{R_{k,0}} (f_\alpha - f_\beta)^2 &= \int_{r\sin \theta_{2k-2}}^{r\sin \theta_{2k}} \biggl[(r^2-x^2)^{1/2} - \frac{1}{y_k}\{(1 - \epsilon^2)r^2 - x_{k,0} x\}\biggr]^2 \, dx \\
&=: r^3 \int_{\sin(\theta_{2k-1} - \theta_1)}^{\sin(\theta_{2k-1} + \theta_1)}  I(t,\theta_{2k-1})^2 \, dt,
\end{align*}
say, where
\[
I(t,\theta) := (1-t^2)^{1/2} - \frac{(1-\epsilon^2)^{1/2}}{\cos \theta} + t \tan \theta.
\]
It is convenient to observe first that
\[
\int_{\sin(\theta - \theta_1)}^{\sin(\theta + \theta_1)} I(t,\theta)^2 \, dt
\]
is a monotonically increasing function of $\theta \in [0,\pi/3]$.  To check this, note that by differentiating under the integral, splitting the range of integration into two intervals of equal length, and then making the substitution $t \mapsto 2(1-\epsilon^2)^{1/2}\sin \theta - t$ in the left interval, we find that 
\begin{align*}
\frac{d}{d\theta} \int_{\sin(\theta - \theta_1)}^{\sin(\theta + \theta_1)} I(t,\theta)^2 \, dt &= 2\int_{\sin(\theta - \theta_1)}^{\sin(\theta + \theta_1)}  I(t,\theta)\biggl\{\frac{t - (1-\epsilon^2)^{1/2}\sin \theta}{\cos^2 \theta}\biggr\} \, dt \\
&=:  2\int_{(1-\epsilon^2)^{1/2}\sin \theta}^{\sin(\theta + \theta_1)} J(t,\theta)\biggl\{\frac{t - (1-\epsilon^2)^{1/2}\sin \theta}{\cos^2 \theta}\biggr\} \, dt,
\end{align*}
where
\[
J(t,\theta) := I(t,\theta) - I(2(1-\epsilon^2)^{1/2}\sin \theta - t,\theta).
\]
But $J((1-\epsilon^2)^{1/2}\sin \theta,\theta) = J(\sin(\theta + \theta_1),\theta) = 0$, and for $t \in [(1-\epsilon^2)^{1/2}\sin \theta,\sin(\theta + \theta_1)]$, we have
\[
\frac{\partial^2}{\partial t^2} J(t,\theta) = -(1-t^2)^{-3/2} + [1-\{2(1-\epsilon^2)^{1/2}\sin \theta - t\}^2]^{-3/2} \leq 0.
\]
We deduce that $J(t,\theta) \geq 0$ for all $t \in [(1-\epsilon^2)^{1/2}\sin \theta,\sin(\theta + \theta_1)]$, and our desired monotonicity as a function of $\theta$ follows.  Hence, for any $\alpha,\beta \in \{0,1\}^K$, we have
\begin{align*}
h^2(f_\alpha,f_\beta) \geq \frac{1}{4}L_2^2(f_\alpha,f_\beta) &\geq \frac{1}{2}\|\alpha - \beta\|_0r^3  \int_{-\epsilon}^{\epsilon} \{(1-t^2)^{1/2} - (1-\epsilon^2)^{1/2}\}^2 \, dt \\
&= \frac{1}{2}\|\alpha - \beta\|_0r^3\biggl\{2\epsilon - \frac{2\epsilon^3}{3} - 2(1-\epsilon^2)^{1/2}\theta_1\biggr\} \\
&\geq \frac{1}{2}\|\alpha - \beta\|_0r^3\biggl\{2\epsilon - \frac{2\epsilon^3}{3} - 2\biggl(1-\frac{\epsilon^2}{2} - \frac{\epsilon^4}{8}\biggr)\biggl(\epsilon + \frac{\epsilon^3}{6} + \frac{113\epsilon^5}{840}\biggr)\biggr\} \\
&\geq \frac{31}{420}\|\alpha-\beta\|_0r^3\epsilon^5.
\end{align*}
This calculation shows that, for the squared Hellinger loss function, we can take $\gamma := \frac{31}{420}r^3 \epsilon^5$ in condition~(i) of Lemma~\ref{Lemma:Kim}.

We now turn to condition (ii).  Since $h^2(f_\alpha,f_\beta) \leq L_2^2(f_\alpha,f_\beta)/(4c_0)$ for all $f_\alpha, f_\beta \in \bar{\mathcal{F}}_1$, it suffices to find an upper bound for $L_2^2(f_\alpha,f_\beta)$ when $\|\alpha-\beta\|_0 = 1$.  Using our monotonicity property again, observe that in that case, 
\begin{align*}
\frac{L_2^2(f_\alpha,f_\beta)}{2r^3} &\leq \int_{\sin(\pi/3 - \theta_1)}^{\sin(\pi/3 + \theta_1)} \{(1-t^2)^{1/2} - 2(1-\epsilon^2)^{1/2} + \sqrt{3}t\}^2 \, dt \\
&= 4\epsilon - \frac{4\epsilon^3}{3} - 4(1-\epsilon^2)^{1/2}\theta_1 \\
&\leq 4\epsilon - \frac{4\epsilon^3}{3} - 4\biggl(1-\frac{\epsilon^2}{2}-\frac{7\epsilon^4}{48}\biggr)\biggl(\epsilon + \frac{\epsilon^3}{6}+\frac{3\epsilon^5}{40}\biggr) \leq \epsilon^5.
\end{align*}
This shows that in condition (ii) of Lemma~\ref{Lemma:Kim}, we may take $C := nr^3\epsilon^5/(2c_0)$.  From Lemma~\ref{Lemma:Kim}, and using the fact that $\lfloor \frac{\pi}{6\theta_1}\rfloor \epsilon \geq \frac{\pi}{24\arcsin(1/2)} = 1/4$ for $\epsilon \leq 1/2$, we conclude that
\[
\inf_{\tilde{f}_n \in \tilde{\mathcal{F}}_n} R(\tilde{f}_n,\mathcal{F}_1) \geq \frac{K}{8}(1-C^{1/2})\gamma \geq \frac{1}{28000} n^{-4/5}.
\]

\emph{The case $d \geq 2$}:  We again apply Lemma~\ref{Lemma:Kim}, but as described in Section~\ref{Sec:LowerBounds} the construction of our finite subset $\bar{\mathcal{F}}_d$ of $\mathcal{F}_d$ is quite different, being based around uniform densities on perturbations of a Euclidean ball.  Let 
\[
\epsilon := \biggl\{\frac{\pi^{1/2}(d-1)^{1/2}}{6^{1/2}}\biggr\}^{1/(d-1)}\frac{1}{2}n^{-1/(d+1)} \leq \frac{1}{2},
\]
Letting $\mathcal{S}_1 := \bar{B}_d(0,1) \setminus B_d(0,1)$ denote the unit Euclidean sphere, we use the well-known fact, proved for convenience in Lemma~\ref{Lemma:PackingSet} in Section~\ref{Sec:ProofLowerBound}, that there exist $x_1,\ldots,x_N \in \mathcal{S}_1$, with $N = N_{2\epsilon} := \lceil \frac{(2\pi)^{1/2}(d-1)^{1/2}}{3^{1/2}2^{d-1}}\epsilon^{-(d-1)}\rceil$, such that $\|x_j - x_k\| > 2\epsilon$ for all $j \neq k$.   Since $N \geq 2$, we can set $K := \lfloor N/2 \rfloor \in \mathbb{N}$.  For $k=1,\ldots,K$ and $\ell \in \{0,1\}$, let $x_{k,\ell} := x_{\ell K + k}$, and define the halfspaces
\[
\mathcal{H}_{k,\ell}^- := \{x \in \mathbb{R}^d: (x_{k,\ell})^T x \leq (1-\epsilon^2/2)\}.
\]
We can now define $\bar{\mathcal{F}}_d := \{f_{\alpha}:\alpha = (\alpha_1,\ldots,\alpha_K)^T \in \{0,1\}^K\}$, where
\[
f_\alpha(x) := \frac{1}{c_{K,\epsilon}}\biggl[\mathbbm{1}_{\{x \in \cap_{k=1}^K \cap_{\ell=0}^1 \mathcal{H}_{k,\ell}^- \cap \bar{B}_d(0,1)\}} + \sum_{k=1}^K \{\alpha_k\mathbbm{1}_{\{x \in \bar{B}_d(0,1) \setminus \mathcal{H}_{k,0}^-\}} + (1-\alpha_k)\mathbbm{1}_{\{x \in \bar{B}_d(0,1) \setminus \mathcal{H}_{k,1}^-\}}\}\biggr],
\]
and
\begin{equation}
\label{Eq:cKe}
c_{K,\epsilon} := \frac{\pi^{d/2}}{\Gamma(1+d/2)} - \frac{K}{2}\frac{\pi^{(d-1)/2}}{\Gamma((d+1)/2)}\int_0^{\epsilon^2-\epsilon^4/4} t^{\frac{d+1}{2}-1}(1-t)^{-1/2} \, dt.
\end{equation}
Thus, each $f_\alpha$ is a uniform density on a closed, convex subset of $\mathbb{R}^d$, so $\bar{\mathcal{F}}_d \subseteq \mathcal{F}_d$.  It is convenient to note that 
\[
\frac{2(d+1)^{1/2}}{3^{1/2}\pi^{1/2}} \leq \frac{\Gamma(1+\frac{d}{2})}{\Gamma(\frac{d+1}{2})} \leq \frac{(d+1)^{1/2}}{2^{1/2}}
\]
for $d \geq 2$.  It follows that
\begin{equation}
\label{Eq:crKe}
\frac{\pi^{d/2}}{2\Gamma(1+d/2)} \leq \frac{\pi^{d/2}}{\Gamma(1+d/2)}\biggl\{1 - \frac{1}{\pi^{1/2}}\frac{K\epsilon^{d+1}}{(d+1)^{1/2}}\biggr\} \leq c_{K,\epsilon} \leq \frac{\pi^{d/2}}{\Gamma(1+d/2)}.
\end{equation}
Again, it remains to verify the conditions of Lemma~\ref{Lemma:Kim}.  First, if $\alpha, \beta \in \{0,1\}^K$, then
\begin{align*}
h^2(f_\alpha,f_\beta) &= \frac{\|\alpha-\beta\|_0}{c_{K,\epsilon}}\frac{\pi^{(d-1)/2}}{\Gamma((d+1)/2)}\int_0^{\epsilon^2-\epsilon^4/4} t^{\frac{d+1}{2}-1}(1-t)^{-1/2} \, dt \\
%&\geq \|\alpha-\beta\|_0 \frac{2}{(d+1)B((d+1)/2,1/2)} \epsilon^{d+1}(1-\epsilon^2/4)^{(d+1)/2} \\
&\geq \frac{4\|\alpha-\beta\|_0}{3^{1/2}\pi(d+1)^{1/2}}\epsilon^{d+1}(1-\epsilon^2/4)^{(d+1)/2} \geq \frac{4 \times 15^{(d+1)/2}\|\alpha-\beta\|_0}{3^{1/2} \times 16^{(d+1)/2}\pi(d+1)^{1/2}}\epsilon^{d+1}.
\end{align*}
For the squared Hellinger loss function, we may therefore take $\gamma:= \frac{4 \times 15^{(d+1)/2}}{3^{1/2} \times 16^{(d+1)/2}\pi(d+1)^{1/2}} \epsilon^{d+1}$ in condition (i) of Lemma~\ref{Lemma:Kim}.  On the other hand, if $\alpha = (\alpha_1,\ldots,\alpha_K)^T,\beta = (\beta_1,\ldots,\beta_K)^T \in \{0,1\}^K$ satisfy $\|\alpha - \beta\|_0 = 1$, then
\[
h^2(f_\alpha,f_\beta) = \frac{1}{c_{K,\epsilon}}\frac{\pi^{(d-1)/2}}{\Gamma((d+1)/2)}\int_0^{\epsilon^2-\epsilon^4/4} t^{\frac{d+1}{2}-1}(1-t)^{-1/2} \, dt \leq \frac{2}{(d+1)^{1/2}}\epsilon^{d+1}.
\]
This shows that we may take $C :=  \frac{2}{(d+1)^{1/2}}n\epsilon^{d+1}$ in condition (ii) of Lemma~\ref{Lemma:Kim}.  We conclude from Lemma~\ref{Lemma:Kim} that
\[
\inf_{\tilde{f}_n \in \tilde{\mathcal{F}}_n} R(\tilde{f}_n,\mathcal{F}_d) \geq \frac{K}{8}(1-C^{1/2})\gamma \\
%&\geq \frac{(2\pi)^{1/2}(d-1)^{1/2}}{3^{1/2}2^{d+4}}\biggl(1 -\frac{\pi^{1/4}(d-1)^{1/4}}{6^{1/4} \times 2^{(d-1)/2}(d+1)^{1/4}}\biggr) \frac{2 \times 15^{(d+1)/2}}{4 \times16^{(d+1)/2}\pi(d+1)^{1/2}}  n^{-2/(d+1)}.
\geq \frac{1}{500 \times 2^d}\Bigl(\frac{15}{16}\Bigr)^{(d+1)/2}n^{-2/(d+1)},
\]
as required.
%The result for the $L_2^2$ loss function follows from the fact that for every $\alpha, \beta \in \{0,1\}^K$,
%\[
%L_2^2(f_\alpha,f_\beta) = \frac{1}{c_{r,K,\epsilon}}h^2(f_\alpha,f_\beta) \geq \frac{\Gamma(1+d/2)}{\pi^{d/2}}h^2(f_\alpha,f_\beta).
%\]
\end{proof}

\subsection{Proofs from Section~\ref{Sec:ConvInt}}

\begin{proof}[Proof of Proposition~\ref{Prop:Conv}]
Suppose that $\dim\bigl(\mathrm{csupp}(\nu)\bigr) = d$.  We first show that $\mathrm{csupp}(\nu) \subseteq \bar{C}$.  Suppose that $x_0 \notin \bar{C}$, so there exists $\delta > 0$ such that $B_d(x_0,\delta) \subseteq C^c$.  If $x^* \in B_d(x_0,\delta)$, then there exists a subsequence $(f_{n_k})$ with $f_{n_k}(x^*) < 1/k$ for each $k \in \mathbb{N}$.  Then $\{x \in \mathbb{R}^d: f_{n_k}(x) \geq 1/k\}$ is a closed, convex set not containing $x^*$, so there exist $b_k \in \mathbb{R}^d$ with $\|b_k\| = 1$ such that $\{x \in \mathbb{R}^d: b_k^T x \leq b_k^T x^*\} \subseteq \{x \in \mathbb{R}^d: f_{n_k}(x) < 1/k\}$.  We can find a subsequence $(b_{k(l)})$, as well as $b_{x^*} \in \mathbb{R}^d$ with $\|b_{x^*}\| = 1$, such that $b_{k(l)} \rightarrow b_{x^*}$.  For any $R \in \mathbb{N}$ and $\eta > 0$, let $A_{R,\eta} := \{x: b_{x^*}^Tx < b_{x^*}^T x^* - \eta, \|x\| < R\}$.  Let $l_0 \in \mathbb{N}$ be large enough that $\|b_{k(l)} - b_{x^*}\| \leq \eta/(2R)$ for $l \geq l_0$.  Then we have for $l \geq l_0$, $R > \|x^*\|$ and $x \in A_{R,\eta}$  that
\[
b_{k(l)}^T(x - x^*) = b_{x^*}^T(x-x^*) + (b_{k(l)} - b_{x^*})^T(x-x^*) < -\eta + \frac{\eta}{2R}(\|x\| + \|x^*\|) < 0.
\]
Hence for $R > \|x^*\|$, we have $f_{n_{k(l)}}(x) < 1/k(l)$ for all $x \in A_{R,\eta}$ and $l \geq l_0$.  Since $A_{R,\eta}$ is open, we have for all $R > \|x^*\|$ and $\eta > 0$ that 
\[
\nu(A_{R,\eta}) \leq \liminf_{l \rightarrow \infty} \nu_{n_{k(l)}}(A_{R,\eta}) = \liminf_{l \rightarrow \infty} \int_{A_{R,\eta}} f_{n_{k(l)}} \, d\mu_d \leq \liminf_{l \rightarrow \infty} \frac{\mu_d(A_{R,\eta})}{k(l)} = 0. 
\]
Since the sets $A_{R,\eta}$ are increasing in $R$, we deduce that $\nu(A_{R,\eta}) = 0$ for all $R \in \mathbb{N}$ and all $\eta > 0$, so 
\[
\nu(\{x: b_{x^*}^Tx < b_{x^*}^T x^*\}) = \nu\biggl(\bigcup_{R=1}^\infty A_{R,1/R}\biggr) = \lim_{R \rightarrow \infty} \nu(A_{R,1/R}) = 0.
\]
This shows that no $x^* \in  B_d(x_0,\delta)$ belongs to $\mathrm{int}\bigl(\mathrm{csupp}(\nu)\bigr)$, so $x_0 \notin \mathrm{csupp}(\nu)$.  We conclude that if $\dim\bigl(\mathrm{csupp}(\nu)\bigr) = d$, then $\mathrm{csupp}(\nu) \subseteq \bar{C}$.

Now suppose that $\dim(C) = d$.  To show that $\bar{C} \subseteq \mathrm{csupp}(\nu)$, it suffices (since $\mathrm{csupp}(\nu)$ is closed) to prove that $C \subseteq \mathrm{csupp}(\nu)$.  Suppose, for a contradiction, that $x_0 \in C \setminus \mathrm{csupp}(\nu)$.  Then there exists $\delta > 0$ such that $B_d(x_0,\delta) \cap \mathrm{csupp}(\nu) = \emptyset$.  Since $\dim(C) = d$, we can find $\epsilon > 0$, $n_0 \in \mathbb{N}$ and $x_1,\ldots,x_d \in B_d(x_0,\delta)$ such that $x_0, x_1,\ldots, x_d$ are affinely independent, and $f_n(x_j) \geq \epsilon$ for $j=0,1,\ldots,d$ and $n \geq n_0$.  We deduce that for $n \geq n_0$,  we have $f_n(x) \geq \epsilon$ for $x \in \mathrm{conv}(\{x_0,x_1,\ldots,x_d\})$.  But then
\begin{align*}
\nu\bigl(\mathrm{conv}(\{x_0,x_1,\ldots,x_d\})\bigr) &\geq \limsup_{n \rightarrow \infty} \nu_n\bigl(\mathrm{conv}(\{x_0,x_1,\ldots,x_d\})\bigr) \\
&\geq \liminf_{n \rightarrow \infty} \nu_n\bigl(\mathrm{conv}(\{x_0,x_1,\ldots,x_d\})\bigr) \\
&\geq \epsilon \mu_d\bigl(\mathrm{conv}(\{x_0,x_1,\ldots,x_d\})\bigr) > 0.
\end{align*}
This contradicts $B_d(x_0,\delta) \cap \mathrm{csupp}(\nu) = \emptyset$, and we conclude that if $\dim(C) = d$, then $\bar{C} \subseteq \mathrm{csupp}(\nu)$.

Thus, if $\dim\bigl(\mathrm{csupp}(\nu)\bigr) = d$, then $\mathrm{csupp}(\nu) \subseteq \bar{C}$, so $\dim(C) = d$, so $\bar{C} \subseteq \mathrm{csupp}(\nu)$, and it follows that $\mathrm{csupp}(\nu) = \bar{C}$.  Moreover, we can reach the same conclusion starting from the hypothesis that $\dim(C) = d$.

Now suppose that $\dim(C) = d$.  To show that $\nu$ is absolutely continuous with respect to $\mu_d$, for $t \in \mathbb{R}$, let $U_{n,t} := \{x \in \mathbb{R}^d:\log f_n(x) \geq t\}$.  We can find $\epsilon \in (0,1)$ and $n_0 \in \mathbb{N}$ such that $\mu_d(U_{n,\log \epsilon}) \geq \epsilon$, for all $n \geq n_0$.  We first want to deduce that $\sup_{x \in \mathbb{R}^d} \sup_{n \in \mathbb{N}} f_n(x) < \infty$.  To this end, let $M_n := \sup_{x \in \mathbb{R}^d} \log f_n(x)$, and suppose, without loss of generality since $f_n$ is upper semi-continuous, that $\log f_n(x_{0,n}) = M_n$.  Assume for now that $M_n \geq \max\{\log(1/\epsilon),4d^2\}$, so for $x \in U_{n,\log \epsilon}$, we have
\[
\log f_n\biggl(x_{0,n} + \frac{x - x_{0,n}}{M_n - \log \epsilon}\biggr) \geq \biggl(\frac{1}{M_n - \log \epsilon}\biggr)\log \epsilon + \biggl(\frac{M_n-1 - \log \epsilon}{M_n - \log \epsilon}\biggr) M_n = M_n - 1.
\]
Thus $\mu_d(U_{n,\log \epsilon}) \leq (M_n - \log \epsilon)^d\mu_d(U_{n,M_n-1}) \leq (2M_n)^d\mu_d(U_{n,M_n-1})$.  But 
\[
1 = \int_{\mathbb{R}^d} f_n \geq e^{M_n - 1}\mu_d(U_{n,M_n-1}),
\]
so 
\[
\epsilon \leq \mu_d(U_{n,\log \epsilon}) \leq (2M_n)^d e^{-(M_n - 1)} \leq e^{-\bigl(\frac{M_n}{2} - 1\bigr)}.
\]
We deduce that $M_n \leq 2 + 2\log(1/\epsilon)$.  Thus, removing the initial assumption on $M_n$, we find that $M_n \leq \max\bigl\{2 + 2\log(1/\epsilon),4d^2\bigr\} =:M$, say.  Now, given $\eta > 0$, choose $\delta = \frac{\eta}{2e^M}$.  If $A$ is a Borel subset of $\mathbb{R}^d$ with $\mu_d(A) \leq \delta$, then since $\mu_d$ is regular, we can find an open set $A' \supseteq A$ in $\mathbb{R}^d$ with $\mu_d(A') \leq 2\delta$.  But then
\[
\nu(A) \leq \nu(A') \leq \liminf_{n \rightarrow \infty} \nu_n(A') = \liminf_{n \rightarrow \infty} \int_{A'} f_n \, d\mu_d \leq 2\delta e^M = \eta.
\]
It follows that $\nu$ is absolutely continuous with respect to $\mu_d$, so by the Radon--Nikodym theorem, we can let $f$ denote the Radon--Nikodym derivative of $\nu$ with respect to $\mu_d$.  The fact that $f = \mathrm{cl}(\liminf f_n)$ then follows from the proof of Proposition~2(a) of \citet{CuleSamworth2010}.

%For the final part of the proposition, let $S$ be a compact subset of $\mathbb{R}^d$ not intersecting $\bar{C}$, so we claim that $\sup_{x \in S} f_n(x) \rightarrow 0$.  Indeed, suppose for a contradiction that there exist $\epsilon > 0$, a subsequence $(f_{n_k})$ and a sequence $(x_k) \in S$ with $f_{n_k}(x_k) \geq \epsilon$.  Since $S$ is compact, there exists a subsequence $(x_{k(l)})$ and $x_0 \in S$ such that $x_{k(l)} \rightarrow x_0$.  Moreover, we can find points $x_1^*,\ldots,x_d^* \in C$ such that $x_0,x_1^*,\ldots,x_d^*$ are affinely independent, and by reducing $\epsilon > 0$ if necessary, we may assume $f_{n_k}(x_j^*) \geq \epsilon$ for $j=1,\ldots,d$ and large $k$.  We can find $b \in \mathbb{R}^d$ and $\beta \in \mathbb{R}$ such that $\bar{C} \subseteq \{x:b^Tx < \beta\}$ but $b^Tx_0 > \beta$.  Letting $c = \mathrm{conv}(\{x_0,x_1^*,\ldots,x_d^*\}) \cap \{x:b^Tx \geq \beta\}$, so that $c > 0$, we can find a closed set $\bar{B} \subseteq \mathrm{conv}(\{x_0,x_1^*,\ldots,x_d^*\}) \cap \{x:b^Tx \geq \beta\}$ such that $\bar{B} \subseteq \mathrm{conv}(\{x_{k(l)},x_1^*,\ldots,x_d^*\})$ for large $l$, and $\mu_d(\bar{B}) \geq c/2$.  But then 
%\[
%\nu(\bar{B}) \geq \limsup_{l \rightarrow \infty} \nu_{n_{k(l)}}(\bar{B}) \geq \epsilon c/2 > 0.
%\]
%This contradicts $\mathrm{csupp}(\nu) = \bar{C} \subseteq \{x:b^Tx < \beta\}$ and therefore proves our claim.    
\end{proof}

\begin{proof}[Proof of Proposition~\ref{Prop:Conv2}]
1. Now suppose that $\dim(C) = d-1$, so $\dim\bigl(\mathrm{csupp}(\nu)\bigr) \leq d-1$.  %The proof of the first statement is similar to the proof of the final part of Proposition~\ref{Prop:Conv}.  
Let $S$ be a compact subset of $\mathbb{R}^d$ not intersecting $\mathrm{aff}(C)$, and suppose for a contradiction that there exist $\epsilon > 0$, a subsequence $(f_{n_k})$ and a sequence $(x_k) \in S$ with $f_{n_k}(x_k) \geq \epsilon$.  Since $S$ is compact, there exists a subsequence $(x_{k(l)})$ and $x_0 \in S$ such that $x_{k(l)} \rightarrow x_0$.  Moreover, we can find affinely independent points $x_1^*,\ldots,x_d^* \in C$, and by reducing $\epsilon > 0$ if necessary, we may assume $f_{n_k}(x_j^*) \geq \epsilon$ for $j=1,\ldots,d$ and large $k$.  Let $c := \mu_d\bigl(\mathrm{conv}(\{x_0,x_1^*,\ldots,x_d^*\})\bigr)$, so $c > 0$.  Let $b \in \mathbb{R}^d$ and $\beta \in \mathbb{R}$ be such that $\mathrm{csupp}(\nu) \subseteq \{x:b^Tx = \beta\}$, so without loss of generality, we may assume $\mu_d\bigl(\mathrm{conv}(\{x_0,x_1^*,\ldots,x_d^*\}) \cap \{x:b^Tx < \beta\}\bigr) \geq c/2$.  It follows that we can find a closed set $\bar{B} \subseteq \mathrm{conv}(\{x_0,x_1^*,\ldots,x_d^*\}) \cap \{x:b^Tx < \beta\}$ such that $\bar{B} \subseteq \mathrm{conv}(\{x_{k(l)},x_1^*,\ldots,x_d^*\})$ for large $l$, and $\mu_d(\bar{B}) \geq c/4$.  But then
\[
\nu(\bar{B}) \geq \limsup_{l \rightarrow \infty} \nu_{n_{k(l)}}(\bar{B}) \geq \epsilon c/4 > 0,
\]
contradicting $\bar{B} \cap \mathrm{csupp}(\nu) = \emptyset$.  We deduce that $\sup_{x \in S} f_n(x) \rightarrow 0$ as $n \rightarrow \infty$.

We now wish to deduce that if $\dim(C) = d-1$, then $\mathrm{csupp}(\nu) \subseteq \mathrm{aff}(C)$.  Suppose for a contradiction that $x_0 \in \mathrm{csupp}(\nu) \setminus \mathrm{aff}(C)$.  Let $H$ be a closed halfspace with $x_0 \in \mathrm{int}(H)$ but $H \cap \mathrm{aff}(C) = \emptyset$, and let $H_R = H \cap \bar{B}_d(0,R)$.  Then by the argument in the previous paragraph, given $\epsilon > 0$, there exists $n_0 \in \mathbb{N}$ such that $f_n(x) \leq \epsilon$ for all $x \in H_R$ and $n \geq n_0$.  It follows that 
\[
\nu\bigl(\mathrm{int}(H_R)\bigr) \leq \liminf_{n \rightarrow \infty} \nu_n\bigl(\mathrm{int}(H_R)\bigr) \leq \epsilon \mu_d\bigl(\mathrm{int}(H_R)\bigr),
\]
so $\nu\bigl(\mathrm{int}(H_R)\bigr) = 0$.  We deduce that $\nu\bigl(\mathrm{int}(H)\bigr) = \lim_{R \rightarrow \infty} \nu\bigl(\mathrm{int}(H_R)\bigr) = 0$, contradicting the hypothesis that $x_0 \in \mathrm{csupp}(\nu)$.  Thus $\mathrm{csupp}(\nu) \subseteq \mathrm{aff}(C)$.

2. Note that $f_{n,U} \in \mathcal{F}_{k,U+a}$, by Theorem~6 of \citet{Prekopa1973}.  If $\nu_{n,U}$ denotes the probability measure corresponding to $f_{n,U}$, then by the Cram\'er--Wold device, $\nu_{n,U} \stackrel{d}{\rightarrow} \nu$.  It follows by Proposition~\ref{Prop:Conv} that $\nu$ is absolutely continuous with respect to $\mu_{k,U+a}$, with Radon--Nikodym derivative $\mathrm{cl}(\liminf f_{n,U}) \in \mathcal{F}_{k,U+a}$.
\end{proof}

\begin{proof}[Proof of Proposition~\ref{Prop:Conv3}]
If $\theta_0 \in U$ and $\theta_1 \in U^\perp$, then 
\[
\int_{\mathbb{R}^d} e^{\theta_0^Tx + \theta_1^T x} \, d\nu(x) = \int_{U+a} e^{\theta_0^Tx + \theta_1^T x} \, d\nu(x) = e^{\theta_1^T a} \int_{U+a} e^{\theta_0^Tx} \, d\nu(x),
\]
so $\Theta = \Theta_0 \oplus U^\perp$, where $\Theta_0$ contains 0.  The fact that $\Theta_0$ is convex follows immediately from the convexity of the exponential function, while the fact that $\Theta_0$ is relatively open follows from the proof of Proposition~2.2 of \citet{SHD2011}, once we note from Part 2 of Proposition~\ref{Prop:Conv2} that $\nu$ has a log-concave Radon--Nikodym derivative with respect to $\mu_{k,U+a}$.

Now fix $\theta \in \Theta$, and let $X_n \sim \nu_n$ and $X \sim \nu$.  By Theorem~6 of \citet{Prekopa1973}, $\theta^T X_n$ has a log-concave density, and by the Cram\'er--Wold device, $\theta^T X_n \stackrel{d}{\rightarrow} \theta^T X$.  Letting $\nu_\theta$ denote the distribution of $\theta^T X$, we consider separately the cases $\dim\bigl(\mathrm{csupp}(\nu_\theta)\bigr) = 1$ and $\dim\bigl(\mathrm{csupp}(\nu_\theta)\bigr) = 0$.  If $\dim\bigl(\mathrm{csupp}(\nu_\theta)\bigr) = 1$, then by Proposition~\ref{Prop:Conv}, $\nu_\theta$ admits an upper semi-continuous, log-concave Radon--Nikodym derivative $f_\theta$, say, with respect to $\mu_1$, and
\[
\int_{-\infty}^\infty e^t f_\theta(t) \, dt = \int_{\mathbb{R}^d} e^{\theta^T x} \, d\nu(x) < \infty.
\]
Letting $f_{n,\theta}(t) := \mathrm{cl}\bigl(\int_{x:\theta^T x = t} f_n(x) \, dx\bigr)$, and noting that $f_{n,\theta} \in \mathcal{F}_1$, we deduce that
\[
\biggl|\int_{\mathbb{R}^d} e^{\theta^T x} \, d\nu_n(x) - \int_{\mathbb{R}^d} e^{\theta^T x} \, d\nu(x)\biggr| \leq \int_{-\infty}^\infty e^t |f_{n,\theta}(t) - f_\theta(t)| \, dt \rightarrow 0,
\]
where the convergence follows from Proposition~2.2 and Theorem~2.1 of \citet{SHD2011}.

Finally, suppose that $\dim\bigl(\mathrm{csupp}(\nu_\theta)\bigr) = 0$, so that $\theta \in U^\perp$, and $\theta^TX_n \stackrel{d}{\rightarrow} \delta_a$, where $\delta_a$ denotes a Dirac point mass at $a$.  Letting $f_{n,\theta}(t) =  \mathrm{cl}\bigl(\int_{x:\theta^T x = t} f_n(x) \, dx\bigr)$ as before, we note that given $\epsilon \in \bigl(0,\frac{\log 2}{20}\bigr)$, we can find $n_0 \in \mathbb{N}$ such that $\int_{a - \epsilon}^{a + \epsilon} f_{n,\theta}(t) \, dt \geq 1/2$ for all $n \geq n_0$.  In particular, for $n \geq n_0$, there exists $t_n \in (a - \epsilon,a+\epsilon)$ such that $f_{n,\theta}(t_n) \geq 1/(4\epsilon)$.  We may also assume that for each $n \geq n_0$ there exists $t_{1,n} \in [a + \epsilon,a + 9\epsilon]$ such that $f_{n,\theta}(t_{1,n}) \leq 1/(8\epsilon)$.  We deduce that for $n \geq n_0$ and $t \geq t_{1,n}$, 
%$t_{1,n} \in [a - 9\epsilon,a - \epsilon]$ and $t_{2,n} \in [a + \epsilon,a + 9\epsilon]$ such that $f_{n,\theta}(t_{1,n}) \leq 1/(8\epsilon)$ and $f_{n,\theta}(t_{2,n}) \leq 1/(8\epsilon)$.  We deduce that for $n \geq n_0$, 
%\[
%f_{n,\theta}(t) \leq \left\{ \begin{array}{ll} \exp\Bigl\{\Bigl(\frac{t_n - t}{t_n - t_{1,n}}\Bigr)\log \frac{1}{8\epsilon} + \Bigl(\frac{t - t_{1,n}}{t_n - t_{1,n}}\Bigr)\log \frac{1}{4\epsilon}\Bigr\} & \mbox{if $t \leq t_{1,n}$} \\
%\exp\Bigl\{\Bigl(\frac{t - t_n}{t_{2,n} - t_n}\Bigr)\log \frac{1}{8\epsilon} + \Bigl(\frac{t_{2,n} - t}{t_{2,n} - t_n}\Bigr)\log \frac{1}{4\epsilon}\Bigr\} & \mbox{if $t \geq t_{2,n}$.} \end{array} \right.
\[
f_{n,\theta}(t) \leq \exp\biggl\{\Bigl(\frac{t - t_n}{t_{1,n} - t_n}\Bigr)\log \frac{1}{8\epsilon} + \Bigl(\frac{t_{1,n} - t}{t_{1,n} - t_n}\Bigr)\log \frac{1}{4\epsilon}\biggr\}
\]
It follows that for $K \geq \max\{2(a+\epsilon),a+9\epsilon\}$, 
\begin{align*}
%\sup_{n \geq n_0} &\int_{x:|\theta^T x| \geq K} e^{\theta^T x} f_n(x) \, dx = \sup_{n \geq n_0} \biggl\{\int_{-\infty}^{-K} e^t f_{n,\theta}(t) \, dt + \int_K^\infty e^t f_{n,\theta}(t) \, dt\biggr\} \\
%&\leq \sup_{n \geq n_0} \int_{-\infty}^{-K} \exp\biggl\{\biggl(\frac{t_n - t}{t_n - t_{1,n}}\biggr)\log \frac{1}{8\epsilon} + \biggl(\frac{t - t_{1,n}}{t_n - t_{1,n}}\biggr)\log \frac{1}{4\epsilon}\biggr\} \, dt \\
%&\hspace{1cm}+ \sup_{n \geq n_0} \int_K^\infty e^t\exp\biggl\{\biggl(\frac{t - t_n}{t_{2,n} - t_n}\biggr)\log \frac{1}{8\epsilon} + \biggl(\frac{t_{2,n} - t}{t_{2,n} - t_n}\biggr)\log \frac{1}{4\epsilon}\biggr\} \, dt \\
%&= \sup_{n \geq n_0} \frac{(t_n - t_{1,n})}{\log 2}\exp\biggl\{\biggl(\frac{t_n + K}{t_n - t_{1,n}}\biggr)\log \frac{1}{8\epsilon} - \biggl(\frac{K + t_{1,n}}{t_n - t_{1,n}}\biggr)\log \frac{1}{4\epsilon}\biggr\} \\
%&\hspace{1cm}+\sup_{n \geq n_0} \frac{(t_{2,n} - t_n)e^K}{\log 2 - (t_{2,n} - t_n)}\exp\biggl\{\biggl(\frac{K - t_n}{t_{2,n} - t_n}\biggr)\log \frac{1}{8\epsilon} + \biggl(\frac{t_{2,n} - K}{t_{2,n} - t_n}\biggr)\log \frac{1}{4\epsilon}\biggr\} \\
%&\leq \frac{5}{2\log 2} 2^{\frac{-K}{10\epsilon}} + \frac{5}{2(\log 2 - 10\epsilon)}e^{-K (\frac{\log 2}{10\epsilon} - 1)} \rightarrow 0 
\sup_{n \geq n_0} \int_{x:\theta^T x \geq K} &e^{\theta^T x} f_n(x) \, dx = \sup_{n \geq n_0} \int_K^\infty e^t f_{n,\theta}(t) \, dt \\
&\leq \sup_{n \geq n_0} \int_K^\infty e^t\exp\biggl\{\biggl(\frac{t - t_n}{t_{1,n} - t_n}\biggr)\log \frac{1}{8\epsilon} + \biggl(\frac{t_{1,n} - t}{t_{1,n} - t_n}\biggr)\log \frac{1}{4\epsilon}\biggr\} \, dt \\
&=\sup_{n \geq n_0} \frac{(t_{1,n} - t_n)e^K}{\log 2 - (t_{1,n} - t_n)}\exp\biggl\{\biggl(\frac{K - t_n}{t_{1,n} - t_n}\biggr)\log \frac{1}{8\epsilon} + \biggl(\frac{t_{1,n} - K}{t_{1,n} - t_n}\biggr)\log \frac{1}{4\epsilon}\biggr\} \\
&\leq \frac{5}{2(\log 2 - 10\epsilon)}e^{-K (\frac{\log 2}{20\epsilon} - 1)} \rightarrow 0 
\end{align*}
as $K \rightarrow \infty$.  We deduce that the sequence $(e^{\theta^T X_n})$ is uniformly integrable, so the result follows by Theorem~A on p.14 of \citet{Serfling1980}.
\end{proof}

\begin{proof}[Proof of Theorem~\ref{Thm:IntEnv}]
(a) Suppose for a contradiction that there exist sequences $(f_n) \in \mathcal{F}_d^{0,I}$ and $(a_n) \searrow 0$ such that $\sup_{x \in \mathbb{R}^d} e^{a_n\|x\|} f_n(x) \geq n$ for all $n \in \mathbb{N}$.  Note that for $R > 0$,
\[
\sup_{n \in \mathbb{N}} \int_{\|x\| > R} f_n(x) \, dx \leq \sup_{n \in \mathbb{N}} \frac{1}{R^2} \int_{\|x\| > R} \|x\|^2 f_n(x) \, dx \leq \frac{d}{R^2} \rightarrow 0
\]
as $R \rightarrow \infty$.  We conclude that the sequence of probability measures $(\nu_n)$ defined by $(f_n)$ is tight, so by Prohorov's theorem, we can find $1 \leq n_1 \leq n_2 \leq \ldots$ and a probability measure $\nu$ on $\mathbb{R}^d$ such that $\nu_{n_k} \stackrel{d}{\rightarrow} \nu$.  If $\Sigma$ denotes the covariance matrix corresponding to $\nu$, then by the remark following Proposition~\ref{Prop:Conv3}, we have $\Sigma = I$.  In particular, $\dim\bigl(\mathrm{csupp}(\nu)\bigr) = d$.  It follows by Proposition~\ref{Prop:Conv} that $\nu$ has a log-concave Radon--Nikodym derivative $f := \mathrm{cl}(\liminf f_{n_k})$ with respect to $\mu_d$.  Pick $x_0 \in \mathrm{int}(\mathrm{dom}(f))$ and $\delta \in (0,1)$ such that $\bar{B}_d(x_0,\delta) \subseteq \mathrm{int}(\mathrm{dom}(f))$.  Since $f_{n_k} \rightarrow f$ uniformly on compact subsets of $\mathrm{int}(\mathrm{dom}(f))$, there exists $k_0 \in \mathbb{N}$ such that $|f_{n_k}(x) - f(x)| < f(x_0)/4$ for all $k \geq k_0$ and all $x \in \bar{B}_d(x_0,\delta)$.  Moreover, by reducing $\delta >0$ if necessary, we may assume that $|f(x) - f(x_0)| < f(x_0)/4$ for all $x \in \bar{B}_d(x_0,\delta)$.  In particular, this means that $f_{n_k}(x) \geq f(x_0)/2$ for all $k \geq k_0$ and all $x \in \bar{B}_d(x_0,\delta)$.

We now claim that there exists $R_0 > 2(\|x_0\|+1)$ such that $f_{n_k}(x) < f(x_0)/4$ for $\|x\| \geq R_0$ and $k \geq k_0$.  To see this, suppose for a contradiction that there exist an $\mathbb{R}^d$-valued sequence $(x_m)$ with $\|x_m\| \rightarrow \infty$ and a sequence of positive integers $(k_m)$ with $k_m \geq k_0$ such that 
\[
f_{n_{k(m)}}(x_m) \geq \frac{f(x_0)}{4}
\]
for all $m$.  Then, since the level sets of each $f_n$ are convex, for each $m$,
\[
\mu_d\bigl(\{x:f_{n_{k(m)}}(x) \geq f(x_0)/4\}\bigr) \geq \mu_d\Bigl(\mathrm{conv}\bigl(\bar{B}_d(x_0,\delta) \cup \{x_m\}\bigr)\Bigr) \rightarrow \infty
\]
as $m \rightarrow \infty$.  This contradicts the fact that each $f_n$ is a density, and establishes our claim. 

But now, if $k \geq k_0$ and $x \in \bar{B}_d(0,R_0) \setminus \bar{B}_d(x_0,\delta)$, then we can set   
\[
x_{1,k} = \biggl(\frac{\|x - x_0\|-\delta/2}{\|x - x_0\|}\biggr)x_0 + \biggl(\frac{\delta/2}{\|x - x_0\|}\biggr)x. 
\]
Observe that $\|x_{1,k} - x_0\| = \delta/2$.  Thus, for all $k \geq k_0$,
\begin{align*}
\log f_{n_k}(x) &\leq \biggl(\frac{2\|x - x_0\|}{\delta}\biggr)\bigl\{\log f_{n_k}(x_{1,k}) - \log f_{n_k}(x_0)\bigr\} + \log f_{n_k}(x_0) \\
&\leq \frac{4R_0}{\delta}\log 2 + \log \Bigl(\frac{5f(x_0)}{4}\Bigr).
\end{align*}
%Moreover, if $R_0 < \|x\| \leq 2R_0$, then $f_{n_k}(x) \leq f(x_0)/4$ for $k \geq k_0$.  Finally, for $\|x\| > 2R_0$, we can find $x_{2,k} \in \bar{B}_d(0,R_0) \setminus B_d(0,R_0)$ such that $\|x - x_{2,k}\| < \|x - x_0\|$ and
%\[
%x_{2,k} = \biggl(\frac{\|x - x_{2,k}\|}{\|x - x_0\|}\biggr)x_0 + \biggl(\frac{\|x_{2,k} - x_0\|}{\|x - x_0\|}\biggr)x.
%\]
%We deduce that for $k \geq k_0$,
%\begin{align*}
%\log &f_{n_k}(x) \leq \biggl(\frac{\|x - x_0\|}{\|x_{2,k} - x_0\|}\biggr)\biggl\{\log f_{n_k}(x_{2,k}) - \biggl(\frac{\|x - x_{2,k}\|}{\|x - x_0\|}\biggr)\log f_{n_k}(x_0)\biggr\} \\
%&= \biggl(\frac{\|x - x_0\|}{\|x_{2,k} - x_0\|}\biggr)\{\log f_{n_k}(x_{2,k}) - \log f_{n_k}(x_0)\} + \biggl(\frac{\|x - x_0\| - \|x - x_{2,k}\|}{\|x_{2,k} - x_0\|}\biggr)\log f_{n_k}(x_0) \\
%&\leq -\biggl(\frac{\|x\| - R_0}{2R_0}\biggr)\log 2 + \max\biggl\{3\log \Bigl(\frac{3f(x_0)}{2}\Bigr),0\biggr\}.
%\end{align*}
Now, for $\|x\| > R_0$, we can find $x_{2,k} \in \bar{B}_d(0,R_0) \setminus B_d(0,R_0)$ and $\lambda \in (0,1)$ such that $x_{2,k} = \lambda x_0 + (1-\lambda)x$.  Notice that 
\[
R_0 = \|x_{2,k}\| \geq (1-\lambda)\|x\| - \lambda \|x_0\| \geq (1-\lambda)\|x\| - \lambda \frac{R_0}{2},
\] 
so $\lambda \geq 2(\|x\| - R_0)/(2\|x\| + R_0)$.  It follows that for $k \geq k_0$, 
\begin{align*}
\log f_{n_k}(x) &\leq \frac{1}{1-\lambda}\{\log f_{n_k}(x_{2,k}) - \log f_{n_k}(x_0)\} + \log f_{n_k}(x_0) \\
&\leq - \biggl(\frac{2\|x\| + R_0}{3R_0}\biggr)\log 3 + \log \Bigl(\frac{5f(x_0)}{4}\Bigr).
\end{align*}
We conclude that there exist $A_{0,d}, B_{0,d} > 0$ such that $f_{n_k}(x) \leq e^{-A_{0,d}\|x\| + B_{0,d}}$ for all $k \geq k_0$ and all $x \in \mathbb{R}^d$, contradicting our original hypothesis, and therefore proving our claim.

(b) Suppose for a contradiction that there exists $x_0 \in \mathbb{R}^d$ with $\|x_0\| \leq 1/4$ and a sequence $(f_n) \in \mathcal{F}_d^{0,I}$ such that $f_n(x_0) \searrow 0$ as $n \rightarrow \infty$.  As in the proof of part (a), the sequence $(\nu_n)$ of corresponding probability measures is tight, so by Prohorov's theorem, there exists a subsequence $(\nu_{n_k})$ and a probability measure $\nu$ on $\mathbb{R}^d$ such that $\nu_{n_k} \stackrel{d}{\rightarrow} \nu$.  The upper semi-continuous version of the probability density $f$ corresponding to $\nu$ belongs to $\mathcal{F}_d^{0,I}$, so letting $C = \mathrm{dom}(\log f)$, we have $0 \in \mathrm{int}(C)$.  Note further that since $(f_{n_k})$ converges to $f$ pointwise on $\mathrm{int}(C)$, we must have that $x_0 \notin \mathrm{int}(C)$ and $cx_0 \in \mathrm{bd}(C)$ for some $c \in (0,1]$.  Now let  
\[
x_* \in \argmin_{x \in \mathrm{bd}(C)} \|x\|,
\]
so $0 < \|x_*\| \leq \|x_0\|$.  Without loss of generality, we may assume $x_* = (\|x_*\|,0,\ldots,0)^T$.  By the supporting hyperplane theorem \citep[][Theorem~11.6]{Rockafellar1997}, there exists $b = (b_1,\ldots,b_d)^T \in \mathbb{R}^d$ with $\|b\|=1$ such that $C \subseteq \{x:b^Tx \leq b^Tx_*\}$.  If $b \neq e_1$, where $e_1$ denotes the first standard basis vector in $\mathbb{R}^d$, then $b^Tx_* < \|x_*\|$ and  there exists $c \in (0,1)$ such that $x_{**} := c\|x_*\|b \in \mathrm{bd}(C)$.  But then $\|x_{**}\| < \|x_*\|$, a contradiction, so $b = e_1$, and $x_1 \leq \|x_*\|$ for all $x = (x_1,\ldots,x_d)^T \in C$.  Letting $f_1^*(x_1) := \mathrm{cl}\bigl(\int_{\mathbb{R}^{d-1}} f(x_1,\ldots,x_d) \, dx_2\ldots dx_d\bigr)$, we then have that $f_1^* \in \mathcal{F}_1^{0,1}$ and $f_1^*(x_1) = 0$ for all $x_1 > \|x_*\|$.

Our claim is that this forces $\|x_*\| > 1/4$.  To see this, let $a := \|x_*\|$, let $m \in [0,a]$ be such that $f_1^*(m) = \max_{x_1 \in [0,a]} f_1^*(x_1) =: M$ and let $\phi_1^* := \log f_1^*$.  Note that  
\[
\frac{Ma^2}{2} \geq \int_0^a uf_1^*(u) \, du \geq \biggl|\int_{-2a}^0 u f_1^*(u) \, du\biggr| \geq 2a^2\inf_{u \in [-2a,0]} f_1^*(u).
\]
Hence $\inf_{u \in [-2a,0]} f_1^*(u) \leq M/4$, and in fact this infimum must be attained when $u = -2a$, so $f_1^*(-2a) \leq M/4$.  Now observe that
\begin{align}
\label{Eq:density}
1 \geq \int_{-2a}^m f_1^* \geq \int_{-2a}^m \exp\biggl\{\frac{u+2a}{m+2a} \phi_1^*(m) + \frac{m-u}{m+2a} \phi_1^*(-2a)\biggr\} \, du &= \frac{(m+2a)\{M-f_1^*(-2a)\}}{\log M - \phi_1^*(-2a)} \nonumber \\
&\geq \frac{3af_1^*(-2a)}{\log 2}.
\end{align}
On the other hand,
\begin{align*}
\int_{-\infty}^{-2a} u^2 &f_1^*(u) \, du \leq \int_{-\infty}^{-2a} u^2 \exp\biggl\{\frac{u+2a}{m+2a} \phi_1^*(m) + \frac{m-u}{m+2a} \phi_1^*(-2a)\biggr\} \, du \\
&= \frac{(m+2a)f_1^*(-2a)}{\log M - \phi_1^*(-2a)}\biggl[\frac{2(m+2a)^2}{\{\log M - \phi_1^*(-2a)\}^2} + \frac{4a(m+2a)}{\log M - \phi_1^*(-2a)} + 4a^2\biggr] < 12a^2.
\end{align*}
Here, we used~(\ref{Eq:density}), as well as $m \leq a$ and $\log M - \phi_1^*(-2a) \geq 2 \log 2$  to obtain the final inequality.  We deduce that
\[
1 = \int_{-\infty}^a u^2 f_1^*(u) \, du < 16a^2,
\]
so $a > 1/4$, as required.
\end{proof}
\begin{proof}[Proof of Corollary~\ref{Cor:Transform}]
(a) Let $\tilde{f} \in \tilde{\mathcal{F}}_d^{\xi,\eta}$.  Then, writing $f(x) := |\det \Sigma_{\tilde{f}}|^{1/2}\tilde{f}(\Sigma_{\tilde{f}}^{1/2}x + \mu_{\tilde{f}})$, we have that $f \in \mathcal{F}_d^{0,I}$.  Thus, by Theorem~\ref{Thm:IntEnv}(a), there exist $A_{0,d}, B_{0,d} > 0$ such that
\[
f(x) \leq e^{-A_{0,d}\|x\| + B_{0,d}}
\]
for all $x \in \mathbb{R}^d$.  We deduce that, for all $x \in \mathbb{R}^d$,
\begin{align*}
\tilde{f}(x) &= |\det \Sigma_{\tilde{f}}|^{-1/2}f\bigl(\Sigma_{\tilde{f}}^{-1/2}(x-\mu_{\tilde{f}})\bigr) \leq (1-\eta)^{-d/2}\exp\biggl\{-\frac{A_{0,d}\bigl|\|x\|-\|\mu_{\tilde{f}}\|\bigr|}{(1+\eta)^{1/2}} + B_{0,d}\biggr\} \\
%&= \left\{ \begin{array}{ll} (1-\eta)^{-d/2}\exp\biggl\{-\frac{A_{0,d}\|\mu_{\tilde{f}}\|}{(1+\eta)^{1/2}} + \frac{A_{0,d}\|z\|}{(1+\eta)^{1/2}} + B_{0,d}\biggr\} & \mbox{if $\|z\| \leq \|\mu_{\tilde{f}}\|$} \\
%(1-\eta)^{-d/2}\exp\biggl\{-\frac{A_{0,d}\|z\|}{(1+\eta)^{1/2}} + \frac{A_{0,d}\|\mu_{\tilde{f}}\|}{(1+\eta)^{1/2}} + B_{0,d}\biggr\} & \mbox{if $\|z\| > \|\mu_{\tilde{f}}\|$.} \end{array} \right. \\
&\leq (1-\eta)^{-d/2}\exp\biggl\{-\frac{A_{0,d}\|x\|}{(1+\eta)^{1/2}} + \frac{A_{0,d}\xi}{(1+\eta)^{1/2}} + B_{0,d}\biggr\}.
\end{align*}

(b) Suppose $(\tilde{f}_n) \in \tilde{\mathcal{F}}_d^{\xi,\eta}$ and $x_0 \in \mathbb{R}^d$ are such that $\tilde{f}_n(x_0) \searrow 0$.  For any $R > 0$,
\[
\sup_{n \in \mathbb{N}} \int_{\|x\| > R} \tilde{f}_n(x) \, dx \leq \sup_{n \in \mathbb{N}} \frac{1}{R^2} \int_{\|x\| > R} \|x\|^2 \tilde{f}_n(x) \, dx \leq \frac{2d(1+\eta) + 2\xi^2}{R^2} \rightarrow 0
\]
as $R \rightarrow \infty$, so the sequence of probability measures corresponding to $(\tilde{f}_n)$ is tight.  By Prohorov's theorem, we assert the existence of $\tilde{f} \in \tilde{\mathcal{F}}_d^{\xi,\eta}$ such that $x_0 \notin \mathrm{int}(C)$, where $C := \mathrm{dom}(\log \tilde{f})$.  But then, writing $f(x) := |\det \Sigma_{\tilde{f}}|^{1/2}\tilde{f}(\Sigma_{\tilde{f}}^{1/2}x + \mu_{\tilde{f}})$, we have that $f \in \mathcal{F}_d^{0,I}$, so by Theorem~\ref{Thm:IntEnv}(b), we must have 
\[
\frac{1}{16} < \|\Sigma_{\tilde{f}}^{-1/2}(x_0 - \mu_{\tilde{f}})\|^2 \leq \frac{(\|x_0\| + \xi)^2}{1-\eta}.
\] 
It follows that $\|x_0\| > (1-\eta)^{1/2}/4 - \xi$, as required.
\end{proof}
\begin{proof}[Proof of Proposition~\ref{Prop:Env2}]
First note that for $x_0 > 0$, the density $f(x) = \frac{1}{x_0}e^{-(x_0-x)/x_0}\mathbbm{1}_{\{x \leq x_0\}}$ belongs to $\mathcal{F}_1^0$ and satisfies $f(x_0) = 1/x_0$.  Similarly, for $x_0 < 0$, the density $f(x) = \frac{1}{|x_0|}e^{-(x-x_0)/x_0}\mathbbm{1}_{\{x \geq x_0\}}$ belongs to $\mathcal{F}_1^0$ and satisfies $f(x_0) = 1/|x_0|$.  We also observe that the sequence of densities $f_n(x) = \frac{n}{2}e^{-n|x|}$ belongs to $\mathcal{F}_1^0$ and satisfies $f_n(0) = \frac{n}{2} \rightarrow \infty$ as $n \rightarrow \infty$.  

Now let $x_0 > 0$ and suppose, for a contradiction, that $f^* \in \mathcal{F}_1^0$ satisfies $f^*(x_0) > 1/x_0$.  We must have $f^*(0) < f^*(x_0)$ (otherwise $\int_0^{x_0} f^* > 1$), so writing $\phi^* := \log f^*$, we have that 
\begin{align*}
\frac{f^*(0)x_0^2}{\{\phi^*(x_0) - \phi^*(0)\}^2} &= -\int_{-\infty}^0 x \exp\biggl\{\frac{x}{x_0}\phi^*(x_0) + \frac{(x_0 - x)}{x_0}\phi^*(0)\biggr\} \, dx \\
&\geq -\int_{-\infty}^0 xf^*(x) \, dx \geq \int_0^{x_0} xf^*(x) \, dx \\
&\geq \int_0^{x_0} x\exp\biggl\{\frac{x}{x_0}\phi^*(x_0) + \frac{(x_0 - x)}{x_0}\phi^*(0)\biggr\} \, dx \\
&= \frac{[f^*(0) + f^*(x_0)\{\phi^*(x_0) - \phi^*(0) - 1\}]x_0^2}{\{\phi^*(x_0) - \phi^*(0)\}^2}.
\end{align*}
We deduce that $\phi^*(0) \geq \phi^*(x_0) - 1$.  It follows that there exists $x^* \in (-\infty,0]$ such that $f^*(x) < \frac{1}{x_0}e^{-(x_0-x)/x_0}$ for $x < x^*$, and $f^*(x) > \frac{1}{x_0}e^{-(x_0-x)/x_0}$ for $x^* < x \leq x_0$.  But then we have for every $x \leq x_0$ that
\[
F^*(x) := \int_{-\infty}^x f^*(t) \, dt \leq \int_{-\infty}^x \frac{1}{x_0}e^{-(x_0-t)/x_0} \, dt =: F(x),
\]
say, with strict inequality for every $x \leq x_0$ except possibly when $x = x_0$, since $F(x_0) = 1$.  We deduce that 
\begin{align*}
\int_{-\infty}^\infty x f^*(x) \, dx &\geq -\int_{-\infty}^0 F^*(x) \, dx + \int_0^{x_0} \{1 - F^*(x)\} \, dx \\
&> -\int_{-\infty}^0 F(x) \, dx + \int_0^{x_0} \{1 - F(x)\} \, dx = \int_{-\infty}^{x_0} \frac{x}{x_0}e^{-(x_0-x)/x_0} \, dx = 0, 
\end{align*}
a contradiction.  A similar argument handles the case $x_0 < 0$.
\end{proof}

\subsection{Proofs from Section~\ref{Sec:Bracketing}}
\label{Sec:ProofBracketingBounds}

\begin{proof}[Proof of Theorem~\ref{Thm:BracketingBounds}]
\emph{(i)} Let $\epsilon_{00} \in (0,e^{-1}]$.  Fix $\epsilon \in (0,\epsilon_{00}]$ and set $y_k := 2^{k/2}$ for $k=0,1,\ldots,k_0$, where $k_0 := \min\{k \in \mathbb{N}: y_k \geq \log (\epsilon_{00}/\epsilon)\}$.  Let $\Phi$ denote the class of upper semi-continuous, concave functions $\phi:[0,1]^d \rightarrow [-\infty,-y_0]$, and let $\mathcal{D}$ denote the class of closed, convex subsets $D$ of $[0,1]^d$.  For $D \in \mathcal{D}$, let $\Phi_{y_0}(D) = \emptyset$ and for $k=1,\ldots,k_0$, define inductively
\[
\Phi_{y_k}(D) := \{\phi \in \Phi: \mathrm{dom}(\phi) = D \text{ and } \phi(x) \geq -y_k  \text{ for all } x \in D\}.
\]
Now let $\mathcal{F}_{y_k}(\mathcal{D}) := \{e^\phi: \phi \in \cup_{D \in \mathcal{D}} \Phi_{y_k}(D)\}$.
%We first show that there exist $\epsilon_1 \in (0,\infty)$ and $C > 0$ such that
%\[
%\log N_{[]}(\epsilon,\Phi,L_2) \leq \frac{C}{\epsilon^{d/2}}
%\]
%for $\epsilon \in (0,\epsilon_1]$.  
%We can write $\Phi = \Phi_1 \cup \Phi_2$, where 
%\[
%\Phi_1 = \{\phi \in \Phi: \phi(x) \geq -y_{k_0} \text{ for all } x \in \mathrm{dom}(\phi)\} 
%\]
%and $\Phi_2 = \Phi \setminus \Phi_1$.  
Write
\[
K_{1,k}^* := \biggl(1 + 5\sum_{j=1}^k e^{-y_{j-1}}\biggr)^{1/2} 
\]
and 
\begin{align*}
K_{2,k,1}^* &:= \sum_{j=1}^k \{e^{-y_{j-1}/2}K_1 + 8e^{-y_{j-1}/4} +  K_1^\circ y_j^{1/2}e^{-y_{j-1}/4}\}, \\
K_{2,k,2}^* &:= \sum_{j=1}^k \{K_2 e^{-y_{j-1}/2} + K_2^\circ y_j e^{-y_{j-1}/2}\}, \\
K_{2,k,3}^* &:= \sum_{j=1}^k \{K_3 e^{-y_{j-1}} + K_3^\circ y_j^2 e^{-y_{j-1}}\},
\end{align*}
where $K_d$ and $K_d^{\circ}$ are the constants defined in Propositions~\ref{Prop:Dudley} and~\ref{Prop:BoundedBrackets} below respectively.  Let 
\[
h_d(\epsilon) := \left\{ \begin{array}{ll} \epsilon^{-1/2} & \mbox{when $d=1$} \\
\epsilon^{-1}\log_{++}^{3/2}(1/\epsilon) & \mbox{when $d=2$} \\
\epsilon^{-2} & \mbox{when $d=3$.} \end{array} \right.
\]
We claim that for $k = 1,\ldots,k_0$ and $d = 1,2,3$, we have
\begin{equation}
\label{Eq:Claim}
\log N_{[]}(K_{1,k}^* \epsilon,\mathcal{F}_{y_k}(\mathcal{D}),L_2) \leq K_{2,k,d}^* h_d(\epsilon),
\end{equation}
and prove this by induction.  First consider the case $k = 1$.  By Proposition~\ref{Prop:Dudley}, we can find pairs of measurable subsets $\{(A_{j,1}^L,A_{j,1}^U):j=1,\ldots,N_{S,1,d}\}$ of $[0,1]^d$, where $N_{S,1,1} := \lfloor e^{K_1-y_0} \epsilon^{-2} \rfloor$ and $N_{S,1,d} := \lfloor \exp(K_d e^{-(d-1)y_0/2} \epsilon^{-(d-1)}) \rfloor$ for $d = 2,3$, with the properties that $L_1(\mathbbm{1}_{A_{j,1}^U},\mathbbm{1}_{A_{j,1}^L}) \leq \epsilon^2 e^{y_0}$ for $j=1,\ldots,N_{S,1,d}$ and, if $A$ is a closed, convex subset of $[0,1]^d$, then there exists $j^* \in \{1,\ldots,N_{S,1,d}\}$ such that $A_{j^*,1}^L \subseteq A \subseteq A_{j^*,1}^U$.  Note that by replacing $A_{j,1}^L$ with the closure of its convex hull if necessary, there is no loss of generality in assuming that each $A_{j,1}^L$ is closed and convex.  Moreover, by Proposition~\ref{Prop:BoundedBrackets} below, for each $j = 1,\ldots,N_{S,1,d}$ for which $A_{j,1}^L$ is $d$-dimensional, there exists a bracketing set $\{[\psi_{j,\ell,1}^L,\psi_{j,\ell,1}^U]:\ell=1,\ldots,N_{B,1,d}\}$ for $\Phi_{y_1}(A_{j,1}^L)$, where $N_{B,1,d}:= \lfloor \exp\{K_d^{\circ}h_d(\epsilon e^{y_0/2}/y_1)\}\rfloor$, such that $-y_1 \leq \psi_{j,\ell,1}^L \leq \psi_{j,\ell,1}^U \leq -y_0$, that $L_2(\psi_{j,\ell,1}^U,\psi_{j,\ell,1}^L) \leq 2\epsilon e^{y_0/2}$ and such that for every $\phi \in \Phi_{y_1}(A_{j,1}^L)$, we can find $\ell^* \in \{1,\ldots,N_{B,1,d}\}$ with $\psi_{j,\ell^*,1}^L \leq \phi \leq \psi_{j,\ell^*,1}^U$.  If $\dim(A_{j,1}^L) < d$, we define a trivial bracketing set $\{[\psi_{j,\ell,1}^L,\psi_{j,\ell,1}^U]:\ell=1,\ldots,N_{B,1,d}\}$ for $\Phi_{y_1}(A_{j,1}^L)$ by $\psi_{j,\ell,1}^L(x) := -y_1$ and $\psi_{j,\ell,1}^U(x) := -y_0$ for $x \in A_{j,1}^L$.  Note that whenever $\dim(A_{j,1}^L) < d$, we have $L_2(\psi_{j,\ell,1}^U,\psi_{j,\ell,1}^L) = 0$.  This enables us to define a bracketing set $\{[f_{j,\ell,1}^L,f_{j,\ell,1}^U]:j=1,\ldots,N_{S,1,d}, \ell=1,\ldots,N_{B,1,d}\}$ for $\mathcal{F}_{y_1}(\mathcal{D})$ by
\[
f_{j,\ell,1}^L(x) := e^{\psi_{j,\ell,1}^L(x)}\mathbbm{1}_{\{x \in A_{j,1}^L\}} \quad \text{and} \quad f_{j,\ell,1}^U(x) := e^{\psi_{j,\ell,1}^U(x)}\mathbbm{1}_{\{x \in A_{j,1}^L\}} + e^{-y_0}\mathbbm{1}_{\{x \in A_{j,1}^U \setminus A_{j,1}^L\}}
\]
for $x \in [0,1]^d$.  Note that
\begin{align*}
L_2^2(f_{j,\ell,1}^U,f_{j,\ell,1}^L) &= \int_{A_{j,1}^L} (e^{\psi_{j,\ell,1}^U} - e^{\psi_{j,\ell,1}^L})^2 \, d\mu_d + e^{-2y_0}\mu_d(A_{j,1}^U \setminus A_{j,1}^L) \\
&\leq e^{-2y_0} L_2^2(\psi_{j,\ell,1}^U,\psi_{j,\ell,1}^L) + e^{-2y_0}L_1(\mathbbm{1}_{A_{j,1}^U},\mathbbm{1}_{A_{j,1}^L}) \leq (K_{1,1}^*)^2\epsilon^2.
\end{align*}
Moreover, when $d=1$ the cardinality of this bracketing set is
\begin{align*}
N_{S,1,1}N_{B,1,1} &\leq e^{K_1-y_0} \epsilon^{-2}\exp\Bigl\{K_1^{\circ}h_1\Bigl(\frac{\epsilon e^{y_0/2}}{y_1}\Bigr)\Bigr\} \\
&\leq \exp\biggl\{e^{-y_0/2}K_1\epsilon^{-1/2} + 8e^{-y_0/4}\epsilon^{-1/2} + K_1^{\circ}h_1\Bigl(\frac{\epsilon e^{y_0/2}}{y_1}\Bigr)\biggr\} \leq e^{K_{2,1,1}^*\epsilon^{-1/2}},
\end{align*}
where we have used the facts that $e^{y_0/2}\epsilon^{1/2} \leq e^{y_{k_0-1}/2}\epsilon^{1/2} \leq \epsilon_{00}^{1/2} \leq 1$ and $2 e^{y_0/4}\epsilon^{1/2}\log(1/\epsilon) \leq 8 e^{y_{k_0-1}/4}\epsilon^{1/4} \leq 8 \epsilon_{00}^{1/4} \leq 8$.  When $d = 2$, the cardinality is
\[
N_{S,1,2}N_{B,1,2} \leq \exp\biggl\{K_2 e^{-y_0/2} \epsilon^{-1}+ K_2^{\circ}h_2\Bigl(\frac{\epsilon e^{y_0/2}}{y_1}\Bigr)\biggr\} \leq e^{K_{2,1,2}^*\epsilon^{-1}\log_{++}^{3/2}(1/\epsilon)}.
\]
Finally, when $d=3$, the cardinality of the bracketing set is
\[
N_{S,1,3}N_{B,1,3} \leq \exp\biggl\{K_3 e^{-y_0}\epsilon^{-2} + K_3^\circ h_3\Bigl(\frac{\epsilon e^{y_0/2}}{y_1}\Bigr)\biggr\} \leq e^{K_{2,1,3}^*\epsilon^{-2}}.
\]
This proves the claim~(\ref{Eq:Claim}) when $k=1$.  Now suppose the claim is true for some $k-1 < k_0-1$, so there exist brackets $\{[f_{j',k-1}^L,f_{j',k-1}^U]:j'=1,\ldots,N_{k-1,d}'\}$ for $\mathcal{F}_{y_{k-1}}(\mathcal{D})$, where $N_{k-1,d}' := \lfloor \exp\{K_{2,k-1,d}^*h_d(\epsilon)\}\rfloor$, such that $L_2(f_{j',k-1}^U,f_{j',k-1}^L) \leq K_{1,k-1}^*\epsilon$, and for every $f \in \mathcal{F}_{y_{k-1}}(\mathcal{D})$, there exists $(j')^* \in \{1,\ldots,N_{k-1,d}'\}$ such that $f_{(j')^*,k-1}^L \leq f \leq f_{(j')^*,k-1}^U$.  Let $A_{j',k-1}^U := \{x \in [0,1]^d: f_{j',k-1}^U(x) > 0\}$.  We use Proposition~\ref{Prop:Dudley} again to find pairs of measurable subsets $\{(A_{j,k}^L,A_{j,k}^U):j=1,\ldots,N_{S,k,d}\}$ of $[0,1]^d$, where $A_{j,k}^L$ is closed and convex and where $N_{S,k,1} := \lfloor e^{K_1-y_{k-1}} \epsilon^{-2} \rfloor$ and $N_{S,k,d} := \lfloor \exp(K_d e^{-y_{k-1}(d-1)/2} \epsilon^{-(d-1)}) \rfloor$ for $d = 2,3$, with the properties that $L_1(\mathbbm{1}_{A_{j,k}^U},\mathbbm{1}_{A_{j,k}^L}) \leq \epsilon^2 e^{y_{k-1}}$ for $j=1,\ldots,N_{S,k,d}$ and, if $A$ is a closed, convex subset of $[0,1]^d$, then there exists $j^* \in \{1,\ldots,N_{S,k,d}\}$ such that $A_{j^*,k}^L \subseteq A \subseteq A_{j^*,k}^U$.  Using Proposition~\ref{Prop:BoundedBrackets} below again, for each $j = 1,\ldots,N_{S,k,d}$ for which $\dim(A_{j,k}^L) = d$, there exists a bracketing set $\{[\psi_{j,\ell,k}^L,\psi_{j,\ell,k}^U]:\ell=1,\ldots,N_{B,k,d}\}$ for $\Phi_{y_k}(A_{j,k}^L)$, where $N_{B,k,d}:= \lfloor \exp\{K_d^{\circ}h_d(\frac{\epsilon e^{y_{k-1}/2}}{y_k})\}\rfloor$, such that $-y_k \leq \psi_{j,\ell,k}^L \leq \psi_{j,\ell,k}^U \leq -y_0$, that $L_2(\psi_{j,\ell,k}^U,\psi_{j,\ell,k}^L) \leq 2\epsilon e^{y_{k-1}/2}$ and that for every $\phi \in \Phi_{y_k}(A_{j,k}^L)$, we can find $\ell^* \in \{1,\ldots,N_{B,k,d}\}$ with $\psi_{j,\ell^*,k}^L \leq \phi \leq \psi_{j,\ell^*,k}^U$.  Similar to the $k=1$ case, whenever $\dim(A_{j,k}^L) < d$, we define $\psi_{j,\ell,k}^L(x) := -y_k$ and $\psi_{j,\ell,k}^U(x) := -y_0$ for $x \in A_{j,k}^L$.  We can now define a bracketing set $\{[f_{j,\ell,j',k}^L,f_{j,\ell,j',k}^U]:j=1,\ldots,N_{S,k,d}, \ell=1,\ldots,N_{B,k,d},j'=1,\ldots,N_{k-1,d}'\}$ for $\mathcal{F}_{y_k}(\mathcal{D})$ by
\begin{align*}
f_{j,\ell,j',k}^L(x) &:= e^{\psi_{j,\ell,k}^L(x)}\mathbbm{1}_{\{x \in A_{j,k}^L \setminus A_{j',k-1}^U\}} + f_{j',k-1}^L(x)\mathbbm{1}_{\{x \in A_{j',k-1}^U\}} \\
f_{j,\ell,j',k}^U(x) &:= e^{\min\{-y_{k-1},\psi_{j,\ell,k}^U(x)\}}\mathbbm{1}_{\{x \in A_{j,k}^L \setminus A_{j',k-1}^U\}} + f_{j',k-1}^U(x)\mathbbm{1}_{\{x \in A_{j',k-1}^U\}} \\
&\hspace{8cm}+ e^{-y_{k-1}}\mathbbm{1}_{\{x \in A_{j,k}^U \setminus (A_{j',k-1}^U \cup A_{j,k}^L)\}}
\end{align*}
for $x \in [0,1]^d$.  Again, we can compute
\begin{align*}
L_2^2(f_{j,\ell,j',k}^U,f_{j,\ell,j',k}^L) &\leq e^{-2y_{k-1}}L_2^2(\psi_{j,\ell,k}^U,\psi_{j,\ell,k}^L) + \epsilon^2\biggl(1 + 5\sum_{j=1}^{k-1} e^{-y_{j-1}}\biggr) + e^{-2y_{k-1}}L_1(\mathbbm{1}_{A_{j,k}^U},\mathbbm{1}_{A_{j,k}^L}) \\
&\leq (K_{1,k}^*)^2\epsilon^2.
\end{align*}
When $d=1$ the cardinality of this bracketing set is
\[
N_{k-1,1}'N_{S,k,1}N_{B,k,1} \leq e^{K_{2,k-1,1}^*h_1(\epsilon)} \times e^{K_1-y_{k-1}} \epsilon^{-2}e^{K_1^{\circ}h_1\bigl(\frac{\epsilon e^{y_{k-1}/2}}{y_k}\bigr)} \leq e^{K_{2,k,1}^* \epsilon^{-1/2}},
\]
as required.  When $d = 2$, the cardinality is 
\begin{align*}
N_{k-1,2}'N_{S,k,2}N_{B,k,2} &\leq \exp\Bigl\{K_{2,k-1,2}^*h_2(\epsilon) + K_2 e^{-y_{k-1}/2}\epsilon^{-1} + K_2^{\circ}h_2\Bigl(\frac{\epsilon e^{y_{k-1}/2}}{y_k}\Bigr)\Bigr\} \\
&\leq e^{K_{2,k,2}^*\epsilon^{-1}\log_{++}^{3/2}(1/\epsilon)}.
\end{align*}
Finally, when $d = 3$, the cardinality of the bracketing set is
\[
N_{k-1,3}'N_{S,k,3}N_{B,k,3} \leq \exp\Bigl\{K_{2,k-1,3}^*h_3(\epsilon) + K_3 e^{-y_{k-1}}\epsilon^{-2} + K_3^{\circ}h_3\Bigl(\frac{\epsilon e^{y_{k-1}/2}}{y_k}\Bigr)\Bigr\} \leq e^{K_{2,k,3}^*\epsilon^{-2}}.
\]
This establishes the claim~(\ref{Eq:Claim}) by induction.

We now consider the class $\bar{\mathcal{F}}_{y_{k_0}}(\mathcal{D}) := \{e^\phi: \phi \in \Phi \setminus \cup_{D \in \mathcal{D}} \Phi_{y_{k_0}}(D)\}$.  A bracketing set for this class is given by $\{[\bar{f}_{j,\ell,j'}^L,\bar{f}_{j,\ell,j'}^U]:j=1,\ldots,N_{S,k_0,d}, \ell=1,\ldots,N_{B,k_0,d},j'=1,\ldots,N_{k_0-1,d}'\}$, where
\begin{align*}
\bar{f}_{j,\ell,j'}^L(x) &:= f_{j,\ell,j',k_0}^L(x) \\
\bar{f}_{j,\ell,j'}^U(x) &:= f_{j,\ell,j',k_0}^U(x)\mathbbm{1}_{\{x \in A_{j,k_0}^U\}} + e^{-y_{k_0}}\mathbbm{1}_{\{x \notin A_{j,k_0}^U\}}
\end{align*}
for $x \in [0,1]^d$.  Observe that
\[
L_2^2(\bar{f}_{j,\ell,j'}^U,\bar{f}_{j,\ell,j'}^L) \leq (K_{1,k_0}^*)^2\epsilon^2 + e^{-2y_{k_0}} \leq \Bigl(K_{1,k_0}^* + \frac{1}{\epsilon_{00}}\Bigr)^2\epsilon^2.
\]
Since $k_0$ depends on $\epsilon$, it is important to observe that for all $k=1,\ldots,k_0$,
\begin{align*}
K_{1,k}^* &\leq 4 \\
K_{2,k,1}^* &\leq 2K_1 + 32 + 8K_1^{\circ} =: \bar{K}_{2,1}^* - \log 2, \\
K_{2,k,2}^* &\leq 2K_2 + K_2^{\circ}(8e^{1/2}+1) =: \bar{K}_{2,2}^* - \log 2, \\
K_{2,k,3}^* &\leq K_3 + K_3^\circ(8e + 1) =: \bar{K}_{2,3}^* - \log 2.
\end{align*}
In particular, these bounds do not depend on $\epsilon$.  For $\tilde{b} > 0$, write $\mathcal{G}_{d,[0,1]^d,\tilde{b}}$ for the set of functions on $[0,1]^d$ of the form $f^{1/2}$, where $f$ is an upper semi-continuous, log-concave function whose domain is a closed, convex subset of $[0,1]^d$, and for which $f^{1/2} \leq \tilde{b}$.  Noting that $\mathcal{G}_{d,[0,1]^d,e^{-1}} \subseteq \{e^\phi:\phi \in \Phi\} = \mathcal{F}_{y_{k_0}}(\mathcal{D}) \cup \bar{\mathcal{F}}_{y_{k_0}}(\mathcal{D})$, and since $\epsilon \in (0,\epsilon_{00}]$ was arbitrary, we conclude that
\begin{align*}
\log N_{[]}\bigl((4+\epsilon_{00}^{-1})\epsilon,\mathcal{G}_{d,[0,1]^d,e^{-1}},L_2\bigr) &\leq \log N_{[]}\bigl((4+\epsilon_{00}^{-1})\epsilon,\{e^\phi:\phi \in \Phi\},L_2\bigr) \\
&\leq \bar{K}_{2,d}^* h_d(\epsilon)
\end{align*}
for all $\epsilon \in (0,\epsilon_{00}]$ and $d = 1,2,3$.  By a simple scaling argument, we deduce that for any $b > 0$,  
\[
\log N_{[]}\bigl((4+\epsilon_{00}^{-1})\epsilon b^{1/2},\mathcal{G}_{d,[0,1]^d,be^{-1}},L_2\bigr) \leq \bar{K}_{2,d}^*h_d(\epsilon/b^{1/2})
\]
for all $\epsilon \in (0,b^{1/2}\epsilon_{00}]$.

%To reduce the number of subscripts, and to allow translation of brackets, we relabel the brackets constructed for $\mathcal{F}_{d,[0,1]^d,-1}$ as $\{[\tilde{f}_{\mathbf{0},\ell}^L,\tilde{f}_{\mathbf{0},\ell}^U]:\ell=1,\ldots,N_{[]}((4+\epsilon_1^{-1})\epsilon,\mathcal{F}_{d,[0,1]^d,-1},L_2)\}$.  
We now show how to translate and scale brackets appropriately for other cubes.  Let $A_{0,d}, B_{0,d} > 0$ be as in Corollary~\ref{Cor:Transform}(a).  Define 
\[
T_d := \frac{A_{0,d}(d^{1/2}+1)}{(1+\eta_d)^{1/2}} + B_{0,d} + \frac{d}{2} \log \biggl(\frac{1}{1-\eta_d}\biggr) + d+1,
\]
set $\epsilon_{01,d} := \min\bigl\{e^{-T_d},\frac{1}{d^d}\epsilon_{00}^4\}$ and fix $\epsilon \in (0,\epsilon_{01,d}]$.  For $\mathbf{j} = (j_1,\ldots,j_d) \in \mathbb{Z}^d$, let
\[
C_\mathbf{j}^2 := \exp\biggl(-\frac{A_{0,d}\|\mathbf{j}\|}{(1+\eta_d)^{1/2}} + T_d\biggr),
\]
where $\|\mathbf{j}\|^2 := \sum_{k=1}^d j_k^2$.  Note from Corollary~\ref{Cor:Transform}(a) that 
\[
\sup_{\tilde{f} \in \tilde{\mathcal{F}}_d^{1,\eta_d}} \sup_{x \in [j_1,j_1+1] \times \ldots \times [j_d,j_d+1]} \tilde{f}(x)^{1/2} \leq C_\mathbf{j}e^{-1}.  
\]
Let $j_0 := \max\{\|\mathbf{j}\|:\mathbf{j} \in \mathbb{Z}^d, C_{\mathbf{j}} \geq \epsilon\{\log(1/\epsilon)\}^{-(d-1)/2}\}$, so we may assume $j_0 \geq 1$.  For $\mathbf{j} = (j_1,\ldots,j_d) \in \mathbb{Z}^d$ such that $\|\mathbf{j}\| \leq j_0$, let $N_{\mathbf{j}} := N_{[]}\bigl((4+\epsilon_{00}^{-1})\epsilon C_{\mathbf{j}}^{1/2},\mathcal{G}_{d,[0,1]^d,C_{\mathbf{j}}e^{-1}},L_2\bigr)$, and let $\{[f_{\mathbf{j},\ell}^L,f_{\mathbf{j},\ell}^U],\ell=1,\ldots,N_\mathbf{j}$\}, denote a bracketing set for $\mathcal{G}_{d,[0,1]^d,C_{\mathbf{j}}e^{-1}}$ with $L_2(f_{\mathbf{j},\ell}^U,f_{\mathbf{j},\ell}^L) \leq (4+\epsilon_{00}^{-1})\epsilon C_{\mathbf{j}}^{1/2}$.  Such a bracketing set can be found because when $\|\mathbf{j}\| \leq j_0$, we have 
\[
\epsilon \leq C_{\mathbf{j}}^{1/2}\epsilon^{1/2}\{\log(1/\epsilon)\}^{d/4} \leq C_{\mathbf{j}}^{1/2}\epsilon^{1/2}(d\epsilon^{-(1/d)})^{d/4} \leq C_{\mathbf{j}}^{1/2}\epsilon_{00}.
\]
Finally, for $\{\boldsymbol{\ell} = (\ell_{\mathbf{j}}) \in \times_{\mathbf{j}:\|\mathbf{j}\| \leq j_0} \{1,\ldots,N_{\mathbf{j}}\}\}$, we define a bracketing set for $\{\tilde{f}^{1/2}:\tilde{f} \in \tilde{\mathcal{F}}_d^{1,\eta_d}\}$ by
\begin{align*}
f_{\boldsymbol{\ell}}^L(x) &:= \sum_{\mathbf{j}:\|\mathbf{j}\|\leq j_0} f_{\mathbf{j},\ell_{\mathbf{j}}}^L(x - \mathbf{j})\mathbbm{1}_{\{x \in [j_1,j_1+1) \times \ldots \times [j_d,j_d+1)\}}, \\
f_{\boldsymbol{\ell}}^U(x) &:= \sum_{\mathbf{j}:\|\mathbf{j}\|\leq j_0} f_{\mathbf{j},\ell_{\mathbf{j}}}^U(x - \mathbf{j})\mathbbm{1}_{\{x \in [j_1,j_1+1) \times \ldots \times [j_d,j_d+1)\}} + e^{-1}\sum_{\mathbf{j}:\|\mathbf{j}\| > j_0} C_{\mathbf{j}} \mathbbm{1}_{\{x \in [j_1,j_1+1) \times \ldots \times [j_d,j_d+1)\}}
%2^{d/4}e^{-\frac{A_{0,d}\|\lfloor x \rfloor\|}{6^{1/2}} + \frac{A_{0,d}}{24^{1/2}} + \frac{B_{0,d}}{2}}\mathbbm{1}_{\{\|\lfloor x\rfloor \| > j_0\}},
%\sum_{\mathbf{j}:\|\mathbf{j}\| > j_0} C_{\mathbf{j}}^{1/2} \mathbbm{1}_{\{x \in [j_1,j_1+1) \times \ldots \times [j_d,j_d+1)\}}. 
\end{align*}
for $x \in \mathbb{R}^d$.  Note that
\begin{align*}
L_2(f_{\boldsymbol{\ell}}^U,f_{\boldsymbol{\ell}}^L) &\leq (4+\epsilon_{00}^{-1})\epsilon \biggl( \, \sum_{\mathbf{j} \in \mathbb{Z}^d} C_{\mathbf{j}}\biggr)^{1/2} + \biggl( \, \sum_{\mathbf{j}:\|\mathbf{j}\| > j_0} C_{\mathbf{j}}^2\biggr)^{1/2}e^{-1} \\
&\leq (4+\epsilon_{00}^{-1})\epsilon \frac{e^{\frac{A_{0,d}d^{1/2}}{4(1+\eta_d)^{1/2}} + \frac{T_d}{4}}d^{1/2}\pi^{d/4}}{\Gamma(1+d/2)^{1/2}}\biggl\{\int_0^\infty r^{d-1} e^{-\frac{r A_{0,d}}{2(1+\eta_d)^{1/2}}} \, dr\biggr\}^{1/2} \\
&\hspace{1cm}+ \frac{e^{\frac{A_{0,d}d^{1/2}}{2(1+\eta_d)^{1/2}} + \frac{T_d}{2}-1}d^{1/2}\pi^{d/4}}{\Gamma(1+d/2)^{1/2}}\biggl\{\int_{j_0}^\infty r^{d-1} e^{-\frac{r A_{0,d}}{(1+\eta_d)^{1/2}}} \, dr\biggr\}^{1/2} \\
&\leq \epsilon (B_1 + B_2),
%&\leq \frac{2^{d/4}e^{\frac{A_{0,d} d^{1/2}}{6^{1/2}} + \frac{A_{0,d}}{24^{1/2}} + \frac{B_{0,d}}{2}}d^{1/2}\pi^{d/4}}{\Gamma(1+d/2)^{1/2}}\biggl\{(4+\epsilon_{00}^{-1})\epsilon + 
%&=: \tilde{K}_d \epsilon.
\end{align*}
where
\begin{align*}
B_1 &:= (4+\epsilon_{00}^{-1})\frac{e^{\frac{A_{0,d}d^{1/2}}{4(1+\eta_d)^{1/2}} + \frac{T_d}{4}}d^{1/2}\pi^{d/4}}{\Gamma(1+d/2)^{1/2}}\frac{\{(d-1)!\}^{1/2}2^{d/2}(1+\eta_d)^{d/4}}{A_{0,d}^{d/2}}, \\
B_2 &:= \frac{e^{\frac{A_{0,d}d^{1/2}}{2(1+\eta_d)^{1/2}} + \frac{T_d}{2}-1}d^{1/2}\pi^{d/4}}{\Gamma(1+d/2)^{1/2}}\frac{(1+\eta_d)^{d/4}}{A_{0,d}^{d/2}}e^{-\frac{T_d}{2} + \frac{A_{0,d}}{2(1+\eta_d)^{1/2}}}(d+2)^{d/2}.
\end{align*}
Note that to obtain the expression for $B_2$, we have used the fact that 
\begin{align*}
\frac{1}{\epsilon} \int_{j_0}^\infty r^{d-1} e^{-\frac{r A_{0,d}}{(1+\eta_d)^{1/2}}} \, dr &= \frac{(1+\eta_d)^{d/4}}{A_{0,d}^{d/2}}\{(d-1)!\}^{1/2}e^{-\frac{j_0A_{0,d}}{2(1+\eta_d)^{1/2}}} \biggl\{\sum_{k=0}^{d-1} \frac{j_0^k A_{0,d}^k}{(1+\eta_d)^{k/2}k!} \biggr\}^{1/2}\epsilon^{-1} \\
&\leq \frac{(1+\eta_d)^{d/4}}{A_{0,d}^{d/2}}e^{-\frac{T_d}{2} + \frac{A_{0,d}}{2(1+\eta_d)^{1/2}}}(d+2)^{d/2},
\end{align*}
using the definition of $j_0$ and $\epsilon_{01,d}$.
%&\geq \frac{e^{\frac{A_{0,d}d^{1/2}}{2(1+\eta_d)^{1/2}} + \frac{T_d}{2}-1}d^{1/2}\pi^{d/4}}{\Gamma(1+d/2)^{1/2}}\frac{(1+\eta_d)^{d/4}}{A_{0,d}^{d/2}}\{(d-1)!\}^{1/2}e^{-\frac{j_0A_{0,d}}{2(1+\eta_d)^{1/2}}} \biggl\{\sum_{k=0}^{d-1} \frac{j_0^k A_{0,d}^k}{(1+\eta_d)^{k/2}k!} \biggr\}^{1/2}\epsilon^{-1}.
%&\leq \frac{e^{\frac{A_{0,d}d^{1/2}}{2(1+\eta_d)^{1/2}} + \frac{T_d}{2}-1}d^{1/2}\pi^{d/4}}{\Gamma(1+d/2)^{1/2}}\frac{(1+\eta_d)^{d/4}}{A_{0,d}^{d/2}}(2e)^{1/2}e^{-\frac{T_d}{2} + \frac{A_{0,d}}{2(1+\eta_d)^{1/2}}}\big\{1 + (d+2)^{(d-1)/2}\bigr\},
%&\leq \frac{e^{\frac{A_{0,d}d^{1/2}}{2(1+\eta_d)^{1/2}} + \frac{T_d}{2}-1}d^{1/2}\pi^{d/4}}{\Gamma(1+d/2)^{1/2}}\frac{(1+\eta_d)^{d/4}}{A_{0,d}^{d/2}}e^{-\frac{T_d}{2} + \frac{A_{0,d}}{2(1+\eta_d)^{1/2}}}(d+2)^{d/2},
%\end{align*}
Moreover, the cardinality of the bracketing set is
\[
\prod_{\mathbf{j}:\|\mathbf{j}\| \leq j_0} N_{\mathbf{j}} = \exp\biggl\{\bar{K}_{2,d}^* \sum_{\mathbf{j}:\|\mathbf{j}\| \leq j_0} h_d\Bigl(\frac{\epsilon}{C_{\mathbf{j}}^{1/2}}\Bigr)\biggr\} \leq \exp\bigl\{\bar{K}_{2,d}^* B_{3,d}h_d(\epsilon)\bigr\},
\]
where
\begin{align*}
B_{3,1} &:= \sum_{\mathbf{j}:\|\mathbf{j}\| \leq j_0} C_{\mathbf{j}}^{1/4} \leq e^{T_1/8}e^{\frac{A_{0,1}}{8(1+\eta_d)^{1/2}}}\frac{16(1+\eta_d)^{1/2}}{A_{0,1}}, \\
%B_{3,2} &:= 2^{3/2}\sum_{\mathbf{j}:\|\mathbf{j}\| \leq j_0} C_{\mathbf{j}}^{1/2} \leq e^{T_2/4}2^{5/2} \pi\frac{16(1+\eta_d)}{A_{0,2}^2}, \\
B_{3,2} &:= 2^{3/2}\sum_{\mathbf{j}:\|\mathbf{j}\| \leq j_0} C_{\mathbf{j}}^{1/2} \leq e^{T_2/4}2^{5/2} \pi e^{\frac{A_{0,2}}{2^{3/2}(1+\eta_d)^{1/2}}}\frac{16(1+\eta_d)}{A_{0,2}^2}, \\
%B_{3,3} &:= \sum_{\mathbf{j}:\|\mathbf{j}\| \leq j_0} C_{\mathbf{j}} \leq e^{T_3/2}\frac{3\pi^{3/2}}{\Gamma(5/2)} \frac{16(1+\eta_d)^{3/2}}{A_{0,3}^3}.
B_{3,3} &:= \sum_{\mathbf{j}:\|\mathbf{j}\| \leq j_0} C_{\mathbf{j}} \leq e^{T_3/2}4\pi e^{\frac{3^{1/2}A_{0,3}}{2(1+\eta_d)^{1/2}}} \frac{8(1+\eta_d)^{3/2}}{A_{0,3}^3}.
\end{align*}
Since $\epsilon \in (0,\epsilon_{01,d}]$ was arbitrary, we conclude that 
\[
\log N_{[]}(\epsilon,\tilde{\mathcal{F}}_d^{1,\eta_d},h) = \log N_{[]}(\epsilon,\{\tilde{f}^{1/2}:\tilde{f} \in \tilde{\mathcal{F}}_d^{1,\eta_d}\},L_2) \leq \overline{\overline{K}}_dh_d(\epsilon),
\]
for all $\epsilon \in (0,\epsilon_{02,d}]$, where $\epsilon_{02,d} := \epsilon_{01,d}(B_1+B_2)$ and where 
\[
\overline{\overline{K}}_d := \bar{K}_{2,d}^* B_{3,d}\max\{(B_1+B_2)^{d/2},(B_1+B_2)^{(d-1)}\}\biggl\{2 + \frac{2\log_{++}(B_1+B_2)}{\log_{++}(e/(B_1+B_2))}\biggr\},
\]
where, as in the proof of Proposition~\ref{Prop:Dudley} below, we have used the fact that $\log_{++}(a/\epsilon) \leq \bigl\{2 + \frac{2\log_{++}(a)}{\log_{++}(e/a)}\bigr\}\log_{++}(1/\epsilon)$ for all $a,\epsilon > 0$.  Now let
\[
\epsilon_{03,d} := \max\biggl\{\epsilon_{02,d},\biggl[\frac{(1+\eta_d)^{d/2}}{(1-\eta_d)^{d/2}}\exp\biggl\{\frac{A_{0,d}}{(1+\eta_d)^{1/2}} + B_{0,d}\biggr\}\frac{d!\pi^{d/2}}{\Gamma(1+d/2)A_{0,d}^d}\biggr]^{1/2}\biggr\},
\]
and let $\overline{K}_d := \overline{\overline{K}}_dh_d(\epsilon_{02,d})/h_d(\epsilon_{03,d})$.  For $\epsilon \in (\epsilon_{02,d},\epsilon_{03,d}]$, we have
\[
\log N_{[]}(\epsilon,\tilde{\mathcal{F}}_d^{1,\eta_d},h) \leq \log N_{[]}(\epsilon_{02,d},\tilde{\mathcal{F}}_d^{1,\eta_d},h) \leq \overline{\overline{K}}_dh_d(\epsilon_{02,d}) = \overline{K}_dh_d(\epsilon_{03,d}) \leq \overline{K}_dh_d(\epsilon).
\]
Finally, if $\epsilon > \epsilon_{03,d}$, we can use a single bracketing pair $\{f^L,f^U\}$, with $f^L(x) := 0$ and $f^U(x)$ defined to be the integrable envelope function from Corollary~\ref{Cor:Transform}(a) with $\xi = 1$ and $\eta = \eta_d$ there.  Note that $h(f^U,f^L) \leq \epsilon_{03,d}$.  This proves the upper bound.

\emph{(ii)} Let $\epsilon_{10,d} := \min\bigl\{10^{-6},\eta_d^2/400\bigr\}$.  We start with the case $d=1$, and construct a subset of $\tilde{\mathcal{F}}_1^{1,\eta_1}$ such that each pair of functions in our subset is well separated in Hellinger distance.  Our construction is similar (but not identical) to that in the proof of Theorem~\ref{Thm:LowerBound}.  In particular, our densities are perturbations of part of a semicircle density (with an appropriate constant subtracted), but we need to choose the radius of the semicircle carefully to ensure that the variances of our densities are close to 1.  Fix $\epsilon \in (0,\epsilon_{10,1}]$, and let $\zeta^*$ be the unique solution in $[0.148,0.149]$ of the equation
\[
\frac{2\zeta - \frac{1}{2}\sin(4\zeta) - \frac{2}{3}\sin^3(2\zeta)\cos(2\zeta)}{4\{2\zeta -\frac{1}{2}\sin(4\zeta)\}^2} = 1.
\]
Set $K := \lfloor \frac{\zeta^*}{\arcsin (\epsilon^{1/2})} \rfloor$ and, for $k=0,1,\ldots,K$, let $w_k := k \arcsin(\epsilon^{1/2})$, so that $\zeta^* - 2\epsilon^{1/2} \leq w_K \leq \zeta^*$.  We also define 
\[
r := \biggl\{w_K - \frac{1}{2}\sin w_{4K} + K\epsilon^{1/2}(1-\epsilon)^{1/2}\biggr\}^{-1/2}.
\]
Note that
\begin{align*}
w_K - \frac{1}{2}\sin w_{4K} + K\epsilon^{1/2}(1-\epsilon)^{1/2} \geq 2w_K -\frac{1}{2}\sin w_{4K} - w_K\epsilon &\geq 0.01.
\end{align*}
As in the proof of Theorem~\ref{Thm:LowerBound}, for $k=1,\ldots,K$ and $\ell \in \{0,1\}$, define
\[
x_{k,\ell} := (-1)^\ell r(1-\epsilon)^{1/2}\sin w_{2k-1}.
\]
For $k=1,\ldots,K$, we also define $R_{k,0} := (r \sin w_{2k-2},r \sin w_{2k})$ and set $R_{k,1} := -R_{k,0} = \{-x:x \in R_{k,0}\}$.  Writing $y_k := r(1-\epsilon)^{1/2}\cos w_{2k-1}$, for $k=1,\ldots,K$, we define auxiliary functions 
\begin{align*}
\psi_k(x) &:= (r^2-x^2)^{1/2} \mathbbm{1}_{\{x \in R_{k,0}\}} + \frac{1}{y_k}\{(1 - \epsilon)r^2 - x_{k,1} x\}\mathbbm{1}_{\{x \in R_{k,1}\}}, \\
\tilde{\psi}_k(x) &:= \frac{1}{y_k}\{(1 - \epsilon)r^2 - x_{k,0} x\}\mathbbm{1}_{\{x \in R_{k,0}\}} + (r^2-x^2)^{1/2} \mathbbm{1}_{\{x \in R_{k,1}\}}.
\end{align*}
We can now define $\mathcal{F}_1^L := \{f_\alpha: \alpha = (\alpha_1,\ldots,\alpha_K)^T \in \{0,1\}^K\}$, where
\begin{align*}
f_\alpha(x) &:= -r\cos w_{2K} \mathbbm{1}_{\{|x| \leq r\sin w_{2K}\}} \\
&+ (r^2-x^2)^{1/2} \mathbbm{1}_{\{|x| \leq r\sin w_{2K}\}}\mathbbm{1}_{\{x \notin \cup_{k=1}^K (R_{k,0} \cup R_{k,1})\}} + \sum_{k=1}^K \bigl\{\alpha_k \psi_k(x) + (1-\alpha_k)\tilde{\psi}_k(x)\bigr\}.
\end{align*}
Note here that the only reason for including the second term in this sum is to ensure that each $f_\alpha$ is continuous at the boundaries of the sets $R_{k,\ell}$.  Observe that
\[
\int_{-r\sin w_{2K}}^{r\sin w_{2K}} f_\alpha = r^2\biggl\{w_K - \frac{1}{2}\sin w_{4K} + K\epsilon^{1/2}(1-\epsilon)^{1/2}\biggr\} = 1,
\]
and $\mathcal{F}_1^L \subseteq \mathcal{F}_1$.  Now
\[
\biggl|\int_{-r\sin w_{2K}}^{r\sin w_{2K}} x f_\alpha(x) \, dx\biggr| \leq w_Kr^3\sin (w_{2K})\biggl\{1 - \frac{\epsilon^{1/2}(1-\epsilon)^{1/2}}{w_1}\biggr\} \leq 50\epsilon_{10,1} \leq \frac{\eta_1^{1/2}}{2^{1/2}},
\] 
since $\eta_1^2/400 \leq \eta_1^{1/2}/(2^{1/2}\times 50)$.  We also compute
\begin{align*}
\int_{-r\sin w_{2K}}^{r\sin w_{2K}} x^2 f_\alpha(x) \, dx &\leq \frac{r^4}{4}\biggl\{2w_K - \frac{1}{2}\sin w_{4K} - \frac{2}{3}\sin^3 w_{2K}\cos w_{2K}\biggr\} \\
&\leq \frac{2w_K - \frac{1}{2}\sin w_{4K} - \frac{2}{3}\sin^3 w_{2K}\cos w_{2K}}{4\{2w_K -\frac{1}{2}\sin w_{4K} - w_K\epsilon\}^2} \leq 1 + 20\epsilon_{10,1}^{1/2} \leq 1+\eta_1.
\end{align*}
Finally, since $f_\alpha(x) \geq \bigl\{r^2(1-\epsilon) - x^2\bigr\}^{1/2} - r\cos  w_{2K}$ for $|x| \leq r(1-\epsilon)^{1/2}\sin w_{2K}$, we have
\[
\int_{-r\sin w_{2K}}^{r\sin w_{2K}} x^2 f_\alpha(x) \, dx \geq \frac{r^4(1-\epsilon)^{3/2}}{4}\biggl\{2w_K - \frac{1}{2}\sin w_{4K} - \frac{2}{3}\sin^3 w_{2K}\cos w_{2K}\biggr\} \geq 1 - \frac{\eta_1}{2},
\]
since $(1-\epsilon)^{3/2} \geq 1 - 3\epsilon_{10,1}/2 \geq 1 - \eta_1/2$, so $\mathcal{F}_1^L \subseteq \mathcal{\tilde{F}}_1^{1,\eta_1}$.  By the Gilbert--Varshamov bound \citep[e.g.][Lemma~4.7]{Massart2007}, there exists a subset $\mathcal{F}_{1,*}^L$ of $\mathcal{F}_1^L$ of cardinality $e^{K/8} \geq e^{\frac{0.148}{16}\epsilon^{-1/2}}$ such that $\|\alpha - \beta\|_0 \geq K/4$ for all $f_\alpha, f_\beta \in \mathcal{F}_{1,*}^L$ with $\alpha \neq \beta$.  But then, since $|f_\alpha| \leq r \leq 10$, and $r \geq 7$, we deduce from the proof of Theorem~\ref{Thm:LowerBound} that for any $f_\alpha, f_\beta \in \mathcal{F}_{1,*}^L$ for $\alpha \neq \beta$, we have
\[
h^2(f_\alpha,f_\beta) \geq \frac{1}{4r}L_2^2(f_\alpha,f_\beta) \geq \frac{31}{420}\|\alpha-\beta\|_0 r^2\epsilon^{5/2} > \frac{1}{16} \epsilon^2.
\]
Since the bracketing number at level $\epsilon$ is bounded below by the packing number at level $2\epsilon$, we can let $\epsilon_{11,1} := \epsilon_{10,1}/8$, and conclude that
\[
\log N_{[]}(\epsilon,\mathcal{\tilde{F}}_1^{1,\eta_1},h) \geq \underline{K}_1\epsilon^{-1/2}
\]
for $\epsilon \in (0,\epsilon_{11,1}]$, where $\underline{K}_1 := \frac{0.148}{8^{1/2}16}$.

Finally, we turn to the case $d \geq 2$.  Set $\epsilon_{10,d} := \min\bigl\{10^{-4},\frac{\eta_d^{1/2}}{4(d+2)^{1/2}}\bigr\}$ and fix $\epsilon \in (0,\epsilon_{10,d}]$.  Here, we recall the finite subset $\bar{\mathcal{F}}_d = \bigl\{f_\alpha: \alpha \in \{0,1\}^K\bigr\}$ of uniform densities on closed, convex sets from the proof of Theorem~\ref{Thm:LowerBound} in the case $d \geq 2$, and set 
\[
\bar{\mathcal{F}}_{d,r} := \{f_{\alpha,r}(\cdot) = r^{-d}f_\alpha(\cdot/r): f_\alpha \in \bar{\mathcal{F}}_d\},
\]
with $r := (d+2)^{1/2}$.  Our reason for choosing $r := (d+2)^{1/2}$ is to ensure that the densities in our class have marginal variances close to 1.  Again, we must check that $\bar{\mathcal{F}}_{d,r} \subseteq \tilde{\mathcal{F}}_d^{1,\eta_d}$.  To this end, note that for any $f_{\alpha,r} \in \bar{\mathcal{F}}_{d,r}$, we have
\begin{align*}
\biggl\|\int_{\mathbb{R}^d} x f_{\alpha,r}(x) \, dx\biggr\| &\leq \frac{Kr}{2c_{K,\epsilon}} \frac{\pi^{(d-1)/2}}{\Gamma((d+1)/2)} \int_0^{\epsilon^2 - \epsilon^4/4} t^{\frac{d+1}{2}-1}(1-t)^{-1/2} \, dt \leq \frac{(d+2)^{1/2}}{2^{d-2}}\epsilon^2,
\end{align*}
where we have used the bound on $c_{K,\epsilon}$ from~\eqref{Eq:crKe} and the fact that $\frac{\Gamma(1+d/2)}{\Gamma(\frac{1+d}{2})} \leq (d+1)^{1/2}/2^{1/2}$.  Now, for any $j = 1,\ldots,d$,
\begin{align*}
\int_{\mathbb{R}^d} x_j^2 f_{\alpha,r}(x) \, dx &\leq \frac{1}{c_{K,\epsilon}r^d} \int_{\bar{B}_d(0,r)} x_j^2 \, dx = \frac{1}{d c_{K,\epsilon}r^d} \int_{\bar{B}_d(0,r)} \|x\|^2 \, dx \\
&\leq \frac{1}{1 - K\epsilon^{d+1}\pi^{-1/2}(d+1)^{-1/2}}\frac{r^2}{d+2} \leq 1+\frac{\eta_d}{2},
\end{align*}
and
\begin{align*}
\int_{\mathbb{R}^d} x_j^2 f_{\alpha,r}(x) \, dx &\geq \frac{1}{d c_{K,\epsilon}r^d} \int_{\bar{B}_d(0,r(1-\epsilon^2/2))} \|x\|^2 \, dx \geq \frac{r^2(1-\epsilon^2/2)^{d+2}}{d+2} \geq 1 - \frac{\eta_d}{2}. 
\end{align*}  
Finally, for $j, k \in \{1,\ldots,d\}$ with $j \neq k$, we have
\[
\biggl|\int_{\mathbb{R}^d} x_jx_k f_{\alpha,r}(x) \, dx\biggr| \leq \frac{Kr^2}{2c_{K,\epsilon}} \frac{\pi^{(d-1)/2}}{\Gamma((d+1)/2)} \int_0^{\epsilon^2 - \epsilon^4/4} t^{\frac{d+1}{2}-1}(1-t)^{-1/2} \, dt \leq \frac{d+2}{2^{d-2}}\epsilon^2.
\]
We deduce from the Gerschgorin circle theorem \citep{Gerschgorin1931,GradshteynRyzhik2007} that if $\Sigma_{\alpha,r}$ denotes the covariance matrix corresponding to $f_{\alpha,r}$, then 
\begin{align*}
1 - \eta_d \leq 1 - \frac{\eta_d}{2} - \frac{(d+2)}{2^{2(d-2)}}\epsilon^4 &- (d-1)\frac{(d+2)}{2^{d-3}}\epsilon^2 \leq \lambda_{\mathrm{min}}(\Sigma_{\alpha,r}) \\
&\leq \lambda_{\mathrm{max}}(\Sigma_{\alpha,r}) \leq 1 + \frac{\eta_d}{2} + (d-1)\frac{(d+2)}{2^{d-3}}\epsilon^2 \leq 1 + \eta_d.
\end{align*}
We conclude that $\bar{\mathcal{F}}_{d,r} \subseteq \tilde{\mathcal{F}}_d^{1,\eta_d}$.  By the Gilbert--Varshamov bound again, there exists a subset $\mathcal{F}_{d,*}^L$ of $\bar{\mathcal{F}}_{d,r}$ of cardinality $e^{K/8} \geq e^{\frac{(d-1)^{1/2}}{2^{d+4}}\epsilon^{-(d-1)}}$ such that $\|\alpha - \beta\|_0 \geq K/4$ for all $f_\alpha, f_\beta \in \mathcal{F}_{d,*}^L$.  But from the proof of Theorem~\ref{Thm:LowerBound}, for any $f_\alpha, f_\beta \in \mathcal{F}_{d,*}^L$, we have
\[
h^2(f_\alpha,f_\beta) \geq \frac{15^{(d+1)/2}}{16^{(d+1)/2}2\pi(d+1)^{1/2}} K\epsilon^{d+1} > \frac{15^{(d+1)/2}}{10 \times 2^{d+1}16^{(d+1)/2}}\epsilon^2.
\]
Setting $\epsilon_d := \frac{1}{2}\frac{15^{(d+1)/4}}{10^{1/2}2^{(d+1)/2}16^{(d+1)/4}}\epsilon_{10,d}$, we conclude that 
\[
\log N_{[]}(\epsilon,\tilde{\mathcal{F}}_d^{1,\eta_d},h) \geq \underline{K}_d\epsilon^{-(d-1)}
\]
for $\epsilon \in (0,\epsilon_d]$, where 
\[
\underline{K}_d := \frac{(d-1)^{1/2}}{2^{2d+3}}\biggl(\frac{15^{(d+1)/2}}{10 \times 2^{d+1}16^{(d+1)/2}}\biggr)^{(d-1)/2}. 
\]
%Thus, if we set $\epsilon_0 := \min(\epsilon_{0,d}^L,\epsilon_{0,d}^U)$, then the upper and lower bounds in Theorem~\ref{Thm:BracketingBounds} hold for all $\epsilon \in (0,\epsilon_0]$, as required.
\end{proof}

\begin{proof}[Proof of Theorem~\ref{Thm:Main}]
Let $\mu := \mathbb{E}(X_1)$ and $\Sigma := \mathrm{Cov}(X_1)$.  Note that since $f_0 \in \mathcal{F}_d$, we have that $\Sigma$ is a finite, positive definite matrix.  We can therefore define $Z_i := \Sigma^{-1/2}(X_i - \mu)$ for $i=1,\ldots,n$, so that $\mathbb{E}(Z_1) = 0$ and $\mathrm{Cov}(Z_1) = I$.  We also set $g_0(z) := (\det \Sigma)^{1/2} f_0(\Sigma^{1/2} z + \mu)$, so $g_0 \in \mathcal{F}_d^{0,I}$, and let $\hat{g}_n(z) := (\det \Sigma)^{1/2}\hat{f}_n(\Sigma^{1/2} z + \mu)$, so by affine equivariance \citep[][Remark~2.4]{DSS2011}, $\hat{g}_n$ is the log-concave maximum likelihood estimator of $g_0$ based on $Z_1,\ldots,Z_n$.

Let $\hat{\mu}_n := \int_{\mathbb{R}^d} z \hat{g}_n(z) \, dz$ and $\hat{\Sigma}_n := \int_{\mathbb{R}^d} (z-\hat{\mu}_n)(z-\hat{\mu}_n)^T \hat{g}_n(z) \, dz$ respectively denote the mean vector and covariance matrix corresponding to $\hat{g}_n$.  Then by Lemma~\ref{Lemma:ThreeTerms} in Section~\ref{Sec:ProofMain} below, there exists $\eta_d \in (0,1)$ and $n_0 \in \mathbb{N}$, depending only on $d$, such that
\[
\sup_{g_0 \in \mathcal{F}_d^{0,I}} \mathbb{P}_{g_0}\bigl(\hat{g}_n \notin \tilde{\mathcal{F}}_d^{1,\eta_d}\bigr) \leq \frac{1}{n^{4/5}}
\]
for $n \geq n_0$.
%Then by the remark following Proposition~\ref{Prop:Conv3} and the strong consistency of $\hat{g}_n$ \citep[e.g.][Theorem~4]{CuleSamworth2010}, it follows that $\hat{\mu}_n \stackrel{a.s.}{\rightarrow} 0$ and $\hat{\Sigma}_n \stackrel{a.s.}{\rightarrow} I$.  We deduce that there exists an event $\Omega_n$, with $\mathbb{P}(\liminf \Omega_n) = 1$, such that on $\Omega_n$ we have $\hat{g}_n \in \tilde{\mathcal{F}}_d^{1,\eta_d}$.  In particular, given $\epsilon > 0$, there exists $n_0 \in \mathbb{N}$ such that $\mathbb{P}(\Omega_n^c) \leq \epsilon/2$ for all $n \geq n_0$.

We can now apply Theorem~\ref{Thm:vdG} in Section~\ref{Sec:ProofMain}, which provides an exponential tail inequality controlling the performance of a maximum likelihood estimator in Hellinger distance in terms of a bracketing entropy integral.  It is an immediate consequence of Theorem~7.4 of \citet{vandeGeer2000}, although our notation is slightly different (in particular her definition of Hellinger distance is normalised with a factor of $1/\sqrt{2}$) and we have used the fact (apparent from her proofs) that, in her notation, we may take $C = 2^{13/2}$.

In Theorem~\ref{Thm:vdG}, we take $\bar{\mathcal{F}} := \bigl\{\frac{\tilde{f}+g_0}{2}:\tilde{f} \in \tilde{\mathcal{F}}_d^{1,\eta_d}\bigr\}$.  Note that if $[f^L,f^U]$ are elements of a bracketing set for $\tilde{\mathcal{F}}_d^{1,\eta_d}$, and we set $\bar{f}^L := \frac{f^L + g_0}{2}$ and $\bar{f}^U := \frac{f^U + g_0}{2}$, then
\[
h^2(\bar{f}^U,\bar{f}^L) = \frac{1}{2} \int_{\mathbb{R}^d} \{(f^U + g_0)^{1/2} - (f^L + g_0)^{1/2}\}^2 \leq  \frac{1}{2}h^2(f^U,f^L).
\]
It follows from this and our bracketing entropy bound (Theorem~\ref{Thm:BracketingBounds}) that 
\[
\log N_{[]}(u,\bar{\mathcal{F}},h) \leq \log N_{[]}(2^{1/2}u,\tilde{\mathcal{F}}_d^{1,\eta_d},h) \leq \left\{ \begin{array}{ll} 2^{-1/4}\overline{K}_1 u^{-1/2} & \mbox{for $d=1$} \\
2^{-1/2}\overline{K}_2u^{-1}\log_{++}^{3/2}(1/u) & \mbox{for $d = 2$} \\
2^{-1}\overline{K}_3u^{-2}& \mbox{for $d = 3$.}
\end{array} \right.   
\]
We now consider three different cases, assuming throughout that $n \geq d+1$ so that, with probability 1, the log-concave maximum likelihood estimator exists and is unique.
\begin{enumerate}
\item For $d=1$, we set $\delta_n := 2^{-1/2}M_1^{1/2}n^{-2/5}$, where $M_1 := \max\bigl\{\bigl(\frac{2^{37/2}}{3}\bigr)^{8/5}\overline{K}_1^{4/5},2^{33}\bigr\}$.  Then
\[
\int_{\delta_n^2/2^{13}}^{\delta_n} \sqrt{\log N_{[]}(u,\bar{\mathcal{F}},h)} \, du \leq \frac{4}{2^{1/2}3}\overline{K}_1^{1/2} M_1^{3/8} n^{-3/10} \leq 2^{-16}n^{1/2}\delta_n^2.
\]
Moreover, $\delta_n \leq 2^{-17}M_1 n^{-3/10} = 2^{-16}n^{1/2}\delta_n^2$.  
%Thus, if 
%\[
%n_{0,1} := \max\biggl[(\epsilon_0)^{-5/2}M_1^{-5/4} \, , \, \frac{2^{140}}{M_1^5}\biggl\{\log(2/\epsilon) + \frac{15}{2}\log 2\biggr\}^5\biggr],
%n_{0,1} := \epsilon_0^{-5/2} M_1^{5/4}
%\]
%then $\delta_n \in (0,2^{-1/2}\epsilon_0]$ for $n \geq n_{0,1}$.  
We conclude by Theorem~\ref{Thm:vdG} that for $t \geq M_1$, 
\begin{align*}
\sup_{g_0 \in \mathcal{F}_d^{0,I}} \mathbb{P}_{g_0}\bigl[\bigl\{n^{4/5}h^2(\hat{g}_n,g_0) \geq t\bigr\} \cap \bigl\{\hat{g}_n \in \tilde{\mathcal{F}}_d^{1,\eta_d}\bigr\}\bigr] &\leq 2^{13/2}\sum_{s=0}^\infty \exp\biggl(-\frac{2^{2s}tn^{1/5}}{2^{28}}\biggr) \\
&\leq 2^{15/2}\exp\biggl(-\frac{tn^{1/5}}{2^{28}}\biggr),
\end{align*}
where the final bound follows because $tn^{1/5}/2^{28} \geq \log 2$.
\item For $d = 2$, we set $\delta_n := 2^{-1/2}M_2^{1/2}n^{-1/3}\log^{1/2} n$, where $M_2 := \max\bigl\{2^{23}\overline{K}_2^{2/3}5^{4/3}/3,2^{33}\bigr\}$.  Let $n_{0,2}$ be large enough that $\delta_n \leq 1/e$ for $n \geq n_{0,2}$.  Then, for such $n$,
\begin{align*}
&\int_{\delta_n^2/2^{13}}^{\delta_n} \sqrt{\log N_{[]}(u,\bar{\mathcal{F}},h)} \, du \leq 2^{-1/4}\overline{K}_2^{1/2} \int_0^{\delta_n} u^{-1/2}\log^{3/4}(1/u) \, du \\
&=2^{-1/4}\overline{K}_2^{1/2}\int_{\log(1/\delta_n)}^\infty \! \! \! \! \! \! s^{3/4}e^{-s/2} \, ds = 2^{-1/4}\overline{K}_2^{1/2}\biggl\{2\delta_n^{1/2}\log^{3/4}\Bigl(\frac{1}{\delta_n}\Bigr) + \frac{3}{2}\int_{\log(1/\delta_n)}^\infty  \! \! \! \! \! \! s^{-1/4}e^{-s/2} \, ds\biggr\} \\
&\leq 2^{-1/4}\overline{K}_2^{1/2}5\delta_n^{1/2}\log^{3/4}(1/\delta_n) \leq  2^{1/2}3^{-3/4}\overline{K}_2^{1/2}5\delta_n^{1/2}\log^{3/4} n \leq 2^{-16}n^{1/2}\delta_n^2.
\end{align*}
where we have used the fact that $2^{1/2}M_2^{-1/2}\log^{-1/2}n \leq n^{1/3}$ in the penultimate inequality.  We conclude that for $n \geq n_{0,2}$ and $t \geq M_2$, we have 
\[
\sup_{g_0 \in \mathcal{F}_d^{0,I}} \mathbb{P}_{g_0}\biggl[\biggl\{\frac{n^{2/3}}{\log n}h^2(\hat{g}_n,g_0) \geq t\biggr\} \cap \bigl\{\hat{g}_n \in \tilde{\mathcal{F}}_d^{1,\eta_d}\bigr\}\bigr] \leq 2^{15/2}\exp\biggl(-\frac{tn^{1/3}\log n}{2^{28}}\biggr).
\]
\item For $d = 3$, the entropy integral diverges as $\delta \searrow 0$, so we cannot bound the bracketing entropy integral by replacing the lower limit with zero.  Nevertheless, we can set $\delta_n := 2^{-1/2}M_3^{1/2}n^{-1/4}\log^{1/2}n$, where $M_3 := \bigl\{2^{33/2}10\overline{K}_3^{1/2},2^{33}\bigr\}$.  For $t \geq M_3$, we have 
\[
\sup_{g_0 \in \mathcal{F}_d^{0,I}} \mathbb{P}_{g_0}\biggl[\biggl\{\frac{n^{1/2}}{\log n} h^2(\hat{g}_n,g_0) \geq t\biggr\} \cap \bigl\{\hat{g}_n \in \tilde{\mathcal{F}}_d^{1,\eta_d}\bigr\}\biggr] \leq 2^{15/2}\exp\biggl(-\frac{tn^{1/2}\log n}{2^{28}}\biggr).
\]
\end{enumerate}
Let $\rho_{n,1}^2 := n^{4/5}$, $\rho_{n,2}^2 := n^{2/3}(\log n)^{-1}$ and $\rho_{n,3}^2 := n^{1/2}(\log n)^{-1}$.  We conclude that if $n \geq \max(n_0,d+1)$ (and also $n \geq n_{0,2}$ when $d=2$), then
\begin{align*}
%\mathbb{P}\bigl\{\rho_{n,d}^2 h^2(\hat{g}_n,g_0) \geq M_d\bigr\} \leq \mathbb{P}\bigl[\bigl\{\rho_{n,d}^2 h^2(\hat{g}_n,g_0) \geq M_d
%\bigr\} \cap \bigl\{\hat{g}_n \in \tilde{\mathcal{F}}_d^{1,\eta_d}\bigr\}\bigr] + \mathbb{P}(\Omega_n^c) \leq \epsilon,
&\rho_{n,d}^2 \sup_{f_0 \in \mathcal{F}_d} \mathbb{E}_{f_0}\{h^2(\hat{f}_n,f_0)\} = \rho_{n,d}^2 \sup_{g_0 \in \mathcal{F}_d^{0,I}} \mathbb{E}_{g_0}\{h^2(\hat{g}_n,g_0)\} \\
&\leq \sup_{g_0 \in \mathcal{F}_d^{0,I}} \int_0^\infty \mathbb{P}_{g_0}\bigl[\bigl\{\rho_{n,d}^2 h^2(\hat{g}_n,g_0) \geq t\} \cap \bigl\{\hat{g}_n \in \tilde{\mathcal{F}}_d^{1,\eta_d}\bigr\}\bigr] \, dt + 2\rho_{n,d}^2 \sup_{g_0 \in \mathcal{F}_d^{0,I}}\mathbb{P}_{g_0}(\hat{g}_n \notin \tilde{\mathcal{F}}_d^{1,\eta_d}) \\
&\leq M_d + 2^{71/2} + 2,
\end{align*}
as required.
\end{proof}

\subsection{Auxiliary results}

\subsubsection{Auxiliary results for the proof of Theorem~\ref{Thm:LowerBound}}

The following lemma is an immediate consequence of Assouad's lemma as stated in, e.g.\ \citet[][p.~347]{vanderVaart1998} or \citet[][pp.~118--9]{Tsybakov2009}.
\begin{lemma}
\label{Lemma:Kim}
Suppose that the loss function $L$ belongs to the set $\{L_1^2,L_2^2,h^2\}$.  Let $K \in \mathbb{N}$, and suppose that $\{f_\alpha: \alpha \in \{0,1\}^K\}$ is a subset of $\mathcal{F}_d$ with the following two properties:
\begin{enumerate}[(i)]
\item There exists $\gamma > 0$ such that 
\[
L(f_\alpha,f_\beta) \geq \gamma \|\alpha - \beta\|_0
\]
for all $\alpha, \beta \in \{0,1\}^K$, where $\|\alpha - \beta\|_0$ denotes the Hamming distance between $\alpha$ and $\beta$
\item There exists $C \in (0,1)$ such that for every $\alpha, \beta \in \{0,1\}^K$ with $\|\alpha - \beta\|_0 = 1$, we have
\begin{equation}
\label{Eq:h2}
h^2(f_\alpha,f_\beta) \leq \frac{C}{n}.
\end{equation}
\end{enumerate} 
Then
\[
\inf_{\tilde{f}_n \in \tilde{\mathcal{F}}_n} \sup_{f \in \mathcal{F}_d} \mathbb{E}_f\{L(\tilde{f}_n,f)\} \geq \frac{K}{8}(1-C^{1/2})\gamma.
\]
\end{lemma}
For completeness, we now give lower and upper bounds on the packing number of the unit Euclidean sphere $\mathcal{S}_1 := \bar{B}_d(0,1) \setminus B_d(0,1)$; the lower bound was used in the proof of Theorem~\ref{Thm:LowerBound} in Section~\ref{Sec:ProofLowerBound} (cf. also the proof of Theorem~\ref{Thm:BracketingBounds} in Section~\ref{Sec:ProofBracketingBounds}).  Similar results can be found in, e.g., \citet{Guntuboyina2012}.  Let $d \geq 2$, and for $\epsilon > 0$, let $N_\epsilon$ denote the packing number with respect to Euclidean distance of $\mathcal{S}_1$; thus $N_\epsilon$ is the maximal $N \in \mathbb{N}$ such that there exist $x_1,\ldots,x_N \in \mathcal{S}_1$ with $\|x_j - x_k\| > \epsilon$ for all $j \neq k$.  
\begin{lemma}
\label{Lemma:PackingSet}
Let $d \geq 2$.  For any $\epsilon \in (0,1/2]$, we have
\begin{align*}
\frac{(2\pi)^{1/2}(d-1)^{1/2}}{3^{1/2}2^{d-1}}\epsilon^{-(d-1)} &\leq \frac{(2\pi)^{1/2}(d-1)^{1/2}\{1 - (4\epsilon^2 - 4\epsilon^4)\}^{1/2}}{2^{d-1}(1-\epsilon^2)^{(d-1)/2}}\epsilon^{-(d-1)} \leq N_{2\epsilon} \\
&\leq \frac{\pi (d-1)^{1/2}}{(1-\epsilon^2/4)^{(d-1)/2}} \epsilon^{-(d-1)} \leq \frac{4^{d-1}\pi (d-1)^{1/2}}{15^{(d-1)/2}}\epsilon^{-(d-1)}.
\end{align*}
\end{lemma}
\begin{proof}
Let $x_1,\ldots,x_{N_{2\epsilon}}$ denote a packing set of $\mathcal{S}_1$ at distance $2\epsilon$.  For $j = 1,\ldots,N_{2\epsilon}$, define the hyperplane $\mathcal{H}_j := \{x \in \mathbb{R}^d:(x_j)^T x = 1-\epsilon^2/2\}$, and let 
\[
\tilde{x}_j := \argmin_{x \in \mathcal{H}_j} \|x\| = (1-\epsilon^2/2)x_j.
\]
Notice that for any $x \in \mathcal{H}_j \cap \mathcal{S}_1$, we have 
\begin{align}
\label{Eq:Disjoint}
\|x - x_j\|^2 &= \|x - \tilde{x}_j\|^2 + \epsilon^4/4 \nonumber \\
&= \|x\|^2 - 2(1-\epsilon^2/2)(x_j)^Tx + \|\tilde{x}_j\|^2 + \epsilon^4/4 = \epsilon^2.
\end{align}
Let $\mathcal{H}_j^+$ and $\mathcal{H}_j^-$ denote the disjoint, open halfspaces separated by $\mathcal{H}_j$, where $\mathcal{H}_j^-$ contains the origin in $\mathbb{R}^d$, and let $\mathcal{C}_j := \mathcal{H}_j^+ \cap \mathcal{S}_1$ denote the corresponding spherical cap.  Then, by~(\ref{Eq:Disjoint}), $\mathcal{C}_1,\ldots,\mathcal{C}_{N_{2\epsilon}}$ are disjoint.  Comparing the surface areas of $\cup_{j=1}^{N_{2\epsilon}} \mathcal{C}_j$ and $\mathcal{S}_1$, we deduce that
\[
N_{2\epsilon} \int_0^{\epsilon^2 - \epsilon^4/4} t^{\frac{d-1}{2} -1}(1-t)^{-1/2} \, dt \leq 2B\Bigl(\frac{d-1}{2},\frac{1}{2}\Bigr)
\]
where $B(\frac{d-1}{2},\frac{1}{2}) := \int_0^1 t^{\frac{d-1}{2}-1}(1-t)^{-1/2} \, dt$ denotes the beta function at $(\frac{d-1}{2},\frac{1}{2})$.  Since $B(\frac{d-1}{2},\frac{1}{2}) \leq \pi (d-1)^{-1/2}$ and $(1-t)^{-1/2} \geq 1$ for $t \in [0,1)$, the upper bound for $N_{2 \epsilon}$ follows.
 
For the lower bound, observe that for any $x \in \mathcal{S}_1$, we can find $j^* \in \{1,\ldots,N_{2\epsilon}\}$ such that $\|x - x_{j^*}\| \leq 2\epsilon$.  Thus, if for $j = 1,\ldots,N_{2\epsilon}$, we let 
\[
\tilde{\mathcal{C}}_j := \{x \in \mathcal{S}_1:\|x - x_j\| \leq 2 \epsilon\},
\]
then $\cup_{j=1}^{N_{2\epsilon}} \tilde{\mathcal{C}}_j = \mathcal{S}_1$.  We deduce that
\[
N_{2\epsilon}\int_0^{4\epsilon^2 - 4\epsilon^4} t^{\frac{d-1}{2} -1}(1-t)^{-1/2} \, dt \geq 2B\Bigl(\frac{d-1}{2},\frac{1}{2}\Bigr).
\]
Since $B(\frac{d-1}{2},\frac{1}{2}) \geq (2\pi)^{1/2}(d-1)^{-1/2}$ and $(1-t)^{-1/2} \leq \{1 - (4\epsilon^2 - 4\epsilon^4)\}^{-1/2}$ for $t \in [0,4\epsilon^2 - 4\epsilon^4]$, the lower bound follows.
\end{proof}

\subsubsection{Auxiliary results for the proof of Theorem~\ref{Thm:BracketingBounds}}

We first provide the following entropy bound for convex sets, which is a minor extension of \citet[][Corollary~8.4.2]{Dudley1999}.  For a $d$-dimensional, closed, convex set $D \subseteq \mathbb{R}^d$, we write $\mathcal{A}_d(D)$ for the class of closed, convex subsets of $D$.  Further, and in a slight abuse of notation, we let $N_{[]}(\epsilon,\mathcal{A}_d(D),L_1)$ denote the $\epsilon$-bracketing number of $\{\mathbbm{1}_A:A \in \mathcal{A}_d(D)\}$ in the $L_1 = L_1(\mu_d)$-metric.  Recall also that we write $\log_{++}(x) = \max(1,\log x)$.
\begin{prop}
\label{Prop:Dudley}
For each $d \in \mathbb{N}$, there exists $K_d \in (0,\infty)$, depending only on $d$, such that 
\[
\log N_{[]}\bigl(\epsilon,\mathcal{A}_d(D),L_1\bigr) \leq K_d\max\biggl\{\log_{++}\Bigl(\frac{\mu_d(D)}{\epsilon}\Bigr),\Bigl(\frac{\mu_d(D)}{\epsilon}\Bigr)^{(d-1)/2}\biggr\}
\]
for all $\epsilon > 0$.
\end{prop}
\begin{proof}
By Fritz John's theorem \citep[][p.~13]{John1948,Ball1997}, there exist $A \in \mathbb{R}^{d \times d}$ and $b \in \mathbb{R}^d$ such that $D' := AD+b$ has the property that $d^{-1}\bar{B}_d(0,1) \subseteq D' \subseteq \bar{B}_d(0,1)$.  Let $a_d := \mu_d\bigl(\bar{B}_d(0,1)\bigr) = \pi^{d/2}/\Gamma(1+d/2)$.  Now, by \citet[][Corollary~8.4.2]{Dudley1999} and the remark immediately preceding it, there exists $\epsilon_{20,d} \in \bigl(0,\min(e^{-1},a_d)\bigr)$ and $\check{\check{K}}_d \in (0,\infty)$ such that 
\[
\log N_{[]}\bigl(\epsilon,\mathcal{A}_d(D'),L_1\bigr) \leq \log N_{[]}\bigl(\epsilon,\mathcal{A}_d(\bar{B}_d(0,1)),L_1\bigr) \leq
\check{\check{K}}_d\max\{\log(1/\epsilon),\epsilon^{-(d-1)/2}\}
\]
for all $\epsilon \in (0,\epsilon_{20,d}]$.  Now set 
\[
\check{K}_d := \check{\check{K}}_d\frac{\max\{\log(1/\epsilon_{20,d}),\epsilon_{20,d}^{-(d-1)/2}\}}{\max\{\log_{++}(1/a_d),a_d^{-(d-1)/2}\}}.  
\]
Then, for $\epsilon \in (\epsilon_{20,d},a_d)$,
\begin{align*}
\log N_{[]}\bigl(\epsilon,\mathcal{A}_d(D'),L_1\bigr) &\leq \log N_{[]}\bigl(\epsilon_{20,d},\mathcal{A}_d(D'),L_1\bigr) \leq \check{\check{K}}_d\max\{\log(1/\epsilon_{20,d}),\epsilon_{20,d}^{-(d-1)/2}\} \\
&= \check{K}_d\max\{\log_{++}(1/a_d),a_d^{-(d-1)/2}\} \leq \check{K}_d\max\{\log_{++}(1/\epsilon),\epsilon^{-(d-1)/2}\}.
\end{align*}
For $\epsilon \geq a_d$, we can use the single bracketing pair $\{\psi^L,\psi^U\}$ with $\psi^L(x) := 0$ and $\psi^U(x) := 1$ for $x \in D'$, noting that $L_1(\psi^U,\psi^L) = \mu_d(D') \leq a_d$.  Thus, for $\epsilon \geq a_d$,  
\[
\log N_{[]}\bigl(\epsilon,\mathcal{A}_d(D'),L_1\bigr) = 0 \leq \check{K}_d\max\{\log_{++}(1/\epsilon),\epsilon^{-(d-1)/2}\}.
\]
We can therefore construct an $\epsilon$-bracketing set in $L_1$ for $\{\mathbbm{1}_A:A \in \mathcal{A}_d(D)\}$ as follows: first find an $\frac{\epsilon a_d}{d^d\mu_d(D)}$-bracketing set $\{[\psi_j^L,\psi_j^U]:j=1,\ldots,N\}$ for $\{\mathbbm{1}_A:A \in \mathcal{A}_d(D')\}$, where 
\[
\log N \leq \check{K}_d\max\biggl\{\log_{++}\Bigl(\frac{d^d\mu_d(D)}{\epsilon a_d}\Bigr),\Bigl(\frac{d^d\mu_d(D)}{\epsilon a_d}\Bigr)^{(d-1)/2}\biggr\}. 
\]
Now define $\phi_j^L,\phi_j^U:D \rightarrow \mathbb{R}$ by $\phi_j^L(x) := \psi_j^L(Ax+b)$ and $\phi_j^U(x) := \psi_j^U(Ax+b)$.  Then
\begin{align*}
L_1(\phi_j^U,\phi_j^L) &= \int_D |\psi_j^U(Ax+b) - \psi_j^L(Ax+b)| \, d\mu_d(x) \\
&\leq \frac{\epsilon a_d}{|\det A|d^d\mu_d(D)} = \frac{\epsilon a_d}{d^d\mu_d(D')} \leq \frac{\epsilon a_d}{d^d\mu_d\bigl(d^{-1}\bar{B}_d(0,1)\bigr)} = \epsilon.  
\end{align*}
Since $\log_{++}(a/\epsilon) \leq \bigl\{2 + \frac{2\log_{++}(a)}{\log_{++}(e/a)}\bigr\}\log_{++}(1/\epsilon)$ for all $a,\epsilon > 0$, the result therefore holds with
\[
K_d := \check{K}_d\max\biggl\{\biggl(2 + \frac{2\log_{++}(d^d/a_d)}{\log_{++}(ea_d/d^d)}\biggr) \, , \, \frac{d^{d(d-1)/2}}{a_d^{(d-1)/2}}\biggr\}.
\]
\end{proof}
%Next, we extend Theorem~3.2 of \citet{GuntuboyinaSen2013}, which gives a metric (as opposed to bracketing) entropy bound for uniformly bounded classes of concave functions on rectangles in $\mathbb{R}^d$.  First, we show in Proposition~\ref{Prop:RectDomain} below that an analogous bound holds for bracketing entropy; then, in Theorem~\ref{Thm:Simplices}, we provide a similar bracketing entropy bound for cases where the domain is a $d$-dimensional simplex.  Finally, in Theorem~\ref{Thm:}, we show how these results can be applied to 
We now provide a bracketing entropy bound for classes of uniformly bounded concave functions on arbitrary domains in $[0,1]^d$ when $d = 1,2,3$.  These results build on the work of \citet{GuntuboyinaSen2013}, who study metric (as opposed to bracketing) entropy and rectangular domains, and a recent result of \citet{GaoWellner2015}, who study various special classes of domains, including $d$-dimensional simplices.  For convenience, we state the result to which we will appeal below.
%It is convenient to define subclasses of concave functions in $\Phi_B$ (defined at the beginning of the proof of Theorem~\ref{Thm:BracketingBounds}) satisfying coordinate-wise Lipschitz conditions on their domains.  For a convex subset $D$ of $[0,1]^d$ and for $\Gamma_1,\ldots,\Gamma_d \in [0,\infty]$, let 
%\begin{align*}
%\Phi_B(D;\Gamma_1,\ldots,\Gamma_d) = \{\phi \in \Phi_B(D)&: |\phi(x) - \phi(x')| \leq \Gamma_j |x_j - x_j'| \text{ for all } j = 1,\ldots,d \text{ and } \\
%&x = (x_1,\ldots,x_d), x' = (x_1,\ldots,x_{j-1},x_j',x_{j+1},\ldots,x_d) \in D\}.
%\end{align*}
%Let $N(\epsilon,\Phi_B(D;\Gamma_1,\ldots,\Gamma_d),L_\infty)$ denote the $\epsilon$-covering number of $\Phi_B(D;\Gamma_1,\ldots,\Gamma_d)$ in the supremum metric $L_\infty$; thus  $N(\epsilon,\Phi_B(D;\Gamma_1,\ldots,\Gamma_d),L_\infty)$ is the smallest $N \in \mathbb{N}$ such that there exist $\psi_1,\ldots,\psi_N:D \rightarrow [-B,B]$ with the property that for each $\phi \in \Phi_B(D;\Gamma_1,\ldots,\Gamma_d)$, there exists $j^* \in \{1,\ldots,N\}$ such that $L_\infty(\phi,\psi_{j^*}) \leq \epsilon$. 
%\begin{prop}
%\label{Prop:Linfty}
%There exist $\epsilon_2 \in (0,1]$ and $K^* \in (0,\infty)$, each depending only on $d$, such that
%\[
%\log N(\epsilon,\Phi_B(D;\Gamma_1,\ldots,\Gamma_d),L_\infty) \leq K^*\biggl(\frac{B + \sum_{j=1}^d \Gamma_j}{\epsilon}\biggr)^{d/2} 
%\]
%for $\epsilon \in \bigl(0,\epsilon_2\{B + \sum_{j=1}^d \Gamma_j\}\bigr]$.
%\end{prop}
%\begin{proof}
%The result follows using small modifications to the proof of Theorem~3.2 of \citet{GuntuboyinaSen2013}.  The details are omitted.
%\end{proof}

Recall that we say $\mathcal{S} \subseteq \mathbb{R}^d$ is a \emph{$d$-dimensional simplex} if there exist affinely independent vectors $u_0,u_1,\ldots,u_d \in \mathbb{R}^d$ such that
\[
\mathcal{S} = \biggl\{u_0 + \sum_{j=1}^d \lambda_j u_j:\lambda_1,\ldots,\lambda_d \geq 0, \, \sum_{j=1}^d \lambda_j \leq 1\biggr\}.
\]
A set $D \subseteq \mathbb{R}^d$ can be \emph{triangulated into simplices} if there exist $d$-dimensional simplices $S_1,\ldots,S_N \subseteq D$ such that $\cup_{j=1}^N S_j = D$ and if $j \neq k$ then there is a common (possibly empty) face $F$ of the boundaries of $S_j$ and $S_k$ with $S_j \cap S_k = F$.  For a $d$-dimensional, closed, convex subset $D$ of $\mathbb{R}^d$, and for $B > 0$, we define $\bar{\Phi}_B(D)$ to be the set of upper semi-continuous, concave functions $\phi$ with $\mathrm{dom}(\phi) = D$ that are bounded in absolute value by $B$. 
\begin{thm}[\citet{GaoWellner2015}, Theorem~1.1(ii)]
\label{Thm:Simplices}
For each $d \in \mathbb{N}$, there exists $K_d^{**} \in (0,\infty)$, depending only on $d$, such that if $D$ is a $d$-dimensional closed, convex subset of $\mathbb{R}^d$ that can be triangulated into $m$ simplices, then
\[
\log N_{[]}\bigl(2\epsilon,\bar{\Phi}_B(D),L_2\bigr) \leq K_d^{**}m\biggl(\frac{B\mu_d^{1/2}(D)}{\epsilon}\biggr)^{d/2}
\]
for all $\epsilon > 0$.
\end{thm}
We also require one further preliminary lemma.  For any $d$-dimensional, compact, convex set $D \subseteq \mathbb{R}^d$ and any $\eta \geq 0$, let
\[
\tensor*[_{\eta}]{D}{} := \{x \in D: w \in D \text{ for all } \|w - x\| \leq \eta\}, \quad \text{and} \quad D^{\eta]} := D + \eta \bar{B}_d(0,1).
\]
Some basic properties of the sets $\tensor*[_{\eta}]{D}{}$ and $D^{\eta]}$ are given below.
\begin{lemma}
\label{Lemma:InnerOuter}
Let $D$, $\tensor*[_{\eta}]{D}{}$ and $D^{\eta]}$ be as above.  Then
\begin{enumerate}[(i)]
\item $\tensor*[_{\eta}]{D}{}$ and $D^{\eta]}$ are compact and convex.
\item If $0 \leq \eta_1 \leq \eta_2$, then $(\tensor*[_{\eta_1}]{D}{})^{\eta_2]} \subseteq D^{(\eta_2 - \eta_1)]}$ and $\tensor*[_{\eta_1}]{{(D^{\eta_2]})}}{} = D^{(\eta_2 - \eta_1)]}$.
\item If $\eta_1, \eta_2 > 0$, then $\tensor*[_{\eta_1}]{{(\tensor*[_{\eta_2}]{D}{})}}{} = \tensor*[_{\eta_1 +\eta_2}]{D}{}$ and $(D^{\eta_1]})^{\eta_2]} = D^{(\eta_1 + \eta_2)]}$.
\item If, in addition, $D$ is a polyhedral convex set, so that we can write $D = \cap_{j=1}^m \{x : b_j^T x \leq \beta_j\}$ for some $m \in \mathbb{N}$, some distinct $b_1,\ldots,b_m \in \mathbb{R}^d$ with $\|b_j\| = 1$ for each $j$, and some $\beta_1,\ldots,\beta_m \in \mathbb{R}$, then $\tensor*[_{\eta}]{D}{} = \cap_{j=1}^m \{x : b_j^T x \leq \beta_j - \eta\}$.
\end{enumerate}
\end{lemma}
\begin{proof}
\emph{(i)} Certainly $\tensor*[_{\eta}]{D}{}$ is bounded because $\tensor*[_{\eta}]{D}{} \subseteq D$.  To show $\tensor*[_\eta]{D}{}$ is closed, let $(x_n) \in \tensor*[_\eta]{D}{}$ with $x_n \rightarrow x$, and suppose that $\|w - x\| \leq \eta$.  Then, setting $w_n := x_n + w - x$, we have $w_n \in D$ and $w_n \rightarrow w$, so $w \in D$ since $D$ is closed.  We conclude that $x \in \tensor*[_\eta]{D}{}$, as required.  To show $\tensor*[_{\eta}]{D}{}$ is convex, let $x_1, x_2 \in \tensor*[_{\eta}]{D}{}$ and $\lambda \in [0,1]$, and suppose that $\|w - \{(1-\lambda)x_1 + \lambda x_2\}\| \leq \eta$.  Define $w_1 := x_1 + w - (1-\lambda)x_1 - \lambda x_2 \in D$ and $w_2 := x_2 + w - (1-\lambda)x_1 - \lambda x_2 \in D$.  Then
\[
w = (1-\lambda)w_1 + \lambda w_2 \in D,
\]
so $(1-\lambda)x_1 + \lambda x_2 \in \tensor*[_{\eta}]{D}{}$, as required.  Thus $\tensor*[_{\eta}]{D}{}$ is compact and convex.

For the second part, $D^{\eta]}$ is bounded, because 
\[
\sup_{x \in D^{\eta]}} \|x\| = \sup_{y \in D,z \in \bar{B}_d(0,1)} \|y + \eta z\| \leq \sup_{y \in D} \|y\| + \eta < \infty.
\]
Now suppose that $(x_n)$ is a sequence in $D^{\eta]}$ with $x_n \rightarrow x$, so we can write $x_n = y_n + \eta z_n$, where $y_n \in D$ and $\|z_n\| \leq 1$.  Since $D$ and $\bar{B}_d(0,1)$ are compact, there exist $y \in D$, $z \in \bar{B}_d(0,1)$ and integers $1 \leq n_1 < n_2 < \ldots$ such that $y_{n_k} \rightarrow y$ and $z_{n_k} \rightarrow z$.  By uniqueness of limits, $x = y + \eta z$, so $x \in D^{\eta]}$, which shows that $D^{\eta]}$ is closed.  Finally, if $x_1,x_2 \in D^{\eta]}$ and $\lambda \in [0,1]$, then we can find $y_1,y_2 \in D$ and $z_1,z_2 \in \bar{B}_d(0,1)$ such that $x_1 = y_1 + \eta z_1$ and $x_2 = y_2 + \eta z_2$.  But then since $D$ is convex and $\|(1-\lambda)z_1 + \lambda z_2\| \leq (1-\lambda)\|z_1\| + \lambda \|z_2\| \leq 1$, we have
\[
(1-\lambda)x_1 + \lambda x_2 = (1-\lambda)y_1 + \lambda y_2 + \eta \{(1-\lambda)z_1 + \lambda z_2\} \in D + \eta \bar{B}_d(0,1),
\]
so $D^{\eta]}$ is convex.   

\emph{(ii)} Let $x_0 \in (\tensor*[_{\eta_1}]{D}{})^{\eta_2]}$.  If $x_0 \in D$, then certainly $x_0 \in D^{(\eta_2 - \eta_1)]}$, so assume $x_0 \notin D$.  Then there exists $y_0 \in \tensor*[_{\eta_1}]{D}{}$ such that $\eta_1 < \|x_0 - y_0\| \leq \eta_2$, and 
\[
w := y_0 + \eta_1 \frac{(x_0 - y_0)}{\|x_0 - y_0\|} \in D.
\]
Moreover,
\[
\|w - x_0\| = \biggl\|y_0 - x_0 - \eta_1 \frac{(y_0 - x_0)}{\|y_0 - x_0\|}\biggr\| = \|y_0 - x_0\| - \eta_1 \leq \eta_2 - \eta_1.
\]
Hence $x_0 \in D^{(\eta_2 - \eta_1)]}$, so $(\tensor*[_{\eta_1}]{D}{})^{\eta_2]} \subseteq D^{(\eta_2 - \eta_1)]}$.

For the second part, suppose that $x \in \tensor*[_{\eta_1}]{{(D^{\eta_2]})}}{}$.  If $x \in D$, then $x \in D^{(\eta_2 - \eta_1)]}$ and we are done; otherwise, let $z$ denote the orthogonal projection of $x$ onto $D$.  Writing
\[
y := x + \eta_1 \frac{(x-z)}{\|x-z\|} = z + (x-z)\frac{\|x-z\| + \eta_1}{\|x-z\|},
\]
we have that $\|y - x\| = \eta_1$, so $y \in D^{\eta_2]}$.  Moreover, for every $t \in D$,
\[
(y-z)^T(t-z) = \frac{\|x-z\| + \eta_1}{\|x-z\|}(x-z)^T(t-z) \leq 0,
\]
so $z$ is the orthogonal projection of $y$ onto $D$.  We deduce that $\|x-z\| + \eta_1 = \|y - z\| \leq \eta_2$, so $x \in D^{(\eta_2 - \eta_1)]}$.

Conversely, let $x \in D^{(\eta_2 - \eta_1)]}$.  Then there exists $z \in D$ such that $\|x - z\| \leq \eta_2 - \eta_1$.  If $\|y - x\| \leq \eta_1$, then
\[
\|y - z\| \leq \|y - x\| + \|x - z\| \leq \eta_2,
\]
so $y \in D^{\eta_2]}$.  Hence $x \in \tensor*[_{\eta_1}]{{(D^{\eta_2]})}}{}$, as required.

\emph{(iii)} Let $x \in \tensor*[_{\eta_1}]{{(\tensor*[_{\eta_2}]{D}{})}}{}$, and let $\|z - x\| \leq \eta_1 + \eta_2$.  If $\|z - x\| \leq \eta_1$, then $z \in \tensor*[_{\eta_2}]{D}{} \subseteq D$; otherwise, $\eta_1 < \|z-x\| \leq \eta_1 + \eta_2$.  In that case,
\[
y := x + \eta_1\frac{z-x}{\|z-x\|}
\]
satisfies $\|y- x\| \leq \eta_1$, so $y \in \tensor*[_{\eta_2}]{D}{}$.  But then $\|z - y\| = \|z - x\| - \eta_1 \leq \eta_2$, so $z \in D$.  Hence $x \in \tensor*[_{\eta_1 +\eta_2}]{D}{}$.

Conversely, suppose that $x \in \tensor*[_{\eta_1 +\eta_2}]{D}{}$ and that $\|y - x\| \leq \eta_1$.  If $\|z - y\| \leq \eta_2$, then $\|z - x\| \leq \eta_1 + \eta_2$, so $z \in D$.  Hence $y \in \tensor*[_{\eta_2}]{D}{}$ and $x \in \tensor*[_{\eta_1}]{{(\tensor*[_{\eta_2}]{D}{})}}{}$, as required.

For the second part, let $x \in (D^{\eta_1]})^{\eta_2]}$.  Then there exists $y \in D^{\eta_1]}$ such that $\|y - x\| \leq \eta_2$, and $z \in D$ such that $\|z - y\| \leq \eta_1$.  But then $\|z - x\| \leq \eta_1 + \eta_2$, so $x \in D^{(\eta_1 + \eta_2)]}$.  

Conversely, suppose that $x \in D^{(\eta_1 + \eta_2)]}$, so there exists $z \in D$ such that $\|z - x\| \leq \eta_1 + \eta_2$.  If $x \in D^{\eta_1]}$, then certainly $x \in (D^{\eta_1]})^{\eta_2]}$; otherwise, we have $\|z - x\| > \eta_1$, and can set
\[
y := z + \eta_1\frac{x-z}{\|x-z\|}.
\]
In that case, $\|y - z\| = \eta_1$, so $y \in D^{\eta_1]}$, and $\|x- y\| = \|x - z\| - \eta_1 \leq \eta_2$, so $x \in (D^{\eta_1]})^{\eta_2]}$, as required.

\emph{(iv)} If $x \in \tensor*[_{\eta}]{D}{}$, then for each $j = 1,\ldots,m$, we have $w_j := x + \eta b_j \in D$.  Thus for each $j$,
\[
\beta_j \geq b_j^T w_j = b_j^T(x + \eta b_j) = b_j^T x + \eta,
\]
so $x \in \cap_{j=1}^m \{x : b_j^T x \leq \beta_j - \eta\}$.

Conversely, if $x \in \cap_{j=1}^m \{x : b_j^T x \leq \beta_j - \eta\}$ and $\|z\| \leq 1$, then by Cauchy--Schwarz,
\[
b_j^T(x + \eta z) \leq b_j^T x + \eta \leq \beta_j,
\]
so $x \in \tensor*[_{\eta}]{D}{}$.
\end{proof}
We are now in a position to state our bracketing entropy bound.
\begin{prop}
\label{Prop:BoundedBrackets}
There exists $K_d^\circ \in (0,\infty)$, depending only on $d$, such that for all $d$-dimensional, convex, compact sets $D \subseteq \mathbb{R}^d$ and all $B, \epsilon > 0$, we have
\[
\log N_{[]}\bigl(2\epsilon,\bar{\Phi}_B(D),L_2\bigr) \leq \left\{ \begin{array}{ll} K_1^\circ \mu_1^{1/4}(D)(B/\epsilon)^{1/2} & \mbox{if $d=1$} \\
K_2^\circ \mu_2^{1/2}(D)(B/\epsilon)\log_{++}^{3/2}(B\mu_2^{1/2}(D)/\epsilon) & \mbox{if $d=2$} \\
K_3^\circ \mu_3(D)(B/\epsilon)^2 & \mbox{if $d=3$.} \end{array} \right.
\]
\end{prop}
\begin{proof}
As a preliminary, recall that the Hausdorff distance between two non-empty, compact subsets $A, B \subseteq \mathbb{R}^d$ is given by
\[
\mathrm{Haus}(A,B) := \max\biggl\{\sup_{x \in A} \inf_{y \in B} \|x-y\| \, , \, \sup_{y \in B} \inf_{x \in A} \|x - y\|\biggr\}.
\]
By the main result of \citet{BronshteynIvanov1975}, there exist $\delta_{\mathrm{BI},d} > 0$ and $C_d > 0$, both depending only on $d$, such that for every $\delta \in (0,\delta_{\mathrm{BI},d}]$ and every $d$-dimensional convex, compact set $D \subseteq \bar{B}_d(0,1)$, we can find a (convex) polytope $P \supseteq D$ such that $P$ has at most $C_d\delta^{-(d-1)/2}$ vertices and $\mathrm{Haus}(P,D) \leq \delta$.  (Throughout, we follow, e.g., \citet{Rockafellar1997}, and define a polytope to be a set formed as the convex hull of finitely many points.)  Moreover, by Lemma~8.4.3 of \citet{Dudley1999}, there exists $c_0 \in (0,16\delta_{\mathrm{BI},d}]$, depending only on $d$ (though this dependence is suppressed for notational simplicity), such that for any $d$-dimensional, closed convex set $D \subseteq \bar{B}_d(0,1)$ and any $\delta > 0$, we have $\mu_d(D \setminus \tensor*[_{c_0 \delta}]{D}{}) \leq \delta/16$. 

We now begin the main proof in the case $B=1$, and handle the general case at the end of the whole argument.  Fix a $d$-dimensional, convex, compact set $D \subseteq \mathbb{R}^d$, and, as in the proof of Proposition~\ref{Prop:Dudley}, apply Fritz John's theorem to construct an affine transformation $D' := AD+b$ of $D$ such that $d^{-1}\bar{B}_d(0,1) \subseteq D' \subseteq \bar{B}_d(0,1)$.  We initially find bracketing sets for $\bar{\Phi}_1(D')$, and consider different dimensions separately.

\emph{The case $d=1$}: This is an extension from metric to bracketing entropy of Theorem~3.1 of \citet{GuntuboyinaSen2013}, and can be found in \citet[][Proposition~4.1]{DossWellner2015}.  In particular, these authors show that there exist $\epsilon_1^\circ \in (0,1)$ and $K_{1,1}^\circ > 0$ such that, when $d=1$,
\[
\log N_{[]}\bigl(2\epsilon,\bar{\Phi}_1(D'),L_2\bigr) \leq K_{1,1}^\circ \epsilon^{-1/2} 
\]
for all $\epsilon \in (0,\epsilon_1^\circ]$.

\emph{The case $d=2$}: 
%By the main result of \citet{BronshteynIvanov1975}, there exists $\delta_{\mathrm{BI},2} > 0$ such that for every $\delta \in (0,\delta_{\mathrm{BI},2}]$, we can find a polytope $P_1 \supseteq \tensor*[_{\delta}]{{D'}}{}$ and a universal constant $C_2 \geq 1$ such that $P_1$ has at most $C_2\delta^{-1/2}$ vertices and $\mathrm{Haus}(P_1,\tensor*[_{\delta}]{{D'}}{}) \leq \delta$.  Now by Lemma~8.4.3 of \citet{Dudley1999}, there exists $c_{0,2} \in (0,4\delta_{\mathrm{BI},2}]$ such that for any $2$-dimensional, closed convex set $\tilde{D} \subseteq [0,1]^2$ and any $\delta > 0$, we have $\mu_2(\tilde{D} \setminus \tensor*[_{c_0 \delta}]{{\tilde{D}}}{}) \leq \delta/16$.  
Set $\epsilon_2^\circ := 1/8$, and fix $\epsilon \in (0,\epsilon_2^\circ]$, noting that $\mu_2(D' \setminus \tensor*[_{c_0 \epsilon^2}]{{D'}}{}) \leq \epsilon^2/16$.  Applying the result of \citet{BronshteynIvanov1975}, we can find a polytope $P_1 \supseteq \tensor*[_{c_0 \epsilon^2}]{{D'}}{}$ such that $P_1$ has at most $C_2c_0^{-1/2}\epsilon^{-1}$ vertices and $\mathrm{Haus}(P_1,\tensor*[_{c_0 \epsilon^2}]{{D'}}{}) \leq c_0\epsilon^2$.  From this and the first part of Lemma~\ref{Lemma:InnerOuter}(ii), we deduce that $P_1 \subseteq (\tensor*[_{c_0 \epsilon^2}]{{D'}}{})^{c_0 \epsilon^2]} \subseteq D'$.  Applying the result of \citet{BronshteynIvanov1975} recursively, with $M := \big\lfloor \log\bigl(\frac{1}{4\epsilon}\bigr)/\log 2\big\rfloor$ (the condition that $\epsilon \leq 1/8$ ensures that $M \in \mathbb{N}$), for each $i = 2,3,\ldots,M$, there exists a polytope $P_i \supseteq \tensor*[_{c_0 4^i\epsilon^2}]{{(P_{i-1})}}{}$ with at most $C_2c_0^{-1/2}2^{-i}\epsilon^{-1}$ vertices such that $\mathrm{Haus}\bigl(P_i,\tensor*[_{c_0 4^i\epsilon^2}]{{(P_{i-1})}}{}\bigr) \leq c_04^i\epsilon^2$.  Observe that the Bronshteyn--Ivanov result can be applied in each case, because for $i=2,3,\ldots,M$,
\[
c_04^i \epsilon^2 \leq c_0 4^M \epsilon^2 \leq \frac{c_0}{16} \leq \delta_{\mathrm{BI},2}.
\]
Note moreover that $P_i \subseteq P_{i-1}$.  We claim that $P_M$ is a two-dimensional polytope, by our choice of $M$.  In fact,
\begin{align*}
\mu_2(P_M) &= \mu_2(D') - \mu_2(D' \setminus P_1) - \sum_{i=2}^M \mu_2(P_{i-1} \setminus P_i) \\
&\geq \frac{\pi}{4} - \mu_2(D' \setminus P_1) - \sum_{i=2}^M \mu_2(P_{i-1} \setminus \tensor*[_{c_0 4^i\epsilon^2}]{{(P_{i-1})}}{}) \\
&\geq \frac{\pi}{4} - \frac{\epsilon^2}{16} \sum_{i=1}^M 4^i \geq \frac{\pi}{4} - 4^{M-1}\epsilon^2 \geq \frac{\pi}{8}.
\end{align*}
\begin{figure}
\begin{center}
\includegraphics[width=0.7\textwidth]{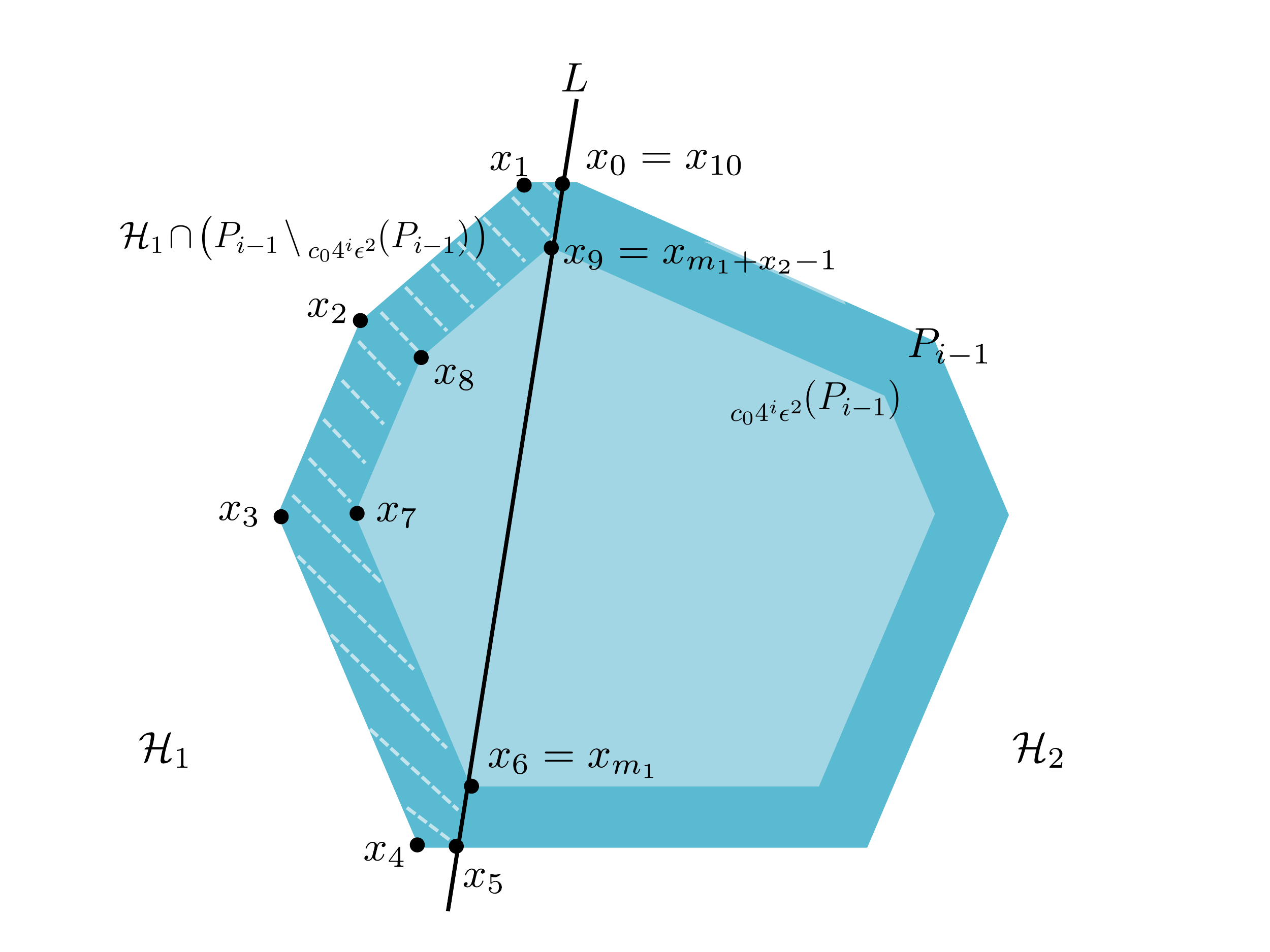}
\end{center}
\caption{\label{Fig:2D}Illustration of triangulation construction when $d=2$.}
\end{figure}

For $i = 2,3,\ldots,M$, we now describe how to construct a finite set of simplices (triangles) $S_{i,1},\ldots,S_{i,N_i}$ that cover $P_{i-1} \setminus \tensor*[_{c_0 4^i\epsilon^2}]{{(P_{i-1})}}{}$, so in particular, they cover $P_{i-1} \setminus P_i$.  Since $\tensor*[_{c_0 4^i\epsilon^2}]{{(P_{i-1})}}{}$ is a two-dimensional polyhedral convex set, we can pick two distinct vertices in this set.  The line $L$ passing through these two points forms the boundary of two closed halfspaces $\mathcal{H}_1$ and $\mathcal{H}_2$; we show how to triangulate $\mathcal{H}_1 \cap \bigl(P_{i-1} \setminus \tensor*[_{c_0 4^i\epsilon^2}]{{(P_{i-1})}}{}\bigr)$, with the triangulation of $\mathcal{H}_2 \cap \bigl(P_{i-1} \setminus \tensor*[_{c_0 4^i\epsilon^2}]{{(P_{i-1})}}{}\bigr)$ being entirely analogous.  We claim that, in the terminology of \citet{DevadossORourke2011}, $\mathcal{H}_1 \cap \bigl(P_{i-1} \setminus \tensor*[_{c_0 4^i\epsilon^2}]{{(P_{i-1})}}{}\bigr)$ is a polygon, i.e. a closed subset of $\mathbb{R}^2$ bounded by a finite collection of line segments forming a simple closed curve.    

To see this, observe that the line $L$ intersects $\mathrm{bd}(P_{i-1})$ at precisely two points; let $x_0 \in L \cap \mathrm{bd}(P_{i-1})$ denote the point that is larger in the lexicographic ordering (with respect to the standard Euclidean basis); see Figure~\ref{Fig:2D}.  Let $m_1 \in \mathbb{N}$ denote the number of vertices of $\mathcal{H}_1 \cap P_{i-1}$.  Now, for $j=1,\ldots,m_1-1$, let $x_j \in \mathcal{H}_1 \cap \mathrm{bd}(P_{i-1})$ denote the vertex of the polyhedral convex set $\mathcal{H}_1 \cap P_{i-1}$ that is the unique neighbour of $x_{j-1}$ not belonging to $\{x_0,\ldots,x_{j-1}\}$.  Note here that $x_{m_1-1}$ is the other point in $L \cap \mathrm{bd}(P_{i-1})$.  Let $x_{m_1}$ denote the closest point of $L \cap \tensor*[_{c_0 4^i\epsilon^2}]{{(P_{i-1})}}{}$ to $x_{m_1-1}$ (so the line segment joining $x_{m_1-1}$ and $x_{m_1}$ is a subset of $L$).  Let $m_2 \in \mathbb{N}$ denote the number of vertices of $\mathcal{H}_1 \cap \tensor*[_{c_0 4^i\epsilon^2}]{{(P_{i-1})}}{}$.  For $j=1,\ldots,m_2-1$, let $x_{m_1+j} \in \mathcal{H}_1 \cap \mathrm{bd}\bigl(\tensor*[_{c_0 4^i\epsilon^2}]{{(P_{i-1})}}{}\bigr)$ denote the vertex of the polyhedral convex set $\mathcal{H}_1 \cap \tensor*[_{c_0 4^i\epsilon^2}]{{(P_{i-1})}}{}$ that is the unique neighbour of $x_{m_1+j-1}$ not belonging to $\{x_{m_1},\ldots,x_{m_1+j-1}\}$.  Finally, let $x_{m_1+m_2} = x_0$.  Let $0 = t_0 < t_1 < \ldots < t_{m_1+m_2} = 1$.  The boundary of the set $\mathcal{H}_1 \cap \bigl(P_{i-1} \setminus \tensor*[_{c_0 4^i\epsilon^2}]{{(P_{i-1})}}{}\bigr)$ is parametrised by the closed curve $\gamma:[0,1] \rightarrow \mathbb{R}^2$ given by
\[
\gamma(t) := \Bigl(\frac{t_{j+1} - t}{t_{j+1}-t_j}\Bigr)x_j + \Bigl(\frac{t - t_j}{t_{j+1}-t_j}\Bigr)x_{j+1}
\]
for $t \in [t_j,t_{j+1}]$.  In fact, we claim that $\gamma$ is a simple closed curve.  To see this, note that $P_{i-1}$ and $\tensor*[_{c_0 4^i\epsilon^2}]{{(P_{i-1})}}{}$ are polyhedral convex sets in $\mathbb{R}^2$, so their (disjoint) boundaries are simple closed curves; $\gamma(t) \in \mathrm{bd}(P_{i-1})$ for $t \in [0,t_{m_1-1}]$ and $\gamma(t) \in \mathrm{bd}\bigl(\tensor*[_{c_0 4^i\epsilon^2}]{{(P_{i-1})}}{}\bigr)$ for $t \in [t_{m_1},t_{m_1+m_2-1}]$.  Moreover, $\gamma(t)$ belongs to the interior of the line segment joining $x_{m_1-1}$ and $x_{m_1}$ (and hence to the interior of $P_{i-1} \setminus \tensor*[_{c_0 4^i\epsilon^2}]{{(P_{i-1})}}{}$) for $t \in (t_{m_1-1},t_{m_1})$ and to the interior of the line segment joining $x_{m_1+m_2-1}$ and $x_{m_1+m_2}$ for $t \in (t_{m_1+m_2-1},t_{m_1+m_2})$; these two line segments are themselves disjoint.  This establishes that $\gamma$ is a simple closed curve, and hence that $\mathcal{H}_1 \cap \bigl(P_{i-1} \setminus \tensor*[_{c_0 4^i\epsilon^2}]{{(P_{i-1})}}{}\bigr)$ is a polygon.  Note, incidentally, that our reason for introducing the line $L$ was precisely to ensure this fact.  We can therefore apply Theorems~1.4 and~1.8 of \citet{DevadossORourke2011} to conclude that there exist simplices $S_{i,1},\ldots,S_{i,N_i}$ that triangulate $P_{i-1} \setminus \tensor*[_{c_0 4^i\epsilon^2}]{{(P_{i-1})}}{}$, where $N_i \leq 4C_2c_0^{-1/2}2^{-i}\epsilon^{-1}$.

For $i=2,3,\ldots,M$ and $j = 1,\ldots,N_i$, let
\[
\alpha_{i,j} := \frac{2^{1/2}}{M^{1/2}}\biggl(\frac{\mu_2(S_{i,j})}{\mu_2(P_{i-1} \setminus \tensor*[_{c_0 4^i\epsilon^2}]{{(P_{i-1})}}{})}\biggr)^{1/2}.
\]
By Theorem~\ref{Thm:Simplices}, there exists a bracketing set $\{[\phi_{i,j,\ell}^L,\phi_{i,j,\ell}^U]:\ell=1,\ldots,n_{i,j}\}$ for $\bar{\Phi}_1(S_{i,j})$, where $\log n_{i,j} \leq K_2^{**}\bigl(\frac{\mu_2^{1/2}(S_{i,j})}{\alpha_{i,j}\epsilon}\bigr)$, such that $L_2(\phi_{i,j,\ell}^U,\phi_{i,j,\ell}^L) \leq \alpha_{i,j}\epsilon$.  Moreover, by the same theorem, there exists a bracketing set $\{[\phi_{M+1,r}^L,\phi_{M+1,r}^U]:r=1,\ldots,n_{M+1}\}$ for $\bar{\Phi}_1(P_M)$, where $\log n_{M+1} \leq 8K_2^{**}C_2c_0^{-1/2}\bigl(\frac{\mu_2^{1/2}(P_M)}{\epsilon}\bigr)$, such that $L_2(\phi_{M+1,r}^U,\phi_{M+1,r}^L) \leq \epsilon$.  This last statement follows, because $2^{-M}\epsilon^{-1} \leq 8$. 

We can therefore define a bracketing set for $\bar{\Phi}_1(D')$ as follows: first, for $i=2,\ldots,M$ and $j=1,\ldots,N_i$, let
\[
\tilde{S}_{i,j} := S_{i,j} \setminus \biggl\{\biggl(\bigcup_{k=2}^{i-1}\bigcup_{m=1}^{N_k} S_{k,m}\biggr) \bigcup \biggl(\bigcup_{m=1}^{j-1} S_{i,m}\biggr)\biggr\} \quad \text{and} \quad \tilde{P}_M := P_M \setminus \bigcup_{k=2}^M \bigcup_{m=1}^{N_k} S_{k,m}.
\]
Now, for the array $\boldsymbol{\ell} = (\ell_{i,j})$ where $i \in \{2,\ldots,M\}$, $j \in \{1,\ldots,N_i\}$ and $\ell_{i,j} \in \{1,\ldots,n_{i,j}\}$, and for $r=1,\ldots,n_{M+1}$, let 
\begin{align}
\label{Eq:UpperLower}
\psi_{\boldsymbol{\ell},r}^U(x) &:= \mathbbm{1}_{\{x \in D' \setminus P_1\}} + \sum_{i=2}^M \sum_{j=1}^{N_i} \phi_{i,j,\ell_{i,j}}^U(x)\mathbbm{1}_{\{x \in \tilde{S}_{i,j}\}} + \phi_{M+1,r}^U(x)\mathbbm{1}_{\{x \in \tilde{P}_M\}}, \\
\psi_{\boldsymbol{\ell},r}^L(x) &:= -\mathbbm{1}_{\{x \in D' \setminus P_1\}} + \sum_{i=2}^M \sum_{j=1}^{N_i} \phi_{i,j,\ell_{i,j}}^L(x)\mathbbm{1}_{\{x \in \tilde{S}_{i,j}\}} + \phi_{M+1,r}^L(x)\mathbbm{1}_{\{x \in \tilde{P}_M\}}, \label{Eq:UpperLower2}
\end{align}
for $x \in D'$.  Observe that 
\begin{align*}
L_2^2(\psi_{\boldsymbol{\ell},r}^U,\psi_{\boldsymbol{\ell},r}^L) &\leq 4\mu_2(D' \setminus P_1) + \sum_{i=2}^M \sum_{j=1}^{N_i} L_2^2(\phi_{i,j,\ell_{i,j}}^U,\phi_{i,j,\ell_{i,j}}^L) + L_2^2(\phi_{M+1,r}^U,\phi_{M+1,r}^L) \\
&\leq 4\mu_2(D' \setminus \tensor*[_{c_0 \epsilon^2}]{{D'}}{}) + \epsilon^2 \sum_{i=2}^M \sum_{j=1}^{N_i} \alpha_{i,j}^2 + \epsilon^2 \leq 4\epsilon^2. 
\end{align*}
Moreover, the logarithm of the cardinality of the bracketing set is
\begin{align*}
\sum_{i=2}^M \sum_{j=1}^{N_i} \log n_{i,j} + \log n_{M+1} &\leq K_2^{**}\sum_{i=2}^M \sum_{j=1}^{N_i} \frac{\mu_2^{1/2}(S_{i,j})}{\alpha_{i,j}\epsilon} + \frac{8K_2^{**}C_2c_0^{-1/2}\mu_2^{1/2}(P_M)}{\epsilon} \\
&\leq \frac{K_2^{**}2^{-1/2}M^{1/2}}{\epsilon} \sum_{i=2}^M N_i \mu_2^{1/2}(P_{i-1} \setminus \tensor*[_{c_0 4^i \epsilon^2}]{{P_{i-1}}}{}) + \frac{16K_2^{**}C_2c_0^{-1/2}}{\epsilon} \\
&\leq \frac{K_2^{**}C_2c_0^{-1/2}M^{3/2}}{\epsilon} + \frac{16K_2^{**}C_2c_0^{-1/2}}{\epsilon} \leq \frac{32K_2^{**}C_2c_0^{-1/2}M^{3/2}}{\epsilon} \\
&\leq \frac{32K_2^{**}C_2c_0^{-1/2}}{\log^{3/2} 2} \epsilon^{-1}\log^{3/2}\Bigl(\frac{1}{4\epsilon}\Bigr).
\end{align*}
Defining $K_{1,2}^\circ := \frac{32K_2^{**}C_2}{\log^{3/2} 2}$, we have therefore proved that when $d=2$,
\[
\log N_{[]}\bigl(2\epsilon,\bar{\Phi}_1(D'),L_2\bigr) \leq K_{1,2}^\circ \epsilon^{-1}\log^{3/2}\Bigl(\frac{1}{4\epsilon}\Bigr) 
\]
for all $\epsilon \in (0,\epsilon_2^\circ]$.

\emph{The case $d=3$}: The proof is similar in spirit to the case $d=2$, so we emphasise the points of difference, and give fewer details where the argument is essentially the same.  

Set $\epsilon_3^\circ := 1/8$, and fix $\epsilon \in (0,\epsilon_3^\circ]$.  The Bronshteyn--Ivanov result once again yields a polytope $P_1$ with $\tensor*[_{c_0 \epsilon^2}]{{D'}}{} \subseteq P_1 \subseteq (\tensor*[_{c_0 \epsilon^2}]{{D'}}{})^{c_0 \epsilon^2]} \subseteq D'$ such that $P_1$ has at most $C_3c_0^{-1}\epsilon^{-2}$ vertices and $\mathrm{Haus}(P_1,\tensor*[_{c_0 \epsilon^2}]{{D'}}{}) \leq c_0\epsilon^2$.  Applying the result of \citet{BronshteynIvanov1975} recursively, with $M := \big\lfloor\log\bigl(\frac{1}{4\epsilon}\bigr)/\log 2 \big\rfloor$, for each $i = 2,3,\ldots,M$, there exists a polytope $\tensor*[_{c_0 4^i\epsilon^2}]{{(P_{i-1})}}{} \subseteq P_i \subseteq P_{i-1}$ with at most $C_3c_0^{-1}4^{-i}\epsilon^{-2}$ vertices such that $\mathrm{Haus}\bigl(P_i,\tensor*[_{c_0 4^i\epsilon^2}]{{(P_{i-1})}}{}\bigr) \leq c_04^i\epsilon^2$.  Again we claim that $P_M$ is a three-dimensional polytope, since
\[
\mu_3(P_M) = \mu_3(D') - \mu_3(D' \setminus P_1) - \sum_{i=2}^M \mu_3(P_{i-1} \setminus P_i) > 0.
\]
The construction of \citet{WangYang2000} (cf.\ also \citet{ChazelleShouraboura1995}) yields, for each $i=2,3,\ldots,M$, simplices $S_{i,1},\ldots,S_{i,N_i}$, where $N_i \leq 16C_3c_0^{-1}4^{-i}\epsilon^{-2}$ that triangulate $P_{i-1} \setminus \tensor*[_{c_0 4^i\epsilon^2}]{{(P_{i-1})}}{}$.  Set
\[
\alpha_{i,j} := \biggl(\frac{2^{-(i-2)/2}}{\sum_{k=2}^M 2^{-k/2}}\biggr)^{1/2}\biggl(\frac{\mu_3(S_{i,j})}{\mu_3(P_{i-1} \setminus \tensor*[_{c_0 4^i\epsilon^2}]{{(P_{i-1})}}{})}\biggr)^{1/2}. 
\]
Applying Theorem~\ref{Thm:Simplices} again, there exists a bracketing set $\{[\phi_{i,j,\ell}^L,\phi_{i,j,\ell}^U]:\ell=1,\ldots,n_{i,j}\}$ for $\bar{\Phi}_1(S_{i,j})$, where $\log n_{i,j} \leq K_3^{**}\bigl(\frac{\mu_3^{1/2}(S_{i,j})}{\alpha_{i,j}\epsilon}\bigr)^{3/2}$, such that $L_2(\phi_{i,j,\ell}^U,\phi_{i,j,\ell}^L) \leq \alpha_{i,j}\epsilon$.  Moreover, by the same theorem, there exists a bracketing set $\{[\phi_{M+1,r}^L,\phi_{M+1,r}^U]:r=1,\ldots,n_{M+1}\}$ for $\bar{\Phi}_1(P_M)$, where $\log n_{M+1} \leq 64C_3c_0^{-1}K_3^{**}\bigl(\frac{\mu_3^{1/2}(P_M)}{\epsilon}\bigr)^{3/2}$, such that $L_2(\phi_{M+1,r}^U,\phi_{M+1,r}^L) \leq \epsilon$.  

Defining brackets $\psi_{\boldsymbol{\ell},r}^U$ and $\psi_{\boldsymbol{\ell},r}^L$ as in~\eqref{Eq:UpperLower} and~\eqref{Eq:UpperLower2}, we find that $L_2^2(\psi_{\boldsymbol{\ell},r}^U,\psi_{\boldsymbol{\ell},r}^L) \leq 4\epsilon^2$, where we have used the fact that 
\[
\sum_{i=2}^M \sum_{j=1}^{N_i} \alpha_{i,j}^2 = 2.
\]
Moreover, the logarithm of the cardinality of the bracketing set is
\begin{align*}
\sum_{i=2}^M \sum_{j=1}^{N_i} \log n_{i,j} &+ \log n_{M+1} \leq K_3^{**}\sum_{i=2}^M \sum_{j=1}^{N_i} \Bigl(\frac{\mu_3^{1/2}(S_{i,j})}{\alpha_{i,j}\epsilon}\Bigr)^{3/2} + 64K_3^{**}C_3c_0^{-1}\Bigl(\frac{\mu_3^{1/2}(P_M)}{\epsilon}\Bigr)^{3/2} \\
&\leq \frac{K_3^{**}}{\epsilon^{3/2}} \sum_{i=2}^M \biggl(\frac{\sum_{k=2}^M 2^{-k/2}}{2^{-(i-2)/2}}\biggr)^{3/4}N_i \mu_3^{3/4}(P_{i-1} \setminus \tensor*[_{c_0 4^i \epsilon^2}]{{P_{i-1}}}{}) + \frac{256K_3^{**}C_3c_0^{-1}}{\epsilon^{3/2}} \\
&\leq \frac{4K_3^{**}C_3c_0^{-1}}{\epsilon^2}\sum_{i=2}^M 2^{-i/8} + \frac{256K_3^{**}C_3c_0^{-1}}{\epsilon^{3/2}} \leq \frac{512K_3^{**}C_3c_0^{-1}}{\epsilon^2}
\end{align*}
Defining $K_{1,3}^\circ := 512K_3^{**}C_3c_0^{-1}$, we have therefore proved that when $d=3$,
\[
\log N_{[]}\bigl(2\epsilon,\bar{\Phi}_1(D'),L_2\bigr) \leq K_{1,3}^\circ \epsilon^{-2}
\]
for all $\epsilon \in (0,\epsilon_3^\circ]$.

For the final steps, we deal with the cases $d=1,2,3$ simultaneously.  Let
\[
\tilde{h}_d(\epsilon) := \left\{ \begin{array}{ll} \epsilon^{-1/2} & \mbox{when $d=1$} \\
\epsilon^{-1}\log_{++}^{3/2}(\frac{1}{4\epsilon}) & \mbox{when $d=2$} \\
\epsilon^{-2} & \mbox{when $d=3$.} \end{array} \right.
\]
(Thus $\tilde{h}_d$ is defined in almost the same way as $h_d$ from the proof of Theorem~\ref{Thm:BracketingBounds}, except for the 4 inside the logarithm when $d=2$.)  Set $K_{2,d}^\circ := K_{1,d}^\circ\tilde{h}_d(\epsilon_d^\circ)/\tilde{h}_d\bigl(\mu_d^{1/2}(D')\bigr)$.  Then, for $\epsilon \in (\epsilon_d^\circ,\mu_d^{1/2}(D')]$, we have
\begin{align*}
\log N_{[]}\bigl(2\epsilon,\bar{\Phi}_1(D'),L_2\bigr) \leq \log N_{[]}\bigl(2\epsilon_d^\circ,\bar{\Phi}_1(D'),L_2\bigr) \leq K_{1,d}^\circ\tilde{h}_d(\epsilon_d^\circ) &= K_{2,d}^\circ\tilde{h}_d\bigl(\mu_d^{1/2}(D')\bigr) \\
&\leq K_{2,d}^\circ \tilde{h}_d(\epsilon).
\end{align*}
On the other hand, for $\epsilon > \mu_d^{1/2}(D')$, it suffices to consider a single bracketing pair consisting of the constant functions $\psi^U(x) := 1$ and $\psi^L(x) := -1$ for $x \in D'$.  Note that $L_2^2(\psi^U,\psi^L) = 4\mu_d(D')$, so that $\log N_{[]}\bigl(2\epsilon,\Phi_B(D'),L_2\bigr) = 0$ for $\epsilon > \mu_d^{1/2}(D')$.  We conclude that when $D'$ is a $d$-dimensional closed, convex subset of $\mathbb{R}^d$ with $d^{-1}\bar{B}_d(0,1) \subseteq D' \subseteq \bar{B}_d(0,1)$, 
\[
\log N_{[]}\bigl(2\epsilon,\bar{\Phi}_1(D'),L_2\bigr) \leq K_{2,d}^\circ \tilde{h}_d(\epsilon)
\]
for all $\epsilon > 0$.

Finally, we show how to transform the brackets to the original domain $D$ and rescale their ranges to $[-B,B]$.  Recall that $D' = AD + b$.  Simplifying our notation from before, given $\epsilon > 0$, we have shown that we can define a bracketing set $\{[\psi_j^L,\psi_j^U]:j=1,\ldots,N\}$ for $\bar{\Phi}_1(D')$ with $L_2^2(\psi_j^U,\psi_j^L) \leq 4\epsilon^2|\det A|/B^2$ and $\log N \leq K_{2,d}^\circ \tilde{h}_d(\epsilon|\det A|^{1/2}/B)$.  We now define transformed brackets for $\bar{\Phi}_B(D)$ by
\[
\tilde{\psi}_j^U(z) := B\psi_j^U(Az+b) \quad \text{and} \quad \tilde{\psi}_j^L(z) := B\psi_j^L(Az+b).
\]
Then 
\begin{align*}
L_2^2(\tilde{\psi}_j^U,\tilde{\psi}_j^L) &= B^2\int_{D} \{\psi_j^U(Az+b) - \psi_j^L(Az+b)\}^2 \, d\mu_d(z) \\
&= \frac{B^2}{|\det A|} L_2^2(\psi_j^U,\psi_j^L) \leq 4\epsilon^2.
\end{align*}  
Now
\[
|\det A| = \frac{\mu_d(AD+b)}{\mu_d(D)} \geq \frac{\mu_d(d^{-1}\bar{B}_d(0,1))}{\mu_d(D)} = \frac{d^{-d}\pi^{d/2}}{\Gamma(1+d/2)\mu_d(D)}.
\]
It is convenient for the case $d=2$ to note that
\[
\tilde{h}_2\biggl(\frac{\epsilon|\det A|^{1/2}}{B}\biggr) \leq \tilde{h}_2\biggl(\frac{\epsilon\pi^{1/2}}{2B\mu_2^{1/2}(D)}\biggr) \leq \frac{2}{\pi^{1/2}}h_2\biggl(\frac{\epsilon}{B\mu_2^{1/2}(D)}\biggr).
\]
The final result therefore follows, taking $K_1^\circ := K_{2,1}^\circ$, $K_2^\circ := \frac{2}{\pi^{1/2}}K_{2,2}^\circ$ and $K_3^\circ := \frac{81}{4\pi}K_{2,3}^\circ$.
\end{proof}

\subsubsection{Auxiliary results for the proof of Theorem~\ref{Thm:Main}}
\label{Sec:ProofMain}

\begin{lemma}
\label{Lemma:ThreeTerms}
There exists $\eta_d \in (0,1)$ such that
\[
\sup_{g_0 \in \mathcal{F}_d^{0,I}} \mathbb{P}_{g_0}(\hat{g}_n \notin \tilde{\mathcal{F}}_d^{1,\eta_d}) = O(n^{-1})
\]
as $n \rightarrow \infty$, where $\hat{g}_n$ denotes the log-concave maximum likelihood estimator based on a random sample $Z_1,\ldots,Z_n$ from $g_0$.
\end{lemma}
\begin{proof}
For $g \in \mathcal{F}_d$, we write $\mu_g := \int_{\mathbb{R}^d} z g(z) \, dz$ and $\Sigma_g := \int_{\mathbb{R}^d} (z - \mu_g)(z - \mu_g)^T \, g(z) \, dz$.  Note that for $n \geq d+1$, and for any $\eta_d \in (0,1)$,
\begin{align}
\label{Eq:ThreeTerms}
\sup_{g_0 \in \mathcal{F}_d^{0,I}} \mathbb{P}_{g_0}(\hat{g}_n \notin \tilde{\mathcal{F}}_d^{1,\eta_d}) &\leq \sup_{g_0 \in \mathcal{F}_d^{0,I}} \mathbb{P}_{g_0}(\|\mu_{\hat{g}_n}\| > 1) + \sup_{g_0 \in \mathcal{F}_d^{0,I}} \mathbb{P}_{g_0}\{\lambda_{\mathrm{max}}(\Sigma_{\hat{g}_n}) > 1 + \eta_d\} \notag \\
&\hspace{4cm}+ \sup_{g_0 \in \mathcal{F}_d^{0,I}} \mathbb{P}_{g_0}\{\lambda_{\mathrm{min}}(\Sigma_{\hat{g}_n}) < 1 - \eta_d\}.
\end{align}
We treat the three terms on the right-hand side of~(\ref{Eq:ThreeTerms}) in turn.   First, we observe by Remark~2.3 of \citet{DSS2011} that $\mu_{\hat{g}_n} = n^{-1}\sum_{i=1}^n Z_i =: \bar{Z}$, where the density of $n^{1/2}\bar{Z} := n^{1/2}(\bar{Z}_1,\ldots,\bar{Z}_d)^T$ belongs to $\mathcal{F}_d^{0,I}$.  Taking $A_{0,d}, B_{0,d} > 0$ from Theorem~\ref{Thm:IntEnv}(a), it follows that for any $t \geq 0$ and $j=1,\ldots,d$, 
\[
\sup_{g_0 \in \mathcal{F}_d^{0,I}} \mathbb{P}_{g_0}(n^{1/2}|\bar{Z}_j| > t) \leq 2 \int_t^\infty e^{-A_{0,d} x + B_{0,d}} \, dx = \frac{2}{A_{0,d}}e^{-A_{0,d}t+B_{0,d}}.
\]
Hence
\[
\sup_{g_0 \in \mathcal{F}_d^{0,I}} \mathbb{P}_{g_0}(\|\mu_{\hat{g}_n}\| > 1) \leq \sup_{g_0 \in \mathcal{F}_d^{0,I}} \sum_{j=1}^d \mathbb{P}_{g_0}\biggl(n^{1/2}|\bar{Z}_j| > \frac{n^{1/2}}{d^{1/2}}\biggr) \leq \frac{2d}{A_{0,d}}e^{-\frac{A_{0,d} n^{1/2}}{d^{1/2}} + B_{0,d}} = O(n^{-1}).
\]
For the second term, we use Remark~2.3 of \citet{DSS2011} again to see that $\lambda_{\max}(\Sigma_{\hat{g}_n}) \leq \lambda_{\max}(\tilde{\Sigma}_n)$, where $\tilde{\Sigma}_n := n^{-1}\sum_{i=1}^n (Z_i - \bar{Z})(Z_i - \bar{Z})^T = n^{-1}\sum_{i=1}^n Z_i Z_i^T - \bar{Z}\bar{Z}^T$ denotes the sample covariance matrix.  For each $j = 1,\ldots,d$, 
\[
\sup_{g_0 \in \mathcal{F}_d^{0,I}} \int_{\mathbb{R}^d} z_j^4 g_0(z) \, dz \leq 2\int_0^\infty z_j^4 e^{-A_{0,1}z_j + B_{0,1}} \, dz_j = \frac{48e^{B_{0,1}}}{A_{0,1}^5}.
\]
Writing $Z_i := (Z_{i1},\ldots,Z_{id})^T$, we deduce from the Gerschgorin circle theorem, Chebychev's inequality and Cauchy--Schwarz that
\begin{align*}
\sup_{g_0 \in \mathcal{F}_d^{0,I}} &\mathbb{P}_{g_0}\{\lambda_{\mathrm{max}}(\Sigma_{\hat{g}_n}) > 1 + \eta_d\} \leq \sup_{g_0 \in \mathcal{F}_d^{0,I}} \mathbb{P}_{g_0}\{\lambda_{\mathrm{max}}(\tilde{\Sigma}_n) > 1 + \eta_d\} \\
&\leq \sup_{g_0 \in \mathcal{F}_d^{0,I}} \mathbb{P}_{g_0}\biggl(\bigcup_{j=1}^d \biggl\{\frac{1}{n}\sum_{i=1}^n Z_{ij}^2 - 1\biggr\} > \frac{\eta_d}{3}\biggr) + \sup_{g_0 \in \mathcal{F}_d^{0,I}} \mathbb{P}_{g_0}\biggl(\bigcup_{1 \leq j < k \leq d} \biggl|\frac{1}{n}\sum_{i=1}^n Z_{ij}Z_{ik}\biggr| > \frac{\eta_d}{3d}\biggr) \\
&\hspace{9cm}+ \sup_{g_0 \in \mathcal{F}_d^{0,I}} \mathbb{P}_{g_0}\biggl(\|\bar{Z}\|^2 > \frac{\eta_d}{3}\biggr) \\
&\leq \frac{9d}{\eta_d^2n}\times \frac{48e^{B_{0,1}}}{A_{0,1}^5} + \frac{9d^2}{\eta_d^2n}\times \frac{24d(d-1)e^{B_{0,1}}}{A_{0,1}^5} + \frac{2d}{A_{0,d}}e^{-\frac{A_{0,d} \eta_d^{1/2} n^{1/2}}{3^{1/2}d^{1/2}} + B_{0,d}} = O(n^{-1}).
\end{align*}
The third term on the right-hand side of~(\ref{Eq:ThreeTerms}) is the most challenging to handle.  Let $\mathcal{P}^{1/10,1/2}$ denote the class of probability distributions $P$ on $\mathbb{R}^d$ such that $\mu_P := \int_{\mathbb{R}^d} x \, dP(x)$ and $\Sigma_P := \int_{\mathbb{R}^d} (x-\mu_P)(x-\mu_P)^T \, dP(x)$ satisfy $\|\mu_P\| \leq 1/10$ and $1/2 \leq \lambda_{\min}(\Sigma_P) \leq \lambda_{\max}(\Sigma_P) \leq 3/2$, and such that 
\[
\int_{\mathbb{R}^d} \|x\|^4 \, dP(x) \leq \frac{2d\pi^{d/2}\Gamma(d+4)}{\Gamma(1+d/2)}\frac{e^{B_0,d}}{A_{0,d}^{d+4}} =: \tau_{4,d},
\]
say, where $A_{0,d}$ and $B_{0,d}$ are taken from Theorem~\ref{Thm:IntEnv}(a).  Observe that by Theorem~\ref{Thm:IntEnv}(a),
\[
\sup_{g_0 \in \mathcal{F}_d^{0,I}} \int_{\mathbb{R}^d} \|x\|^4 g_0(x) \, dx \leq \int_{\mathbb{R}^d} \|x\|^4 e^{-A_{0,d}\|x\| + B_{0,d}} \, dx = \frac{d\pi^{d/2}e^{B_{0,d}}}{\Gamma(1+d/2)}\int_0^\infty r^{d+3} e^{-A_{0,d} r} \, dr = \frac{\tau_{4,d}}{2}.
\]
Recall from Theorem~2.2 of \citet{DSS2011} that for $P \in \mathcal{P}^{1/10,1/2}$, there exists a unique log-concave projection $\psi^*(P) \in \mathcal{F}_d$ given by
\[
\psi^*(P) := \argmax_{f \in \mathcal{F}_d} \int_{\mathbb{R}^d} \log f \, dP.
\]
Our first claim is that there exists $M_{0,d} > 0$, depending only on $d$, such that
\[
\sup_{P \in \mathcal{P}^{1/10,1/2}} \sup_{x \in \mathbb{R}^d} \log \psi^*(P)(x) \leq M_{0,d}.
\]
To see this, suppose for a contradiction that there exist $(P_n) \in \mathcal{P}^{1/10,1/2}$ such that 
\[
\sup_{x \in \mathbb{R}^d} \log \psi^*(P_n)(x) \rightarrow \infty.
\]
Similar to the proof of Theorem~\ref{Thm:IntEnv}(a), the sequence $(P_n)$ is tight, so there exists a subsequence $(P_{n_k})$ and a probability measure $P$ on $\mathbb{R}^d$ such that $P_{n_k} \stackrel{d}{\rightarrow} P$.  If $(Y_{n_k})$ is a sequence of random vectors on the same probability space with $Y_{n_k} \sim P_{n_k}$, then $\{\|Y_{n_k}\|:k \in \mathbb{N}\}$ is uniformly integrable, because $\mathbb{E}(\|Y_{n_k}\|^2) \leq 3d/2 + 1/100$.  We deduce that $\int_{\mathbb{R}^d} \|x\| \, dP_{n_k}(x) \rightarrow \int_{\mathbb{R}^d} \|x\| \, dP(x)$.  Together with the weak convergence, this means that $P_{n_k}$ converges to $P$ in the Wasserstein distance.  Moreover, for any unit vector $u \in \mathbb{R}^d$, the family $\{(u^TY_{n_k})^2:k \in \mathbb{N}\}$ is uniformly integrable, because $\mathbb{E}\{(u^TY_{n_k})^4\} \leq \mathbb{E}(\|Y_{n_k}\|^4) \leq \tau_{4,d}$.  Thus $u^T\Sigma_P u = \lim_{k \rightarrow \infty} u^T \Sigma_{P_{n_k}}u \geq 1/2$, so in particular, $P(H) < 1$ for every hyperplane $H$ in $\mathbb{R}^d$.  We conclude by Theorem~2.15 and Remark~2.16 of \citet{DSS2011} that $\psi^*(P_{n_k})$ converges to $\psi^*(P)$ uniformly on closed subsets of $\mathbb{R}^d \setminus \mathrm{disc}(\psi^*(P))$, where $\mathrm{disc}(\psi^*(P))$ denotes the set of discontinuity points of $\psi^*(P)$.  In turn, this implies that
\[
\sup_{x \in \mathbb{R}^d} \psi^*(P_{n_k})(x) \leq \sup_{x \in \mathbb{R}^d} \psi^*(P)(x) + 1
\]
for sufficiently large $k$, which establishes our desired contradiction.

Moreover, by Theorem~\ref{Thm:IntEnv}(b), there exists $a_{0,d} > 0$, depending only on $d$, such that
\[
\inf_{f \in \mathcal{F}_d^{0,I}} f(0) \geq a_{0,d}.
\]
It follows that for any $\mu \in \mathbb{R}^d$,
\[
\inf_{f \in \mathcal{F}_d^{\mu,\Sigma}} \sup_{x \in \mathbb{R}^d} f(x) \geq a_{0,d} (\det \Sigma)^{-1/2}.
\]
Thus, using our claim, if $\det \Sigma < a_{0,d}^2 e^{-2M_{0,d}}$, then $\{\psi^*(P):P \in \mathcal{P}^{1/10,1/2}\} \cap (\cup_{\mu \in \mathbb{R}^d} \mathcal{F}_d^{\mu,\Sigma}) = \emptyset$.  Since $\sup_{P \in \mathcal{P}^{1/10,1/2}} \lambda_{\max}(\Sigma_P) \leq 3/2$, we deduce that if $\lambda_{\min}(\Sigma) < 2^{d-1}a_{0,d}^2 e^{-2M_{0,d}}/3^{d-1}$, then $\{\psi^*(P):P \in \mathcal{P}^{1/10,1/2}\} \cap (\cup_{\mu \in \mathbb{R}^d} \mathcal{F}_d^{\mu,\Sigma}) = \emptyset$.

Finally, we conclude that if we define $\eta_d := 1 - \frac{2^{d-2}a_{0,d}^2 e^{-2M_{0,d}}}{3^{d-1}}$, then
\begin{align*}
&\sup_{g_0 \in \mathcal{F}_d^{0,I}} \mathbb{P}_{g_0}\{\lambda_{\mathrm{min}}(\Sigma_{\hat{g}_n}) < 1-\eta_d\} \leq \sup_{g_0 \in \mathcal{F}_d^{0,I}} \mathbb{P}_{g_0}\{\lambda_{\mathrm{min}}(\tilde{\Sigma}_n) < 1/2\} + \sup_{g_0 \in \mathcal{F}_d^{0,I}} \mathbb{P}_{g_0}\{\lambda_{\mathrm{max}}(\tilde{\Sigma}_n) > 3/2\} \\
&+ \sup_{g_0 \in \mathcal{F}_d^{0,I}} \mathbb{P}_{g_0}(\|\bar{Z}\| > 1/10) + \sup_{g_0 \in \mathcal{F}_d^{0,I}} \mathbb{P}_{g_0}\biggl(\biggl|\frac{1}{n}\sum_{i=1}^n \bigl\{\|Z_i\|^4 - \mathbb{E}(\|Z_1\|^4)\bigr\}\biggr| > \frac{\tau_{4,d}}{2}\biggr) = O(n^{-1}),
\end{align*}
using very similar arguments to those used above, as well as Chebychev's inequality for the last term.
\end{proof}

\begin{thm}[\citet{vandeGeer2000}, Theorem~7.4]
\label{Thm:vdG}
Let $\mathcal{F}$ denote a class of (Lebesgue) densities on $\mathbb{R}^d$, let $X_1,X_2,\ldots$ be independent and identically distributed with density $f_0 \in \mathcal{F}$, and let $\hat{f}_n$ denote a maximum likelihood estimator of $f_0$ based on $X_1,\ldots,X_n$.  Write $\bar{\mathcal{F}} := \bigl\{\bigl(\frac{f+f_0}{2}\bigr):f \in \mathcal{F}\bigr\}$, and let
\[
J_{[]}(\delta,\bar{\mathcal{F}},h) := \max\biggl\{\int_{\delta^2/2^{13}}^\delta \sqrt{\log N_{[]}(u,\bar{\mathcal{F}},h)} \, du \, , \, \delta\biggr\}.
\]
If $(\delta_n)$ is such that $2^{-16}n^{1/2}\delta_n^2 \geq J_{[]}(\delta_n,\bar{\mathcal{F}},h)$, then for all $t \geq \delta_n$,
\[
\mathbb{P}_{f_0}\{h(\hat{f}_n,f_0) \geq 2^{1/2}t\} \leq 2^{13/2}\sum_{s=0}^\infty \exp\biggl(-\frac{2^{2s}nt^2}{2^{27}}\biggr).
\]
\end{thm}

\end{document}